\documentclass{amsart}

\usepackage{amssymb,esint,stmaryrd}

\usepackage[pdfborder={0 0 0},draft=false]{hyperref}
\newcommand\bartoncrossref[2]{\texorpdfstring{\hyperref[#2]{#1~\ref*{#2}}}{#1~\ref*{#2}}}

\numberwithin{equation}{section}
\newtheorem{thm}[equation]{Theorem}
\newtheorem{lem}[equation]{Lemma}
\newtheorem{cor}[equation]{Corollary}

\theoremstyle{definition}
\newtheorem{dfn}[equation]{Definition}
\theoremstyle{remark}
\newtheorem{rmk}[equation]{Remark}

\newcommand\thmref[1]{\bartoncrossref{Theo\-rem}{thm:#1}}
\newcommand\lemref[1]{\bartoncrossref{Lem\-ma}{lem:#1}}
\newcommand\crlref[1]{\bartoncrossref{Corollary}{cor:#1}}

\newcommand\dfnref[1]{\bartoncrossref{Definition}{dfn:#1}}
\newcommand\rmkref[1]{\bartoncrossref{Remark}{rmk:#1}}
\newcommand\secref[1]{\bartoncrossref{Section}{sec:#1}}

\newcommand\dist{\mathop{\mathrm{dist}}}
\newcommand\Div{\mathop{\mathrm{div}}}
\newcommand\supp{\mathop{\mathrm{supp}}}
\newcommand\re{\mathop{\mathrm{Re}}}
\newcommand\im{\mathop{\mathrm{Im}}}
\newcommand\RR{\mathbb{R}}
\newcommand\CC{\mathbb{C}}
\newcommand\e{\vec{e}}
\newcommand\D{\mathcal{D}}
\newcommand\s{\mathcal{S}}
\newcommand\E{\mathcal{E}}
\newcommand\F{\mathcal{F}}
\newcommand\1{\mathbf{1}}

\newcommand\abs[1]{\lvert{#1}\rvert}

\newcommand\biggabs[1]{\biggl\lvert{#1}\biggr\rvert}
\newcommand\doublebar[1]{\lVert{#1}\rVert}
\newcommand\triplebar[1]{\mathopen\interleave{#1}\mathclose\interleave}

\makeatletter 
    \def\HyPsd@CatcodeWarning#1{} 
\makeatother

\begin{document}

\title[The Dirichlet problem for higher order equations]
{The Dirichlet problem for higher order equations in composition form}
  
\author{Ariel Barton}
\address{Ariel Barton, School of Mathematics, University of Minnesota, 127 Vincent Hall, 206 Church St.\ SE, Minneapolis, Minnesota 55455}
\email{abarton@math.umn.edu}
\author{Svitlana Mayboroda}
\address{Svitlana Mayboroda, School of Mathematics, University of Minnesota, 127 Vincent Hall, 206 Church St.\ SE, Minneapolis, Minnesota 55455}
\email{svitlana@math.umn.edu}
\thanks{Svitlana Mayboroda is partially supported by the Alfred P. Sloan Fellowship, the NSF CAREER Award DMS 1056004,  and the NSF Materials Research Science and Engineering Center Seed Grant.}




\begin{abstract}
The present paper commences the study of higher order differential equations in composition form. Specifically, we consider the equation $Lu=\Div B^*\nabla(a\,\Div A\nabla u)=0$, where $A$ and $B$ are elliptic matrices with complex-valued bounded measurable coefficients and $a$ is an accretive function. Elliptic operators of this type naturally arise, for instance, via a pull-back of the bilaplacian $\Delta^2$ from a Lipschitz domain to the upper half-space. More generally, this form is preserved under a Lipschitz change of variables, contrary to the case of divergence-form fourth order differential equations. We establish well-posedness of the Dirichlet problem for the equation $Lu=0$, with boundary data in~$L^2$, and with optimal estimates in terms of nontangential maximal functions and square functions. 
\end{abstract}

\maketitle


\section{Introduction}

The last few decades have witnessed a surge of activity on boundary-value problems on Lipschitz domains for second-order divergence-form elliptic equations $-\Div A\nabla u=0$. Their investigation has, in particular, been guided by two principles. 

First, divergence-form equations are naturally associated to a bilinear ``energy'' form, and admit a variational formulation. It turns out that some smoothness of the coefficients $A$ in a selected direction is necessary for well-posedness of the underlying boundary problems in $\RR^{n+1}_+=\{(x,t):\,x\in \RR^n, \,t>0\}$ (see \cite{CafFK81}). 
This observation led to the study of the coefficients \emph{constant along a single coordinate} (the $t$-coordinate when $n\geq 2$). The well-posedness of the corresponding boundary-value problems was established for real symmetric matrices in \cite{JerK81a, KP93}, and the real non-symmetric case was recently treated in \cite{KKPT00, KR09, Rul07, HofKMP12}. In addition, the resolution of the Kato problem \cite{AusHLMT02} provided  well-posedness for complex $t$-independent matrices in a block form; see \cite{AusAH08,May10a}. Furthermore, a number of perturbation-type results have been obtained, pertaining to the coefficient matrices close to ``good'' ones  in the sense of the $L^\infty$ norm  \cite{FabJK84, AusAH08, AusAM10, AlfAAHK11, B12}, or in the sense of Carleson measures  \cite{Dah86a, Fef89, FefKP91, Fef93, KP93, KP95, DinPP07, DinR10, AusA11, AusR11, HofMM}.

A seemingly different point of view emerges from the ultimate goal of treating boundary-value problems on non-smooth domains rather than just the upper half-space. However, it brings to focus equations of the same type as above. Indeed, the direct pull-back of the Laplacian $\Delta$ from a Lipschitz domain $\{(x,t): t>\varphi(x)\}$ to $\RR^{n+1}_+$ yields a boundary-value problem for an operator of the form $-\Div A\nabla u=0$ on $\RR^{n+1}_+$, and the corresponding matrix $A$ is, once again, independent of $t$. More generally, if $\rho$ is a change of variables, then there is a real symmetric matrix $A=A(x)$ such that if $\Delta u=0$ in~$\Omega$ and $\tilde u=u\circ \rho$, then $\Div A\nabla \tilde u=0$ in $\rho^{-1}(\Omega)$.

The model higher order differential operator is the bilaplacian $\Delta^2=-\Delta (-\Delta)$. Investigating the behavior of biharmonic functions under changes of variables, we find that there exists a scalar-valued function $a$ and a real elliptic matrix $A$ such that, if $\Delta^2 u=0$ in~$\Omega$, then
\begin{equation}\label{comp1} \Div A\nabla(a\,\Div A\nabla \tilde u)=0\quad\text{in }\rho^{-1}(\Omega),\end{equation}
More generally, such a form  is preserved under changes of variables. We emphasize that this is not the case for higher order operators in divergence form  $(-1)^m\sum_{\abs{\beta}=\abs{\gamma}=m} \partial^\beta A_{\beta\gamma} \partial^\gamma$. 

In addition,  the  form appearing in \eqref{comp1} mimics the structure of the bilaplacian as a composition of two Laplace operators. As it turns out, this is an important feature that underpins several key properties of the solutions to the biharmonic and  polyharmonic equations $(-\Delta)^m u=0$, $m\geq 1$.

Motivated by these considerations, the present paper commences the study of well-posedness problems for higher order equations in composition form. Specifically, consider the equation 
\begin{equation}\label{eqn:introprodform}
\Div B^*\nabla(a\,\Div A\nabla u) =0.
\end{equation}
Here $a:\RR^{n+1}\mapsto\CC$ is a scalar-valued accretive function and $A$ and $B$ are ${(n+\nobreak1)}\allowbreak\times{(n+\nobreak1)}$ elliptic matrices with complex coefficients. That is,  there exist constants $\Lambda>\lambda>0$ such that, if $M=A$ or $M=B$, then
\begin{equation}
\label{eqn:ellipticaccretive}
\lambda\leq \re a(X)\leq \abs{a(X)}\leq \Lambda
,\quad
\lambda \abs{\eta}^2\leq \re\bar\eta\cdot M(X)\eta,
\quad
\abs{\xi\cdot M(X)\eta}\leq \Lambda\abs{\eta}\abs{\xi}
\end{equation}
for all $X\in\RR^{n+1}$ and all vectors $\xi$, $\eta\in \CC^{n+1}$. The second-order operators $\Div A\nabla $ and $\Div B^*\nabla $ are meant in the weak sense; see \dfnref{weaksoln} below for a precise definition. We assume that the coefficients $a$, $A$ and~$B$ are $t$-independent; no additional regularity assumptions are imposed.
For technical reasons we also require that $A$ and $B$ satisfy the De Giorgi-Nash-Moser condition, that is, that solutions $u$ to $\Div M\nabla u=0$ are locally H\"older continuous for $M=A$, $B$, $A^*$ and~$B^*$.

The main result of this paper is as follows. 
We show that whenever the \emph{second} order \emph{regularity} boundary-value problems for $A$ and $B$ are well-posed, and whenever the operator $L=\Div B^*\nabla\,a\,\Div A\nabla$ is close to being self-adjoint,  
the $L^2$-Dirichlet problem
\begin{equation}
\label{eqn:4dirichlet}\left\{
\begin{aligned}
\Div B^*\nabla(a\,\Div A\nabla u)&=0 &&\text{in }\RR^{n+1}_+,\\
u&=f &&\text{on }\partial\RR^{n+1}_+,\quad \nabla f\in L^2(\RR^n),\\
\e_{n+1}\cdot A\nabla u&=g &&\text{on }\partial\RR^{n+1}_+, \quad g\in L^2(\RR^n),
\end{aligned}\right.\end{equation}
has a unique solution $u$ that satisfies the optimal estimates
\begin{equation*}\widetilde N_+ (\nabla u) \in L^2(\RR^n)\quad\text{and}\quad\triplebar{t\,\Div A\nabla u}_+<\infty,\end{equation*}
where
\begin{equation*}\widetilde N_+(\nabla u)(x) = \sup\biggl\{\biggl(\fint_{B((y,s),s/2)} \abs{\nabla u}^2\biggr)^{1/2}:
y\in\RR^n,\>\abs{x-y}<s\biggr\}\end{equation*}
and where
\begin{equation*}\triplebar{t\,F}_+=\int_{\RR^{n}}\int_0^\infty \abs{F(x,t)}^2 \,t\,dt\,dx.\end{equation*}
Specifically, we will construct solutions whenever $\doublebar{\im a}_{L^\infty(\RR^n)}$ and $\doublebar{A-B}_{L^\infty(\RR^n)}$ are sufficiently small. It is assumed that the second-order regularity problem
\[\Div A\nabla u=0 \text{ in }\RR^{n+1}_\pm,
\quad u=f \text{ on }\partial\RR^{n+1}_\pm,
\quad \widetilde N_\pm(\nabla u)\in L^2(\RR^n)\]
has a unique solution in both the upper and lower half-spaces whenever $\nabla f\in L^2(\RR^n)$, and that the same is true for $A^*$. By perturbation results in \cite{AusAM10}, an analogous statement is then automatically valid for  $B$ and~$B^*$.

We will construct solutions using layer potentials; the De Giorgi-Nash-Moser requirement mentioned above is necessary for this approach and  at the moment is common in the theory of second-order problems. 
If, for example, $A$ and $B$ are real symmetric or complex and constant, then the  De Giorgi-Nash-Moser condition is valid and regularity problems are well-posed; hence, in these cases there is well-posedness of \eqref{eqn:4dirichlet}. We mention that in passing we prove that if  the second-order regularity problems are well-posed then the solutions necessarily can be written as layer potentials.\footnote{This result is tantamount to proving invertibility of the single layer potential. The method of establishing injectivity from the regularity problem and from jump relations is known; the authors would like to thank Carlos Kenig for bringing this argument to their attention. The method of establishing surjectivity is new and again uses jump relations.} This fact  is new and interesting on its own right.   We will precisely state the main theorems in \secref{main}, after the notation of this paper has been established.

Let us point out that to the best of the authors' knowledge, this is the first result regarding well-posedness of higher order boundary-value problems with non-smooth variable coefficients and with boundary data in~$L^p$. For divergence form equations, some results for boundary data in Besov and Sobolev spaces $L^p_\alpha$, $0<\alpha<1$, are available (see \cite{Agr07,MazMS10}), but they do not reach out to the ``end-point'' case of $L^p$ data. Until now, well-posedness in Lipschitz domains  with $L^p$ data was known only for constant coefficient higher order operators (see \cite{DahKV86, PipV92, PipV95B, Ver96, She06A, She06B, She06C}). As explained above, our results extend to Lipschitz domains automatically via a change of variables. (See \thmref{lipschitz}  for the precise statement.)


Let us discuss the history of the subject and our methods in more detail. We will concentrate on higher order operators and only mention the second order results that directly affect our methods. 
The basic boundary-value problem for elliptic differential equations of order greater than~2 is the $L^p$-Dirichlet problem for the biharmonic operator~$\Delta^2$. It is said to be well-posed in a domain $\Omega$ if, for every $f\in W^p_1(\partial\Omega)$ and $g\in L^p(\partial\Omega)$, there exists a unique function $u$ that satisfies
\begin{equation}
\label{eqn:biharmonicDirichlet}
\Delta^2 u=0 \text{ in } \Omega,\quad
u=f \text{ on } \partial\Omega,\quad
\nu\cdot\nabla u=g \text{ on } \partial\Omega,
\quad N_\Omega (\nabla u)\in L^p(\partial\Omega)
\end{equation}
where $\nu$ is the outward unit normal derivative and
\begin{equation}\label{eqn:Nlipschitz}
N_\Omega (\nabla u)(X)=\sup\{\abs{\nabla u(Y)}:Y\in\Omega,\>\abs{X-Y}<(1+a)\dist(Y,\partial\Omega)\}
\end{equation}
for some constant $a>0$. In \cite{SelS81,CohG83,Ver87}, well-posedness of the $L^p$-Dirichlet problem for $\Delta^2$ was established in $C^1$ domains for any $1<p<\infty$. (This result is also valid in convex domains; see \cite{She06C,KilS11A}.) For general Lipschitz domains, 
the $L^2$-Dirichlet problem for $\Delta^2$ was shown to be well-posed by Dahlberg, Kenig and Verchota in \cite{DahKV86} (when the domain is bounded; cf.\ \cite[Theorem~3.7]{PipV92} for domains above Lipschitz graphs). 

The sharp range of $p$ for which the $L^p$-Dirichlet problem is well-posed in $n$-dimensional Lipschitz domains is a difficult problem, still open in higher dimensions even for the bilaplacian (cf.\ \cite{CalFS79,K94}). We do not tackle the well-posedness in $L^p$, $p\neq 2$, in the present paper; it is a subject for future investigation. However, let us mention in passing that for the bilaplacian the sharp results are only known in dimensions less than or equal to~7 \cite{She06A,She06B,She06C} and, in a dramatic contrast with the case of the second order boundary-value problems, there is a sharp dimension-dependent upper bound on the range of well-posedness. That is, if $\Omega\subset\RR^n$ is a Lipschitz domain, then solutions to \eqref{eqn:biharmonicDirichlet} are guaranteed to exist only for $2\leq p\leq p_n$ for some $p_n<\infty$. Related counterexamples can be found in \cite[Theorem~10.7]{PipV92}. See also \cite{BM12} for a review of this and related matters. 

Our methods in the present paper depart from the ideas in \cite{DahKV86} and \cite{PipV92}. 
The solution is represented via 
\begin{equation*}
u(X)=-\D_A f(X) -\s_A g(X) + \E_{B,a,A} h(X).
\end{equation*}
Here $\D_A$ and $\s_A$ are the classic double and single layer potentials associated to the operator $\Div A\nabla$, given by the formulas
\begin{align*}
\D_A f(x,t) &= -\int_{\RR^n} \e_{n+1}\cdot \overline{A^*(y) \nabla\Gamma_{(x,t)}^{A^*}(y,0)}\, f(y) \,dy,\\
\s_A g(x,t) &= \int_{\RR^n} \overline{\Gamma_{(x,t)}^{A^*}(y,0)}\,g(y) \, dy
\end{align*}
where $\Gamma_X^{A^*}$ is the fundamental solution to $-\Div A^* \nabla$ (that is, the solution to  $-\Div A^* \nabla\Gamma_X^{A^*} =\delta_X$). On the other hand, 
$\E_{B,a,A}$ is a new layer potential, specifically built for the problem at hand, to satisfy
\begin{equation*}-a\,\Div A\nabla\E_{B,a,A} h = \1_{\RR^{n+1}_+} \partial_{n+1}^2 \s_{B^*} h.\end{equation*}
See \secref{fourthorderpotential}. This resembles the formula used in \cite{PipV92} to construct solutions to $\Delta^2 u=0$. 
To prove existence of solutions to \eqref{eqn:4dirichlet} or \eqref{eqn:biharmonicDirichlet}, in  addition to the second order results, we require appropriate nontangential maximal function (and square-function) estimates for the new potential $\E_{B,a,A}$, as well as the invertibility of $h\mapsto\partial_{n+1} \E_{B,a,A} h\vert_{\partial\RR^{n+1}_+}$ in~$L^2$. 
However, beyond the representation formula and invertibility argument, our method is necessarily different from  \cite{DahKV86} and \cite{PipV92}.

After a certain integration by parts, the bounds on the nontangential maximal function of the new potential $\E h$ in the case of the \emph{bilaplacian} become  an automatic consequence of the Calder\'on-Zygmund theory and boundedness of the Cauchy integral in $L^2$. On the other hand, for a general composition operator, the related singular integral operators do not fall under the scope of the Calder\'on-Zygmund theory and, because of the presence of non-smooth matrices $A$ and $B$, are not amenable to a similar integration by parts. We develop an alternative argument, appealing to some elements of the method in \cite{AlfAAHK11}, to obtain square function bounds, and then employ the jump relations and intricate interplay between solutions in the upper and lower half-spaces to obtain the desired nontangential maximal function estimates.  

It is interesting to observe that, given an involved composition form of the operator, with several ``layers'' of non-smooth coefficients, the difficulties also manifest themselves in the absence of a classical variational formulation. In particular, such standard properties of solutions as the Caccioppoli inequality have to be reproven and even the existence of the Green function or fundamental solution in $\RR^{n+1}$ is not obvious, in any function space. In the same vein, the existence of the normal derivative of a solution cannot be viewed as the result of an integration by parts and an approximation scheme. Instead, it once again calls for some special properties of the associated higher order potentials.

Needless to say, our results build extensively on the developments from the theory of second-order divergence form operators $L_A = -\Div A\nabla$. We refer the reader to \cite{K94} for a detailed summary of the theory as it stood in the mid-1990s, and to the papers \cite{KKPT00, KP01, DinPP07, KR09, Rul07, AusAH08, AusAM10, May10a, DinR10, AlfAAHK11, AusA11, AusR11, HofKMP12, HofMM} for more recent developments.

Finally, let us mention that aside from the the Dirichlet case, it is natural to consider the Neumann and regularity problems as well as the inhomogeneous equation $Lu=f$, for $f$ in a suitable function space. Recent achievements in this direction for higher order equations include \cite{CohG85,Ver05,She07B}, \cite{Ver90,PipV92,KilS11A}, and \cite{AdoP98,MitMW11,Agr07,MazMS10} respectively. Unfortunately, much as in the homogeneous Dirichlet case, they concentrate mostly on constant coefficients, with the exception of \cite{Agr07,MazMS10}; these two papers consider the inhomogeneous problem $Lu=f$ but require that the boundary data have extra smoothness in~$L^p$.

The outline of this paper is as follows. In \secref{notation} we will define the notation used throughout this paper and state our main results. In \secref{secondorder} we will review known results from the theory of second-order operators of the form $L_A=-\Div A\nabla$. In \secref{Caccioppoli}, we will prove fourth-order analogues to some basic theorems concerning solutions to second-order equations, such as the Caccioppoli inequality. We will construct solutions to \eqref{eqn:4dirichlet} using potential operators and establish that these potentials are well-defined and bounded in Sections~\ref{sec:potentials}, \ref{sec:square} and~\ref{sec:nontangential}. The invertibility and uniqueness results will be presented in \secref{final} together with the end of proof of the main theorems.

\section{Notation and the main theorems}
\label{sec:notation}

In this section we define the notation used throughout this paper; in \secref{main} we will state our main theorems. (The proofs will be delayed until \secref{final}.)

We work in the upper half-space $\RR^{n+1}_+=\RR^n\times(0,\infty)$ and the lower half-space $\RR^{n+1}_-=\RR^n\times(-\infty,0)$.
We identify $\partial\RR^{n+1}_\pm$ with $\RR^n$. The coordinate vector $\e=\e_{n+1}$ is the inward unit normal to $\RR^{n+1}_+$ and the outward unit normal to $\RR^{n+1}_-$. We will reserve the letter $t$ to denote the $(n+1)$st coordinate in~$\RR^{n+1}$. 

If $\Omega$ is an open set (contained in $\RR^n$ or $\RR^{n+1}$), then  $W^2_1(\Omega)$ denotes the Sobolev space of functions $f\in L^2(\Omega)$ whose weak gradient $\nabla f$ also lies in $L^2(\Omega)$, and $W^2_{-1}(\Omega)$ denotes its dual space. The local Sobolev space $W^2_{1,loc}(\Omega)$ denotes the set of all functions $f$ that lie in $W^2_1(V)$ for all open sets $V$ compactly contained in~$\Omega$.
We let $\dot W^2_1(\RR^n)$ be the completion of  $C^\infty_0(\RR^n)$ under the norm $\doublebar{f}_{\dot W^2_1(\RR^n)}=\doublebar{\nabla f}_{L^2(\RR^n)}$; equivalently $\dot W^2_1(\RR^n)$ is the space of functions $f\in W^2_{1,loc}(\RR^n)$ for which $\doublebar{\nabla f}_{L^2(\RR^n)}$ is finite. Observe that functions in $\dot W^2_1(\RR^n)$ are only defined up to additive constants.

If $u\in W^2_{1,loc}(\Omega)$ for some $\Omega\subset\RR^{n+1}$, we let $\nabla_\parallel u$ denote the gradient of $u$ in the first $n$ variables, that is, $\nabla_\parallel u=(\partial_1 u,\partial_2 u,\dots,\partial_n u)$. We will occasionally use $\nabla_\parallel$ to denote the \emph{full} gradient of a function defined on~$\RR^n$.

As in \cite{AlfAAHK11,AusAH08} and other papers, we will let the triple-bar norm denote the $L^2$ norm with respect to the measure $(1/\abs{t})\,dx\,dt$. That is, we will write
\begin{equation}
\label{eqn:squarenormdfn}
\triplebar{F}_\pm^2 =\int_0^\infty \int_{\RR^n}\abs{F(x,\pm t)}^2\,dx\,\frac{1}{t}\,dt,
\quad
\triplebar{F}^2 = \triplebar{F}_+^2 + \triplebar{F}_-^2
\end{equation}
with the understanding that a $t$ inside a triple-bar norm denotes the $(n+1)$st coordinate, that is, 
\begin{equation*}
\triplebar{t\,F}_\pm^2 =\int_0^\infty \int_{\RR^n}\abs{F(x,\pm t)}^2\,dx\,t\,dt
.\end{equation*}

We let $B(X,r)$ denote balls in $\RR^{n+1}$ and let $\Delta(x,r)$ denote ``surface balls'' on $\partial\RR^{n+1}_\pm$, that is, balls in~$\RR^n$. 
If $Q\subset\RR^n$ or $Q\subset\RR^{n+1}$ is a cube, we let $\ell(Q)$ denote its side-length, and let $r Q$ denote the concentric cube with side-length $r\ell(Q)$. If $E$ is a set and $\mu$ is a measure, we let $\fint$ denote the average integral $\fint_E f\,d\mu=\frac{1}{\mu(E)} \int_E f\,d\mu$.

We will use the standard nontangential maximal function~$N$, as well as the modified nontangential maximal function $\widetilde N$ introduced in \cite{KP93}. These functions are defined as follows.
If $a>0$ is a constant and $x\in\RR^n$, then the \emph{nontangential cone} $\gamma_\pm(x)$ is given by
\begin{align}
\label{eqn:cone}
\gamma_{\pm}(x)=\bigl\{(y,s)\in\RR^{n+1}_\pm:\abs{x-y}<a \abs{s}\bigr\}.\end{align}

The nontangential maximal function and modified nontangential maximal function are given by 
\begin{align}
\label{eqn:NTM}
 N_{\pm} F(x) &= \sup\left\{\abs{F(y,s)}:(y,s)\in \gamma_\pm(x)\right\},
\\
\label{eqn:modNTM}
\widetilde N_{\pm} F(x) &= \sup
\biggl\{\biggl(\fint_{B((y,s),\abs{s}/2)} \abs{F}^2\biggr)^{1/2}:
(y,s)\in \gamma_\pm(x)
\biggr\}.
\end{align}
We remark that by \cite[Section~7, Lemma~1]{FefS72}, if we let 
\begin{equation*}N_a F(x)=\sup\left\{\abs{F(y,s)}:\abs{x-y}<a {s},\>0< s\right\},\end{equation*}
then for each $1\leq p\leq \infty$ and for each $0<b<a$, there is a constant $C$ depending only on $p$, $a$ and $b$ such that
$\doublebar{N_a F}_{L^p(\RR^n)} \leq C \doublebar{N_b F}_{L^p(\RR^n)}$. Thus, for our purposes, the exact value of $a$ in \eqref{eqn:cone} is irrelevant provided $a>0$.

Suppose that $A:\RR^{n+1}\mapsto\CC^{(n+1)\times(n+1)}$ is a bounded measurable matrix-valued function defined on $\RR^{n+1}$. We let $A^T$ denote the transpose matrix and let $A^*$ denote the adjoint matrix $\overline {A^T}$.
Recall from the introduction that $A$ is \emph{elliptic} if there exist constants $\Lambda>\lambda>0$ such that
\begin{equation}
\label{eqn:elliptic}
\lambda \abs{\eta}^2\leq \re\bar\eta\cdot A(X)\eta,
\quad
\abs{\xi\cdot A(X)\eta}\leq \Lambda\abs{\eta}\abs{\xi}
\end{equation}
for all $X\in\RR^{n+1}$ and all vectors $\eta$, $\xi\in\CC^{n+1}$. We refer to $\lambda$ and $\Lambda$ as the \emph{ellipticity constants of~$A$}.
Recall also that a scalar function $a$ is \emph{accretive} if 
\begin{equation}
\label{eqn:accretive}
\lambda\leq \re a(X)\leq \abs{a(X)}\leq \Lambda
\quad\text{for all }X\in\RR^{n+1}.
\end{equation}

We say that a function or coefficient matrix $A$ is \emph{$t$-independent} if
\begin{equation}
\label{eqn:tindependent}
A(x,t)=A(x,s) \quad\text{for all }x\in\RR^n\text{ and all }s,\>t\in\RR.
\end{equation}

\subsection{Elliptic equations and boundary-value problems}

If  $A$ is an elliptic matrix, then for any $u\in W^2_{1,loc}(\Omega)$, the expression $L_A u = -\Div A\nabla u\in W^2_{-1,loc}(\Omega)$ is defined by 
\begin{equation}\label{eqn:2weaksoln}
\langle \varphi, L_A u\rangle
=
\int_\Omega \overline\varphi \, L_A u = \int \overline{ \nabla \varphi}\cdot A\nabla u
\quad\text{for all} \quad \varphi\in C^\infty_0(\Omega).
\end{equation}
If $A$ and $B$ are elliptic matrices and $a$ is an accretive function, we may define $L_B^*(a\,L_Au)$ in the weak sense as follows.
\begin{dfn}\label{dfn:weaksoln}
Suppose $u\in W^2_{1,loc}(\Omega)$. Then $L_A u=-\Div A\nabla u$ is a well-defined element of $W^{2}_{-1,loc}(\Omega)$. Suppose that $a L_A u=v$, for some $v\in W^2_{1,loc}(\Omega)$, in the sense that
\begin{equation*}\int \overline{ \nabla \varphi}\cdot A\nabla u = \int \overline{ \varphi} \frac{1}{a} v
\quad\text{for all} \quad \varphi\in C^\infty_0(\Omega).\end{equation*}
If $v$ satisfies
\begin{equation*}\int \overline{\nabla v}\cdot B\nabla \eta=\int \eta \overline f\quad\text{for all} \quad \eta\in C^\infty_0(\Omega),\end{equation*}
that is, if $-\Div B^*\nabla v=f$ in the weak sense,
then we say that $L_B^*(a\,L_Au)=f$.
\end{dfn}

Suppose that $U$ is defined in $\RR^{n+1}_\pm$. We define the boundary values of $U$ as as the $L^2$ limit of $U$ up to the boundary, that is,
\begin{equation}
\label{eqn:trace}
U\big\vert_{\partial\RR^{n+1}_\pm} = F 
\quad\text{provided} \quad
\lim_{t\to 0^\pm}\doublebar{U(\,\cdot\,,t)-F}_{L^2(\RR^n)}=0.
\end{equation}

Given these definitions, we may define the Dirichlet problem for the fourth-order operator $L_B^*(a\,L_A)$ as follows.

\begin{dfn}\label{dfn:Dirichlet}
Suppose that  there is a constant $C_0$ such that, for any $f\in W^2_1(\RR^n)$ and any $g\in L^2(\RR^n)$, there exists a unique function $u\in W^2_{1,loc}(\RR^{n+1}_+)$ that satisfies
\begin{equation}
\label{eqn:Dirichletprob}
\left\{\begin{aligned}
L_B^*(a\,L_Au)&=0  &&\text{in }\RR^{n+1}_+,\\
u&=f, \quad 
\nabla_\parallel u=\nabla f  &&\text{on } \partial\RR^{n+1}_+,\\
\e\cdot A\nabla u &= g &&\text{on }\partial\RR^{n+1}_+,\\
\doublebar{\widetilde N_+(\nabla u)}_{L^2(\RR^n)}
+\triplebar{t\,L_A u}_+
&\multispan3${}\leq
C_0 \doublebar{\nabla f}_{L^2(\RR^n)}
+C_0 \doublebar{g}_{L^2(\RR^n)}$
\end{aligned}\right.
\end{equation}
where $L_B^*(a\,L_Au)=0$ in the sense of \dfnref{weaksoln} and where the boundary values are in the sense of \eqref{eqn:trace}.

Then we say that the $L^2$-Dirichlet problem for $L_B^*(a L_A)$ is well-posed in~$\RR^{n+1}_+$.
\end{dfn}
We specify $\nabla_\parallel u=\nabla f$ as well as $u=f$ in order to emphasize that $u(\,\cdot\,,t)\to f$ in $W^2_1(\RR^n)$ and not merely in~$L^2(\RR^n)$.

\begin{rmk} \label{rmk:4square}
The solutions $u$ to \eqref{eqn:Dirichletprob} constructed in the present paper will also satisfy the square-function estimate
\begin{equation}
\label{eqn:4square}
\triplebar{t\,\nabla \partial_t u}_+\leq 
C_1 \doublebar{\nabla f}_{L^2(\RR^n)}^2
+C_1 \doublebar{g}_{L^2(\RR^n)}^2
\end{equation}
for some constant $C_1$.
If $u$ is a solution to a fourth-order elliptic equation with \emph{constant} coefficients, then by \cite{DahKPV97} we have a square function bound on the complete Hessian matrix~$\nabla^2 u$. However, in the case of solutions to variable-coefficient operators in the composition form of \dfnref{weaksoln}, we do not expect all second derivatives to be well-behaved, and so \eqref{eqn:4square} cannot be strengthened.
\end{rmk}

Our main theorem is that, if $a$, $A$ and $B$ satisfy certain requirements, then the fourth-order Dirichlet problem \eqref{eqn:Dirichletprob} is well-posed. We now define these requirements.
We begin with the De Giorgi-Nash-Moser condition.
\begin{dfn}
We say that a function $u$ is locally H\"older continuous in the domain $\Omega$ if there exist constants $H$ and $\alpha>0$ such that, whenever $B(X_0,2r)\subset\Omega$, we have that
\begin{equation}
\label{eqn:DGN}
\abs{u(X)-u(X')}\leq H\left(\frac{\abs{X-X'}}{r}\right)^\alpha
\left(\fint_{B(X_0,2r)} \abs{u}^2\right)^{1/2}
\end{equation}
for all $X$, $X'\in B(X_0,r)$.
If $u$ is locally H\"older continuous in $B(X,r)$ for some $r>0$, then $u$ also satisfies Moser's ``local boundedness'' estimate 
\begin{equation}
\label{eqn:localbound}
\abs{u(X)}\leq C\left(\fint_{B(X,r)} \abs{u}^2\right)^{1/2}
\end{equation}
for some constant $C$ depending only on $H$ and the dimension $n+1$.

If $A$ is a matrix, we say that \emph{$A$ satisfies the De Giorgi-Nash-Moser condition}
if $A$ is elliptic and, for every open set $\Omega$ and every function $u$ such that $\Div A\nabla u=0$ in $\Omega$, we have that $u$ is locally H\"older continuous in $\Omega$, 
with constants $H$ and $\alpha$ depending only on $A$ (not on $u$ or $\Omega$).
\end{dfn}
Throughout we reserve the letter $\alpha$ for the exponent in the  estimate~\eqref{eqn:DGN}. We will show (see \crlref{4holder} below) that if $A$, $A^*$ and~$B^*$ satisfy the De Giorgi-Nash-Moser condition then solutions to $L_B^*(a\,L_Au)=0$ are also locally H\"older continuous.

We say that the $L^2$-regularity problem $(R)^A_2$ is well-posed in $\RR^{n+1}_\pm$ if, for each $f\in \dot W^2_1(\RR^n)$, there is a function~$u$, unique up to additive constants, that satisfies
\begin{equation*}
(R)^A_2\left\{\begin{aligned}
\Div A\nabla u&=0 && \text{in }\RR^{n+1}_\pm,\\
u &=f &&\text{on }\partial\RR^{n+1}_\pm,\\
\doublebar{\widetilde N_\pm (\nabla u)}_{L^2(\RR^n)}
&\multispan3${}\leq C\doublebar{\nabla f}_{L^2(\RR^n)}.$\hfil
\end{aligned}\right.
\end{equation*} 

\begin{rmk}\label{rmk:NTlimit}
If $\widetilde N_\pm(\nabla u)\in L^2(\RR^n)$, then averages of $u$ have a weak nontangential limit at the boundary, in the sense that there is some function $f$ such that
\[\biggl(\fint_{B((x,t),\abs{t}/2)} \abs{u(y,s)-f(x^*)}^2\,dy\,ds\biggr)^{1/2}
\leq C \abs{t}\widetilde N_\pm(\nabla u)(x^*)\]
for all $(x,t)\in\gamma_\pm(x^*)$.
See the proof of \cite[Theorem~3.1a]{KP93}; here $C$ is a constant depending only on the constant $a$ in the definition \eqref{eqn:cone} of~$\gamma_\pm$. If $u$ is locally H\"older continuous then this implies that $u$ itself has a nontangential limit, which by the dominated convergence theorem must equal its limit in the sense of \eqref{eqn:trace}. Thus, the requirement in \eqref{eqn:Dirichletprob} or $(R)^A_2$ that $u=f$ on $\partial\RR^{n+1}_\pm$ in the sense of vertical $L^2$ limits is equivalent to the requirement that $u=f$ on $\partial\RR^{n+1}_\pm$ in the sense of pointwise nontangential limits almost everywhere.



\end{rmk}

\subsection{The main theorems}
\label{sec:main}

The main theorems of this monograph are as follows.

\begin{thm} \label{thm:Dirichletexists}
Let $a:\RR^{n+1}\mapsto \RR$ and $A:\RR^{n+1}\mapsto \CC^{(n+1)\times(n+1)}$, where $n+1\geq 3$.
Assume that
\begin{itemize}
\item $a$ is accretive and $t$-independent.
\item $a$ is real-valued. 
\item  $A$ is elliptic and $t$-independent.
\item  $A$ and $A^*$ satisfy the De Giorgi-Nash-Moser condition. 
\item  The regularity problems $(R)^A_2$ and $(R)^{A^*}_2$ are well-posed in~$\RR^{n+1}_\pm$.
\end{itemize}


Then the $L^2$-Dirichlet problem for $L_{A}^* (a\,L_A)$ is well-posed in $\RR^{n+1}_+$, and the constants $C_0$ and $C_1$ in \eqref{eqn:Dirichletprob} and \eqref{eqn:4square} depend only on the dimension $n+1$, the ellipticity and accretivity constants $\lambda$, $\Lambda$ in \eqref{eqn:elliptic} and \eqref{eqn:accretive}, the De Giorgi-Nash-Moser constants $H$ and~$\alpha$, and the constants~$C$ in the definition of $(R)^A_2$ and~$(R)^{A^*}_2$.
\end{thm}

We can generalize \thmref{Dirichletexists} to the following perturbative version.

\begin{thm}\label{thm:perturb} 
Let $A$ be as in \thmref{Dirichletexists}, and let $a:\RR^{n+1}\mapsto \CC$ be accretive and $t$-independent. Let $B:\RR^{n+1}\mapsto\CC^{(n+1)\times(n+1)}$ be $t$-independent.

There is some $\varepsilon>0$, depending only on the quantities listed in \thmref{Dirichletexists}, such that if 
\begin{equation*}
\doublebar{\im a}_{L^\infty(\RR^n)}<\varepsilon \quad\text{and}\quad \doublebar{A-B}_{L^\infty(\RR^n)}<\varepsilon
,\end{equation*}
then the $L^2$-Dirichlet problem for $L_B^*(a\,L_A)$ is well-posed in~$\RR^{n+1}_+$, and the constants $C_0$ and $C_1$ in \eqref{eqn:Dirichletprob} and \eqref{eqn:4square} depend only on the quantities listed in \thmref{Dirichletexists}.
\end{thm}
We will see that if $\varepsilon$ is small enough then $B$ also satisfies the conditions of \thmref{Dirichletexists}; see \thmref{conditions}. It is possible to generalize from $\RR^{n+1}_+$ to domains above Lipschitz graphs; see \thmref{lipschitz} below.

We remark that throughout this paper, we will let $C$ denote a positive constant whose value may change from line to line, but which in general depends only on the quantities listed in \thmref{Dirichletexists}; any other dependencies will be indicated explicitly.

In the remainder of this section we will remind the reader of some known sufficient conditions for a matrix $A$ to satisfy the De Giorgi-Nash-Moser condition or for the regularity problem $(R)^A_2$ to be well-posed. To prove Theorems~\ref{thm:Dirichletexists} and~\ref{thm:perturb}, we will need some consequences of these conditions; we will establish notation for these consequences in \secref{2layer}.

We begin with the De Giorgi-Nash-Moser condition. Suppose that $A$ is elliptic.
It is well known that if $A$ is constant then solutions to $\Div A\nabla u=0$ are smooth (and in particular are H\"older continuous). More generally, the De Giorgi-Nash-Moser condition was proven to hold for real symmetric coefficients $A$ by De Giorgi and Nash in \cite{DeG57,Nas58} and extended to real nonsymmetric coefficients by Morrey in \cite{Mor66}. The De Giorgi-Nash-Moser condition is also valid if $A$ is $t$-independent and the ambient dimension $n+1=3$; this was proven in \cite[Appendix~B]{AlfAAHK11}.

Furthermore, this condition is stable under perturbation.
That is, let $A_0$ be elliptic, and suppose that $A_0$ and $A_0^*$ both satisfy the De Giorgi-Nash-Moser condition. Then there is some constant~$\varepsilon>0$, depending only on the dimension $n+1$ and the constants $\lambda$, $\Lambda$ in \eqref{eqn:elliptic} and $H$, $\alpha$ in~\eqref{eqn:DGN}, such that if $\doublebar{A-A_0}_{L^\infty(\RR^{n+1})}<\varepsilon$, then $A$ satisfies the De Giorgi-Nash-Moser condition.
This result is from 
\cite{Aus96}; see also \cite[Chapter~1, Theorems~6 and~10]{AusT98}.

We observe that in dimension $n+1\geq 4$, or in dimension $n+1=3$ for $t$-dependent coefficients, the De Giorgi-Nash-Moser condition may fail; see \cite{Fre08} for an example.

The regularity problem $(R)^A_2$ has been studied extensively. In particular, if $A$ is $t$-independent, then $(R)^A_2$ is known to be well-posed in $\RR^{n+1}_\pm$ provided $A$ is constant, real symmetric (\cite{KP93}), self-adjoint (\cite{AusAM10}), or of ``block'' form
\begin{equation*}
A(x)=\begin{pmatrix}A_\parallel(x) &\vec 0
\\ {\vec 0}{}^T & a_\perp(x)\end{pmatrix}
\end{equation*}
for some $n\times n$ matrix $A_\parallel$ and some complex-valued function $a_\perp$. The block case follows from validity of the Kato conjecture, as explained in \cite[Remark~2.5.6]{K94}; see \cite{AusHLMT02} for the proof of the Kato conjecture and \cite[Consequence~3.8]{AxeKM06} for the case $a_\perp\not\equiv 1$.

Furthermore, well-posedness of $(R)^A_2$ is stable under perturbation by \cite{AusAM10}; that is, if $(R)^{A_0}_2$ is well-posed in $\RR^{n+1}_\pm$ for some elliptic $t$-independent matrix~$A_0$, then so is $(R)^A_2$ for every elliptic $t$-independent matrix~$A$ with $\doublebar{A-A_0}_{L^\infty}$ small enough.

We mention that if $A$ is a nonsymmetric matrix, the $L^2$-regularity problem need not be well-posed, even in the case where $A$ is real. See the appendix to \cite{KR09} for a counterexample.

We may summarize the results listed above as follows.

\begin{thm}
\label{thm:conditions}
Let $A$ be elliptic and $t$-independent, and suppose that the dimension $n+1$ is at least~$3$. If $A$ satisfies any of the following conditions, then $A$ satisfies the conditions of \thmref{Dirichletexists}.
\begin{itemize}
\item $A$ is constant.
\item $A$ is real symmetric.
\item $A$ is a $3\times 3$ self-adjoint matrix.
\item $A$ is a real or $3\times 3$ block matrix. 
\end{itemize}
If $A_0$ satisfies the conditions of \thmref{Dirichletexists}, and if $\doublebar{A-A_0}_{L^\infty(\RR^n)}<\varepsilon$ for some $\varepsilon$ depending only on the quantities enumerated in \thmref{Dirichletexists}, then $A$ also satisfies the conditions of \thmref{Dirichletexists}.
\end{thm}

\subsection{Second-order boundary-value problems and layer potentials}
\label{sec:2layer}

In order to prove \thmref{Dirichletexists}, we will need well-posedness of several second-order boundary-value problems and good behavior of layer potentials; these conditions follow from well-posedness of $(R)^A_2$ and $(R)^{A^*}_2$. 

We say that the oblique $L^2$-Neumann problem $(N^\perp)^A_2$ is well-posed in $\RR^{n+1}_\pm$ if, for each $g\in L^2(\RR^n)$, there exists a unique (modulo constants) function $u$ that satisfies
\begin{equation*}
(N^\perp)^A_2\left\{\begin{aligned}
\Div A\nabla u&=0 && \text{in }\RR^{n+1}_\pm,
\\
\partial_{n+1} u &=g &&\text{on }\partial\RR^{n+1}_\pm,
\\
\doublebar{\widetilde N_\pm (\nabla u)}_{L^2(\RR^n)}
&\multispan3${}\leq C\doublebar{g}_{L^2(\RR^n)}.$\hfil
\end{aligned}\right.
\end{equation*} 
By \cite[Proposition~2.52]{AusAH08} and \cite[Corollary~3.6]{AusAM10} (see also \cite[Proposition~4.4]{AusAM10}), if $A$ is $t$-independent, then $(R)^{A^*}_2$ is well-posed in $\RR^{n+1}_\pm$ if and only if $(N^\perp)^A_2$ is well-posed in $\RR^{n+1}_\pm$.

We say that the $L^2$-Dirichlet problem $(D)^A_2$ is well-posed in $\RR^{n+1}_\pm$ if there is some constant $C$ such that, for each $f\in L^2(\RR^n)$, there is a unique function $u$ that satisfies
\begin{equation*}
(D)^A_2\left\{\begin{aligned}
\Div A\nabla u&=0 && \text{in }\RR^{n+1}_\pm,\\
u &=f &&\text{on }\partial\RR^{n+1}_\pm,\\
\doublebar{N_\pm u}_{L^2(\RR^n)}
&\multispan3${}\leq C\doublebar{f}_{L^2(\RR^n)}.$\hfil
\end{aligned}\right.
\end{equation*}
Observe that if $u$ is a solution to $(N^\perp)^A_2$ with boundary data~$f$, then $v=\partial_{n+1} u$ is a solution to $(D)^A_2$ with boundary data~$f$; thus, well-posedness of $(N^\perp)^A_2$ implies existence of solutions to $(D)^A_2$. Recall that well-posedness of $(N^\perp)^A_2$ also implies existence of solutions to $(R)^{A^*}_2$. It is possible to show that if $(R)^{A^*}_2$ solutions exist then solutions to $(D)^A_2$ are unique; see the proof of \cite[Lemma~4.31]{AlfAAHK11}.

Thus, if $(R)^{A^*}_2$ is well-posed in $\RR^{n+1}_\pm$ then so is $(D)^A_2$. We observe that this result was proven for real symmetric~$A$ by Kenig and Pipher in \cite{KP93}.

The $L^2$-Neumann problem $(N)^A_2$ more usually considered differs from $(N^\perp)^A_2$ in that the boundary condition is $\e\cdot A\nabla u=g$ rather than $\partial_{n+1} u=g$ on~$\partial\RR^{n+1}_+$. We remark that if $A$ is constant, self-adjoint or of block form then $(N)^A_2$ is known to be well-posed (again see \cite{AusAM10} or \cite[Remark~2.5.6]{K94} and \cite{AusHLMT02,AxeKM06}). $(N)^A_2$ is in many ways more natural than the oblique Neumann problem; however, we will not use well-posedness of the traditional Neumann problem and so do not provide a definition here.

A classic method for constructing solutions to second-order boundary-value problems is the method of layer potentials. 
We will use the double and single layer potentials of the second-order theory, as well as a new fourth-order potential (see \secref{fourthorderpotential}) to construct solutions to \eqref{eqn:Dirichletprob}; thus, we will need some properties of these potentials.

These potentials for second-order operators are defined as follows.
The fundamental solution to $L_A=-\Div A\nabla$ with pole at~$X$ is a function $\Gamma_X^A$ such that (formally)
$-\Div A\nabla \Gamma_X^A = \delta_X$.
For general complex coefficients $A$ such that $A$ and $A^*$ satisfy the De Giorgi-Nash-Moser condition, the fundamental solution was constructed by Hofmann and Kim. See \cite[Theorem 3.1]{HofK07} (reproduced as \thmref{fundsoln} below) for a precise definition of the fundamental solution.

If $f$ and $g$ are functions defined on $\RR^n$, the classical double and single layer potentials $\D_A f$ and $\s_A g$ are defined by the formulas
\begin{align}
\label{eqn:D}
\D_A f(x,t) &= -\int_{\RR^n} \e\cdot A^T(y) \nabla\Gamma_{(x,t)}^{A^T}(y,0)\, f(y) \,dy,\\
\label{eqn:S}
\s_A g(x,t) &= \int_{\RR^n} \Gamma_{(x,t)}^{A^T}(y,0)\,g(y) \, dy.
\end{align}
For well-behaved functions $f$ and~$g$, these integrals converge absolutely for $x\in\RR^n$ and for $t\neq 0$, and satisfy  $\Div A\nabla \D_A f=0$ and $\Div A\nabla \s_A g=0$ in $\RR^{n+1}\setminus\RR^n$; see \secref{goodlayer}.

We define the boundary layer potentials
\begin{align*}
\D_{A}^\pm f&=\D_A f\big\vert_{\partial\RR^{n+1}_\pm},
&
(\nabla \s_A)^\pm g&=\nabla \s_A g\big\vert_{\partial\RR^{n+1}_\pm}
,\\
\s_A^\pm g&=
\s_A g\big\vert_{\partial\RR^{n+1}_\pm},
&
\s_A^{\perp,\pm} g &= \partial_{n+1} \s_A g\big\vert_{\partial\RR^{n+1}_\pm}
\end{align*}
where the boundary values are in the sense of \eqref{eqn:trace}.

We remind the reader of the classic method of layer potentials for constructing solutions to boundary-value problems.
Suppose the nontangential estimate $\doublebar{\widetilde N_\pm(\nabla\s_A g)}_{L^2(\RR^n)}\leq C\doublebar{g}_{L^2(\RR^n)}$ is valid.
Then if $\s_A^\pm$ is invertible $L^2(\RR^n)\mapsto \dot W^2_1(\RR^n)$, then $u=\s_A((\s_A^\pm)^{-1} g)$ is a solution to $(R)^A_2$ with boundary data~$g$. Similarly, if $\s_A^{\perp,\pm}$ or the operator $g\mapsto \e\cdot A(\nabla \s_A)^\pm g$ is invertible on $L^2(\RR^n)$, then we may construct solutions to $(N^\perp)^A_2$ or $(N)^A_2$, respectively. The adjoint to $g\mapsto \e\cdot A(\nabla \s_A)^\pm g$ is the operator $\D_{A^*}^\mp$, and so if we can construct solutions to $(N)^A_2$ using the single layer potential then we can construct solutions to $(D)^{A^*}_2$ using the double layer potential. In the case of $t$-independent matrices, we may also construct solutions to $(D)^A_2$ by using invertibility of $\s_A^{\perp,\pm}$.

Thus, if layer potentials are bounded and invertible, we have well-posedness of boundary-value problems. The formulas \eqref{eqn:D} and \eqref{eqn:S} for solutions are often useful; thus, the layer potential results above are of interest even if well-posedness is known by other methods. It turns out that we can derive the layer potential results from well-posedness.

\begin{lem} \label{lem:regularitytobound}
Suppose $(R)^A_2$ and $(R)^{A^*}_2$ are well-posed in $\RR^{n+1}_\pm$.
Then there exists a constant $C$ such that
\begin{equation}
\label{eqn:Svertsquare}
\triplebar{t\,\partial_t^2\s_A g}^2
=
\int_{\RR^{n+1}} \abs{\partial_t^2 \s_A g(x,t)}^2
\,\abs{t}\,dx\,dt \leq C \doublebar{g}_{L^2(\RR^n)}^2
.\end{equation}
\end{lem}
\begin{proof}
Recall that if $(R)^A_2$ is well-posed then so is~$(D)^{A^*}_2$. The square-function estimate \eqref{eqn:Svertsquare} follows from well-posedness of $(D)^{A^*}_2$ and $(R)^{A^*}_2$ via a local $T(b)$ theorem for square functions.
This argument was carried out in \cite[Section~8]{AlfAAHK11} in the  case where $A$ is real and symmetric; we refer the reader to \cite[Section~5.3]{Gra12} for appropriate functions $b$ to use in the general case. (The interested reader should note that \cite[Chapter~5]{Gra12} is devoted to a proof of \eqref{eqn:Svertsquare} in the case of elliptic \emph{systems}.)
\end{proof}
It was observed in \cite[Proposition~1.19]{HofMM} that by results from \cite{AlfAAHK11} and \cite{AusA11}, if \eqref{eqn:Svertsquare} is valid then 
\begin{align}
\label{eqn:NnablaS}
\doublebar{\widetilde N_\pm(\nabla \s_A g)}_{L^2(\RR^n)}&\leq C\doublebar{g}_{L^2(\RR^n)}
.\end{align}

Using classic techniques involving jump relations, we will show (see \thmref{invertiblelayer} below) that if $(R)^A_2$ is well-posed in $\RR^{n+1}_\pm$ then $\s_A^\pm$ is invertible $L^2(\RR^n)\mapsto \dot W^2_1(\RR^n)$, and if in addition $(N^\perp)^A_2$ is well-posed then $\s_A^{\perp,\pm}$ is invertible. Thus, if $(R)^A_2$ and $(R)^{A^*}_2$ are well-posed, then not only do solutions to $(R)^A_2$, $(N^\perp)^A_2$ and $(D)^A_2$ exist, they are given by the formulas $u=\s_A((\s_A^\pm)^{-1} g)$, $u=\s_A((\s_A^{\perp,\pm})^{-1} g)$ and $u=\partial_{n+1}\s_A((\s_A^{\perp,\pm})^{-1} g)$, respectively.

In many of the results in the later parts of this paper, we will only need a few specific consequences of well-posedness of $(R)^A_2$ and $(R)^{A^*}_2$. These consequences are as follows.

\begin{dfn}\label{dfn:goodlayer}
Suppose that $A$ and $A^*$ are elliptic, $t$-independent and satisfy the De Giorgi-Nash-Moser condition. 

Suppose that 
\begin{itemize}
\item the square-function estimate \eqref{eqn:Svertsquare} is valid,
\item the operators $\s_A^{\perp,\pm}$ are invertible $L^2(\RR^n)\mapsto L^2(\RR^n)$, and 
\item solutions to the Dirichlet problem $(D)^A_2$ are unique in $\RR^{n+1}_+$ and in $\RR^{n+1}_-$.
\end{itemize}
Then we say that $A$ \emph{satisfies the single layer potential requirements}.

\end{dfn}
We remark that if $A$ satisfies the single layer potential requirements then $(N^\perp)^A_2$ (and hence $(R)^{A^*}_2$ and $(D)^A_2$) are well-posed, and so both $A$ and $A^*$ satisfy the single layer potential requirements if and only if $A$ satisfies the conditions of \thmref{Dirichletexists}. We will prove a few bounds under the assumption that $A$ (not necessarily $A^*$) satisfies the single layer potential requirements. We also remark that because we have no need of well-posedness of the Neumann problem $(N)^A_2$, we have not required invertibility of $\D_A^\pm$ or its adjoint. Consequently, our layer potential requirements are weaker than those considered elsewhere in the literature.

\subsection{Layer potentials for fourth-order differential equations}
\label{sec:fourthorderpotential}

We will construct solutions to the fourth-order Dirichlet problem \eqref{eqn:Dirichletprob} using potential operators. For $u=\E_{B,a,A} h$ to be a solution to $L_{B}^*(a\,L_A (\E h))=0$, we must have that $v=a\,L_A\E_{B,a,A} h$ is a solution to $L_{B^*} v=0$ in $\RR^{n+1}_+$. We choose $v=\partial_{n+1}^2 \s_{B^*} h$; if $B^*$ is $t$-independent then $L_B^*(\partial_{n+1}^2 \s_{B^*} h)=\partial_{n+1}^2 L_B^*(\s_{B^*} h)=0$. It will be seen that with this choice of~$v$, the operator $\E_{B,a,A}$ is bounded and invertible in some sense.

Formally, the solution to the equation
\begin{equation*}
-a\Div A\nabla \E_{B,a,A} h=\1_{\RR^{n+1}_+}\partial_{n+1}^2 \s_{B^*} h\end{equation*}
is given by
\begin{equation*}\E_{B,a,A} h(x,t) = 
\int_{\RR^n}\int_0^\infty
\Gamma_{(y,s)}^A(x,t) \,\frac{1}{a(y)}
\,\partial_s^2 \s_{B^*} h(y,s)\,ds\,dy
.\end{equation*}
However, to avoid certain convergence issues, we will instead define
\begin{align}
\label{eqn:E}
\E_{B,a,A} h(x,t) 
&= 
\F_{B,a,A} h(x,t) 
-\s_A\biggl(\frac{1}{a}\s_{B^*}^{\perp,+} h\biggr)(x,t)
\end{align}
where the auxiliary potential $\F_{B,a,A}$ is given by
\begin{align}
\label{eqn:F}
\F_{B,a,A} h(x,t) 
&= -\int_{\RR^n}\int_0^\infty
\partial_s \Gamma_{(y,s)}^A(x,t) \,\frac{1}{a(y)}
\,\partial_s \s_{B^*} h(y,s)\,ds\,dy
.\end{align}
That these two definitions are formally equivalent may be seen by integrating by parts in~$s$.

In \secref{potentials}, we will show that the integral in \eqref{eqn:F} converges absolutely for sufficiently well-behaved functions~$h$; we will see that if $A$ and $B^*$ satisfy the single layer potential requirements, then by the second-order theory the difference $\E_{B,a,A} h-\F_{B,a,A} h=-\s_A\bigl((1/a)\s_{B^*}^{\perp,+} h\bigr)$ is also well-defined and satisfies square-function and nontangential bounds.

\subsection{Lipschitz domains}
\label{sec:Lipschitz}

As we discussed in the introduction, by applying the change of variables $(x,t)\mapsto (x,t-\varphi(x))$, boundary-value problems for the second-order operator $\Div A\nabla$ in the domain $\Omega$ given by
\begin{equation}\label{eqn:specialLipschitz}
\Omega=\{(x,t):x\in\RR^n,\>t>\varphi(x)\}
\end{equation}
may be transformed to boundary-value problems in the upper half-space for an appropriate second-order operator $\Div \widetilde A\nabla$. In particular, the theory of harmonic functions in domains of the form \eqref{eqn:specialLipschitz} is encompassed by the theory of solutions to $\Div \widetilde A\nabla u=0$, for elliptic $t$-independent matrices~$\widetilde A$, in the upper half-space.


We now investigate the behavior of fourth-order operators under this (or another) change of variables.
Let $\rho:\Omega\mapsto\RR^{n+1}_+$ be any bilipschitz change of variables and let $J_\rho$ be the Jacobean matrix, so $\nabla (\tilde u\circ \rho)=J_\rho^T (\nabla \tilde u)\circ\rho$. For any accretive function~$a$ and elliptic matrix~$M$, we let $\tilde a$ and $\widetilde M$ be such that
\begin{align}
\label{eqn:changevars}
a&=\frac{\tilde a\circ\rho}{\abs{J_\rho}},&
J_\rho\,M\,J_\rho^T&=\abs{J_\rho}\, (\widetilde M\circ\rho).
\end{align}
By the weak definition \eqref{eqn:2weaksoln} of $L_{\widetilde M}=-\Div \widetilde M\nabla$ and elementary multivariable calculus, we have that if $\RR^{n+1}_+$ is a domain, and if $\tilde u\in W^2_{1,loc}(\RR^{n+1}_+)$ and $L_{\widetilde M} \tilde u \in L^2_{loc}(\RR^{n+1}_+)$, then 
\begin{equation}\label{eqn:Jacobean}
L_M(\tilde u\circ\rho) = \abs{J_\rho} \, (L_{\widetilde M} \tilde u)\circ\rho\quad\text{ in $\Omega=\rho^{-1}(\RR^{n+1}_+)$.}
\end{equation}
Observe that $u$ is H\"older continuous if and only if $\tilde u$ is; thus  $M$ satisfies the De Giorgi-Nash-Moser condition if and only if $\widetilde M$ does.

Now, suppose that $L_{\widetilde B}^*(\tilde a \,L_{\widetilde A} \tilde u)=0$ in~$\RR^{n+1}_+$, where $\tilde a$, $\widetilde A$, $\widetilde B$ are given by \eqref{eqn:changevars}. Let $u=\tilde u\circ\rho$. By \dfnref{weaksoln}, $\tilde v=\tilde a \,L_{\widetilde A} \tilde u$ lies in $W^2_{1,loc}(\RR^{n+1}_+)$. By \eqref{eqn:Jacobean}, $\tilde v\circ\rho= a\, (L_{ A}  u)\in W^2_{1,loc}(\Omega)$ and $L_{B^*}(\tilde v\circ\rho)=0$. Thus $ a\, (L_{ A}  u)\in W^2_{1,loc}(\Omega)$ and so $L_{ B}^*( a\,L_{ A} u)$ is well-defined in~$\Omega$, and furthermore $L_{ B}^*( a\,L_{ A} u)=0$. Thus, existence of solutions in $\Omega$ follows from existence of solutions in~$\RR^{n+1}_+$.

Similarly, if $L_{ B}^*( a\,L_{ A} u)=0$ in~$\Omega$ and $\tilde u=u\circ \rho^{-1}$ then $L_{\widetilde B}^*( \tilde a\,L_{\widetilde A} \tilde u)=0$ in $\RR^{n+1}_+$. Thus, uniqueness of solutions in $\Omega$ follows from uniqueness of solutions in~$\RR^{n+1}_+$.

We remark that the preceding argument is valid with $\Omega$ and $\RR^{n+1}_+$ replaced by $V$ and $\rho(V)$ for any domain~$V$. 

Because we wish to preserve $t$-independence, we consider only the change of variables $\rho(x,t)=(x,t-\varphi(x))$. This change of variables allows us to generalize \thmref{perturb} to domains $\Omega$ of the form \eqref{eqn:specialLipschitz}. The argument is straightforward; however, to state the result we must first define the fourth-order Dirichlet and second-order regularity problems in such domains.

Let $\Omega$ be a Lipschitz domain of the form \eqref{eqn:specialLipschitz}. Let $\nu$ denote the unit outward normal to~$\Omega$ and $\sigma$ denote surface measure on~$\partial\Omega$.
Let $W^2_1(\partial\Omega)$ denote the Sobolev space of functions in $L^2(\partial\Omega)$ whose weak tangential derivative also lies in $L^2(\partial\Omega)$; in both cases we take the norm with respect to surface measure. We say that $U=F$ on $\partial\Omega$ if $F$ is the vertical limit of $U$ in~$L^2(\partial\Omega)$, that is, if
\begin{equation}\label{eqn:Lipschitztrace}\lim_{t\to 0^+} \int_{\partial\Omega} \abs{ U(X+t\e_{n+1}) - F(X)}^2
\,d\sigma(X)=0.\end{equation}
If $f$ is defined on $\partial\Omega$, we let $\nabla_\tau f$ be the tangential gradient of $f$ along~$\partial\Omega$. If $u$ is defined in~$\Omega$, then $\nabla_\parallel u$ is the gradient of $u$ parallel to~$\partial\Omega$; that is, if $X=X_0+t\e_{n+1}$ for some $X_0\in\partial\Omega$ and some $t>0$, then $\nabla_\parallel u(X)=\nabla u(X)-(\nu(X_0)\cdot \nabla u(X))\nu(X_0)$.

If $\Omega$ is the domain above a Lipschitz graph, we say that the $L^2$-Dirichlet problem for $L_{B}^*(a\,L_A)$ is well-posed in~$\Omega$ if there is a constant~$C_0$ such that, for every $f\in W^2_1(\partial\Omega)$ and every $g\in L^2(\partial\Omega)$, there exists a unique function~$u$ that satisfies
\begin{equation*}
\left\{\begin{aligned}
L_B^*(a\,L_Au)&=0  &&\text{in }\Omega,\\
u&=f,\quad \nabla_\parallel u=\nabla_\tau f && \text{on }\partial\Omega,\\
\nu\cdot A\nabla u&=g && \text{on }\partial\Omega,\\
\doublebar{\widetilde N_\Omega(\nabla u)}_{L^2(\partial\Omega)}
&\multispan3${}\leq
C_0 \doublebar{\nabla_\tau f}_{L^2(\RR^n)}
+C_0 \doublebar{g}_{L^2(\RR^n)}
$,\\
\int_\Omega \abs{L_A u(X)}^2\,\dist(X,\partial\Omega)\,dX
&\multispan3${}\leq
C_0 \doublebar{\nabla_\tau f}_{L^2(\RR^n)}
+C_0 \doublebar{g}_{L^2(\RR^n)}
$.\end{aligned}\right.\end{equation*}
The modified nontangential maximal function $\widetilde N_\Omega$ is given by 
\begin{multline*}
\widetilde N_\Omega (\nabla u)(X)=\sup\biggl\{
\biggl(\fint_{B(Y,\dist(Y,\partial\Omega)/2)} \abs{\nabla u}^2\biggr)^{1/2}
\\
:Y\in\Omega,\>\abs{X-Y}<(1+a)\dist(Y,\partial\Omega)\biggr\}
.\end{multline*}
As in \rmkref{NTlimit}, we also have that $u\to f$ nontangentially in the sense that 
\begin{equation*}\lim_{Y\to X,\>Y\in\gamma(X)} u(Y)=f(X)
,\quad
\gamma(X)=\{
Y\in\Omega,\>\abs{X-Y}<(1+a)\dist(Y,\partial\Omega)\}.\end{equation*}

We say that the $L^2$-regularity problem $(R)^A_2$ is well-posed in $\Omega$ if for every $f\in W^2_1(\partial\Omega)$ there is a unique function $u$ that satisfies
\begin{equation*}
(R)^A_2\left\{\begin{aligned}
\Div A\nabla u&=0 && \text{in }\Omega,\\
u &=f &&\text{on }\partial\Omega,\\
\doublebar{\widetilde N_\Omega (\nabla u)}_{L^2(\RR^n)}
&\multispan3${}\leq C\doublebar{\nabla_\tau f}_{L^2(\RR^n)}$\hfil
\end{aligned}\right.
\end{equation*} 
where $u=f$ in the sense of either \eqref{eqn:Lipschitztrace} or in the sense of nontangential limits.
Clearly, $(R)^A_2$ is well-posed in $\Omega$ if and only if $(R)^{\widetilde A}_2$ is well-posed in $\RR^{n+1}_+$.

We may now generalize \thmref{perturb} to Lipschitz domains.
\begin{thm}
\label{thm:lipschitz}
Let $\Omega=\{(x,t):t>\varphi(x)\}$  for some Lipschitz function~$\varphi$. Let $a:\RR^{n+1}\mapsto\CC$, and let $A$, $B:\RR^{n+1}\mapsto \CC^{(n+1)\times (n+1)}$, where $n+1\geq 3$.

Suppose that $a$, $A$ and~$B$ are $t$-independent, that $a$ is accretive, that $A$ and $A^*$ satisfy the De Giorgi-Nash-Moser condition, and that $(R)^A_2$ and $(R)^{A^*}_2$ are well-posed in $\Omega$ and in~$\bar\Omega^C$.

Then there is some $\varepsilon>0$, depending only on $\doublebar{\nabla \varphi}_{L^\infty(\RR^n)}$ and the quantities listed in \thmref{Dirichletexists}, such that if 
\begin{equation*}
\doublebar{\im a}_{L^\infty(\RR^n)}<\varepsilon \quad\text{and}\quad \doublebar{A-B}_{L^\infty(\RR^n)}<\varepsilon
,\end{equation*}
then the $L^2$-Dirichlet problem for $L_B^*(a\,L_A)$ is well-posed in~$\Omega$.
\end{thm}

\section{Preliminaries: the second-order theory}
\label{sec:secondorder}

In this section, we will review some known results concerning solutions to second-order elliptic equations of the form $\Div A\nabla u=0$, and more specifically, concerning solutions to second-order boundary-value problems. In Sections~\ref{sec:fundamental} and~\ref{sec:goodlayer}, we will discuss the second-order fundamental solution and some properties of layer potentials.

In this section we will state several results valid in $\RR^{n+1}_+$; the obvious analogues are also valid in $\RR^{n+1}_-$.

The following two lemmas are well known.
\begin{lem}[The Caccioppoli inequality]\label{lem:Caccioppoli2}
Suppose that $A$ is elliptic. Let $X\in\RR^{n+1}$ and let $r>0$. Suppose that $\Div A\nabla u=0$ in $B(X,2r)$, for some $u\in W^2_1(B(X,2r))$. Then
\begin{equation*}\int_{B(X,r)}\abs{\nabla u}^2 \leq \frac{C}{r^2}
\int_{B(X,2r)\setminus B(X,r)} \abs{u}^2\end{equation*}
for some constant $C$ depending only on the ellipticity constants of $A$ and the dimension $n+1$.
\end{lem}
\begin{lem}\label{lem:PDE2} \textup{\cite[Theorem~2]{Mey63}.}
	Let $A$ be elliptic, and suppose that $\Div A\nabla u=0$ in $B(X,2r)$. Then there exists a $p>2$, depending only on the constants $\lambda$, $\Lambda$ in \eqref{eqn:elliptic}, such that 
	\begin{equation*}
	\left(\fint_{B(X,r)}\abs{\nabla u}^p\right)^{1/p}
	\leq
	C\left(\fint_{B(X,2r)}\abs{\nabla u}^2\right)^{1/2}
	.\end{equation*}	
\end{lem}


These conditions may be strengthened in the case of $t$-independent coefficients.
Suppose that $Q\subset\RR^n$ is a cube. If $u$ satisfies $\Div A\nabla u=0$ in $2Q\times(t-\ell(Q),t+\ell(Q))$, and if $A$ is $t$-independent, then by 
\cite[Proposition~2.1]{AlfAAHK11}, there is some $p_0>2$ such that if $1\leq p\leq p_0$, then
\begin{align}
\label{eqn:2slabs}
\biggl(\fint_Q \abs{\nabla u(x, t)}^p\,dx\biggr)^{1/p}
&\leq
C
\biggl(\fint_{2Q}\fint_{t-\ell(Q)/4}^{t+\ell(Q)/4} \abs{\nabla u(x, s)}^2\,ds\,dx\biggr)^{1/2}
.\end{align}
We will show that a similar formula holds for solutions to fourth-order equations in \lemref{slabs} below.

More generally, we have the following lemma.
\begin{lem} Suppose $\vec f\in L^2(\RR^n\mapsto\CC^{n+1})$ is a vector-valued function, and that $\Div A\nabla u=0$ in $2Q\times (t-\ell(Q),t+\ell(Q))$ for some $t$-independent elliptic matrix~$A$. Then
\begin{align}
\label{eqn:2slabsgeneral}
\fint_Q \abs{\nabla u(x, t)-\vec f(x)}^2\,dx
&\leq
C
\fint_{2Q}\fint_{t-\ell(Q)/2}^{t+\ell(Q)/2} \abs{\nabla u(x, s)-\vec f(x)}^2\,ds\,dx
.\end{align}
\end{lem}

\begin{proof}
Let $v(x,t)=\fint_{t-\ell(Q)/4}^{t+\ell(Q)/4} u(x,s)\,ds$. Then
\begin{equation*}\fint_Q \abs{\nabla u(x,t)-\vec f(x)}^2\,dx
\leq
2\fint_Q \abs{\nabla u(x,t)-\nabla v(x,t)}^2\,dx
+
2\fint_Q \abs{\nabla v(x,t)-\vec f(x)}^2\,dx
.\end{equation*}
Define $D(r)=2Q\times (t-r\ell(Q),t+r\ell(Q))$.
If $A$ is $t$-independent then $\Div A\nabla(u-v)=0$ in $D(3/4)$.
Applying \eqref{eqn:2slabs} to $u-v$, we have that 
\begin{equation*}\fint_Q \abs{\nabla u(\,\cdot\,,t)-\nabla v(\,\cdot\,,t)}^2\leq C\fint_{D(1/4)} \abs{\nabla u-\nabla v}^2.\end{equation*}
By definition of $v$,
\begin{equation*}\fint_Q \abs{\nabla v(x,t)-\vec f(x)}^2\,dx
=\fint_Q \biggabs{\fint_{t-\ell(Q)/4}^{t+\ell(Q)/4} \nabla u(x,s)-\vec f(x)\,ds}^2\,dx.\end{equation*}

Writing $\nabla u-\nabla v=(\nabla u-\vec f)+(\nabla v-\vec f)$, and again applying the definition of~$v$ to bound $\nabla v(x,s)-\vec f$ completes the proof.
\end{proof}

We also have the following theorems from \cite{AusA11}. Although we quote these theorems for $t$-independent coefficients only, in fact they are valid for $t$-dependent coefficients that satisfy a Carleson-measure condition.

\begin{thm}
\label{thm:squareN}
\textup{\cite[Theorem~2.4(i)]{AusA11}.}
Suppose that $A$ is $t$-independent and elliptic, that $\Div A\nabla u=0$ in $\RR^{n+1}_+$, and that $u$ satisfies the square-function estimate
\begin{equation*}
\int_{0}^\infty \int_{\RR^n} \abs{\nabla u(x,t)}^2 \,t\,dx\,dt<\infty.\end{equation*}
Then there is a constant $c$ and a function $f\in L^2(\RR^n)$ such that
\begin{equation*}\lim_{t\to 0^+} \doublebar{u(\,\cdot\,,t)-f-c}_{L^2(\RR^n)} = 0\end{equation*}
and such that 
\begin{equation*}\doublebar{\widetilde N_+ (u-c)}_{L^2(\RR^n)}^2 \leq C\int_{0}^\infty \int_{\RR^n} \abs{\nabla u(x,t)}^2 \,  t\,dx\,dt\end{equation*}
where $C$ depends only on the dimension $n+1$ and the ellipticity constants $\lambda$, $\Lambda$ of~$A$.
If $A$ satisfies the De Giorgi-Nash-Moser condition then we may replace $\widetilde N_+ u$ by~$N_+ u$.
\end{thm}

In the $t$-independent setting, \eqref{eqn:2slabsgeneral} lets us state {\cite[Theorem~2.3(i)]{AusA11}} as follows.
\begin{thm}
\label{thm:2L2limits}
Suppose that $\Div A\nabla u=0$ in $\RR^{n+1}_+$ and that $\widetilde N_+ (\nabla u)\in L^2(\RR^n)$, where $A$ is elliptic and $t$-independent. Then there exists a function $\vec G:\RR^n\mapsto\CC^{n+1}$ with $\doublebar{\vec G}_{L^2(\RR^n)}\leq C\doublebar{\widetilde N_+ (\nabla u)}_{L^2(\RR^n)}$ such that
\begin{equation*}
\lim_{t\to 0^+}
\doublebar{\nabla u(\,\cdot\,,t)-\vec G}_{L^2(\RR^n)}^2
=0 = \lim_{t\to\infty}\doublebar{\nabla u(\,\cdot\,,t)}_{L^2(\RR^n)}^2.\end{equation*}
\end{thm}

By the divergence theorem, there is a standard weak formulation of the boundary value $\e\cdot A\nabla u$ for any solution $u$ to $\Div A\nabla u=0$ with $\nabla u\in L^2(\RR^{n+1}_+)$. 
\thmref{2L2limits} implies that if $u$ is an appropriately bounded solution, then the boundary value $\e\cdot A\nabla u\vert_{\partial\RR^{n+1}_\pm}$ exists in the sense of $L^2$ limits. By the following result, these two formulations yield the same value.
\begin{thm}[{\cite[Lemma~4.3]{AlfAAHK11}}]
\label{thm:2traces}
Suppose that $A$ is elliptic, $t$-independent and satisfies the De Giorgi-Nash-Moser condition. Suppose that $\Div A\nabla u=0$ in $\RR^{n+1}_+$ and $\widetilde N_+ (\nabla u)\in L^2(\RR^n)$.

Then there is some function $g\in L^2(\RR^n)$ such that, if $\varphi\in C^\infty_0(\RR^{n+1})$, then
\begin{equation*}\int_{\RR^{n+1}_+} \nabla\varphi(x,t)\cdot A(x)\nabla u(x,t)\,dx\,dt = \int_{\RR^n} \varphi(x,0)\,g(x)\,dx.\end{equation*}
Furthermore, $-\e\cdot A\nabla u(\,\cdot\,,t)\to g$ as $t\to 0^+$ in $L^2(\RR^n)$.
\end{thm}

Given a function $u$ that satisfies $\Div A\nabla u=0$ in $\RR^{n+1}_\pm$, we will frequently analyze the  vertical derivative $\partial_{n+1}u$. In order to make statements about~$u$, given results concerning $\partial_{n+1} u$, we will need a uniqueness result.
\begin{lem}
\label{lem:utunique}
Let $A$ be elliptic and $t$-independent. Suppose that $\Div A\nabla u=0$ and $\partial_{n+1} u\equiv 0$ in $\RR^{n+1}_+$.

If there is some constant $c>0$ and some $t_0>0$ such that
\begin{equation}\label{eqn:graddecay}
\fint_{B((x,t),t/2)} \abs{\nabla u}^2 \leq c t^{-n}
\end{equation}
for all $t>t_0$ and all $x\in\RR^n$, then $u$ is constant in~$\RR^{n+1}_+$.
\end{lem}

If $\widetilde N_+(\nabla u)\in L^2(\RR^n)$ then \eqref{eqn:graddecay} is valid for all $t>0$. Thus in particular, uniqueness of solutions to $(D)^A_2$ implies uniqueness of solutions to $(N^\perp)^A_2$.

\begin{proof} Since $\partial_{n+1} u(x,t)=0$, we have that $u(x,t)=v(x)$ for some function $v:\RR^n\mapsto\CC$. By letting $t\to \infty$ in \eqref{eqn:graddecay} we have that $\nabla v\in L^2(\RR^n)$. 

Choose some $x\in\RR^n$ and some $t>0$. If $\tau>Ct$, then let $Y(\tau)$ be the cylinder $\Delta(x,t)\times (\tau-\tau/C,\tau+\tau/C)$. By H\"older's inequality and \lemref{PDE2}, there is some $p>2$ such that
\begin{align*}
\int_{\Delta(x,t)} \abs{\nabla v}^2
&=
\frac{C}{\tau} \int_{Y(\tau)} \abs{\nabla u}^2
\leq
\frac{C}{\tau} (t^n\tau)^{1-2/p} \left(\int_{B((x,\tau),\tau/C)} \abs{\nabla u}^p\right)^{2/p}
\\&\leq
Ct^{n(p-2)/p} \tau^{-n(p-2)/p}
\tau^n\fint_{B((x,\tau),2\tau/C)} \abs{\nabla u}^2
.\end{align*}
But if $\tau$ is large enough, then
$\tau^n\fint_{B((x,\tau),2\tau/C)} \abs{\nabla u}^2\leq c$.
By letting $\tau\to\infty$, we see that $\nabla v\equiv 0$ and so $u(x,t)=v(x)$ is a constant.
\end{proof}

We will also need the following uniqueness result.
\begin{lem}\label{lem:jumpunique}
Let $A$ be elliptic, $t$-independent and satisfy the De Giorgi-Nash-Moser condition. Suppose that $u_+$ and $u_-$ are two functions that satisfy
\begin{equation*}
\Div A\nabla u_\pm =0 \text{ in }\RR^{n+1}_\pm,
\quad \widetilde N_\pm(\nabla u_\pm) \in L^2(\RR^n),
\quad \nabla u_+\vert_{\partial\RR^{n+1}_+}=\nabla u_-\vert_{\partial\RR^{n+1}_+}.\end{equation*}
Then $u_+$ and $u_-$ are constant in $\RR^{n+1}_\pm$.
\end{lem}

\begin{proof}
Let $u=u_+$ in $\RR^{n+1}_+$, $u=u_-+c$ in $\RR^{n+1}_-$, where $c$ is such that $u_+=u_-$ on~$\RR^n$.
We claim that $\Div A\nabla u=0$ in the \emph{whole} space $\RR^{n+1}$ in the weak sense of \eqref{eqn:2weaksoln}. This is clearly true in each of the half-spaces $\RR^{n+1}_\pm$. Let $\varphi\in C^\infty_0(\RR^{n+1})$. Then
\begin{align*}
\int_{\RR^{n+1}}\nabla\varphi\cdot A\nabla u
&=
\int_{\RR^{n+1}_+}\nabla\varphi\cdot A\nabla u
+\int_{\RR^{n+1}_-}\nabla\varphi\cdot A\nabla u.
\end{align*}
Let $\vec G=\nabla u\vert_{\partial\RR^{n+1}_+}=\nabla u\vert_{\partial\RR^{n+1}_-}$. By \thmref{2traces}, we have that
\begin{align*}
\int_{\RR^{n+1}_\pm}\nabla\varphi(x,t)\cdot A(x)\nabla u(x,t)\,dx\,dt
&=
\mp\int_{\RR^n}\varphi(x,0)\,\e\cdot A(x)\vec G(x)\,dx
\end{align*}
and so $\int_{\RR^{n+1}}\nabla\varphi\cdot A\nabla u=0$; thus, $\Div A\nabla u=0$ in $\RR^{n+1}$.

We now show that $u$ is constant in all of $\RR^{n+1}$.
Fix some $X$, $X'\in\RR^{n+1}$. By the De Giorgi-Nash-Moser estimate \eqref{eqn:DGN} and by the Poincar\'e inequality, if $r$ is large enough then
\begin{equation*}\abs{u(X)-u(X')}
\leq 
Cr\biggl(\frac{\abs{X-X'}}{r}\biggr)^\alpha
\biggl(\fint_{B(0,2r)} \abs{\nabla u}^2\biggr)^{1/2}
.\end{equation*}
By definition of $\widetilde N_\pm (\nabla u)$ we have that
\begin{equation*}\abs{u(X)-u(X')}
\leq 
Cr^{1-n/2}\biggl(\frac{\abs{X-X'}}{r}\biggr)^\alpha
\biggl(\int_{\RR^{n+1}} \widetilde N_+ (\nabla u)^2 + \widetilde N_-(\nabla u)^2\biggr)^{1/2}
\end{equation*}
and so, taking the limit as $r\to\infty$, we have that $u$ is constant in $\RR^{n+1}$, as desired.
\end{proof}

\subsection{The fundamental solution}
\label{sec:fundamental}

We now discuss the second-order fundamental solution.
Let $2^*=2(n+1)/(n-1)$, and let $Y^{1,2}(\RR^{n+1})$ be the space of functions $u\in L^{2^*}(\RR^{n+1})$ that have weak derivatives $\nabla u$ that lie in $L^2(\RR^{n+1})$.
From \cite{HofK07}, we have the following theorems (essentially their Theorems 3.1 and~3.2).
\begin{thm}
\label{thm:fundsoln}
Assume that $A$ and $A^*$ are elliptic and satisfy the De Giorgi-Nash-Moser condition. Assume that $n+1\geq 3$. 

Then there is a unique fundamental solution $\Gamma^A_Y$ with the following properties.

\begin{itemize}
\item
$v(X,Y)=\Gamma^A_Y(X)$ is continuous in $\{(X,Y)\in \RR^{n+1} \times\RR^{n+1}:X\neq Y\}$.
\item
 $v(Y)=\Gamma^A_Y(X)$ is locally integrable in $\RR^{n+1}$ for any fixed $X\in\RR^{n+1}$.
 \item
For all smooth, compactly supported functions $f$ defined in $\RR^{n+1}$, the function $u$ given by
\begin{equation*}u(X) := \int_{\RR^{n+1}} \Gamma^A_Y(X) f(Y)\,dY\end{equation*}
belongs to $Y^{1,2}(\RR^{n+1})$ and satisfies $-\Div A\nabla u=f$ in the sense that
\begin{equation*}\int_{\RR^{n+1}} A\nabla u\cdot \nabla \varphi
=\int_{\RR^{n+1}} f\varphi\end{equation*}
for all $\varphi$ smooth and compactly supported in $\RR^{n+1}$.
\end{itemize}
\end{thm}

\begin{thm}
\label{thm:fundsoln:properties}
$\Gamma^A$ has the property
\begin{equation}
\label{eqn:fundsolnweakdfn}
\int_{\RR^{n+1}} A\nabla\Gamma^A_Y\cdot\nabla\varphi
=\varphi(Y)
\end{equation}
for all $Y\in\RR^{n+1}$ and all $\varphi$ smooth and compactly supported in $\RR^{n+1}$.

Furthermore, $\Gamma^A$ satisfies the following estimates:
\begin{align}
\label{eqn:fundsolnsize}
\abs{\Gamma^A_Y(X)}&\leq \frac{C}{\abs{X-Y}^{n-1}}
\\
\label{eqn:gradsize:far}
\doublebar{\nabla\Gamma^A_Y}_{L^2(\RR^{n+1}\setminus B(Y,r))} &\leq Cr^{(1-n)/2},
\\
\label{eqn:gradsize:near}
\doublebar{\nabla\Gamma^A_Y}_{L^p(B(Y,r))} &\leq Cr^{-n+(n+1)/p}\quad\text{if}\quad 1\leq p<\frac{{n+1}}{n},
\end{align}
for some $C$ depending only on the dimension $n+1$, the constants $\lambda$, $\Lambda$ in \eqref{eqn:elliptic}, and the constants $H$, $\alpha$ in the De Giorgi-Nash-Moser bounds. 

Finally, if $f\in L^{2(n+1)/(n+3)}(\RR^{n+1})\cap L^p_{loc}(\RR^{n+1})$ for some $p>(n+1)/2$, then
\begin{equation*}u(X)=\int_{\RR^{n+1}} \Gamma^A_Y(X) f(Y)\,dY\end{equation*}
is continuous, lies in $Y^{1,2}(\RR^{n+1})$, and satisfies
\begin{equation*}\int_{\RR^{n+1}} A\nabla u\cdot \nabla \varphi
=\int_{\RR^{n+1}} f\varphi\end{equation*}
for all $\varphi$ smooth and compactly supported in $\RR^{n+1}$.
\end{thm}

We will need some additional properties of the fundamental solution. First, by uniqueness of the fundamental solution, we have that if $A$ is $t$-independent then
$\Gamma_{(y,s)}^A(x,t)=\Gamma_{(y,s+r)}^A(x,t+r)$ for any $x$, $y\in\RR^n$ and any $r$, $s$, $t\in\RR$. In particular, this implies that
\begin{equation}
\label{eqn:fundsoln:vert}
\partial_s^k\Gamma_{(y,s)}^A(x,t)=(-1)^k \partial_t^k \Gamma_{(y,s)}^A(x,t).
\end{equation}

Second, by the Caccioppoli inequality and the De Giorgi-Nash-Moser estimates, we have the bounds
\begin{align}
\label{eqn:fundsolnsize:k}
\abs{\partial_{n+1}^k\Gamma_Y^{A}(X)}&\leq \frac{C_k}{\abs{X-Y}^{{n}+k-1}},
\\
\label{eqn:fundsolnholder:k}
\abs{\partial_{n+1}^k\Gamma_Y^{A}(X)-\partial_{n+1}^k\Gamma_{Y}^{A}(X')}
&\leq 
\frac{C_k\abs{X-X'}^\alpha}{\abs{X-Y}^{{n}+k+\alpha-1}}
\end{align}
for any $k\geq 0$ and any $X$, $X'$, $Y\in\RR^n$ with $\abs{X-X'}\leq \frac{1}{2}\abs{X-Y}$.

Finally, it is straightforward to show that if $X$, $Y\in\RR^{n+1}$ with $X\neq Y$, then 
\begin{equation}
\label{eqn:fundsoln:symm}
\Gamma_X^A(Y)=\Gamma_Y^{A^T}(X)
.\end{equation}

\subsection{Layer potentials}
\label{sec:goodlayer}

Recall that if $A$ and $A^*$ satisfy the De Giorgi-Nash-Moser condition, then the double and single layer potentials are given by the formulas
\begin{align*}
\D_A f(X) &= -\int_{\RR^n} f(y) \, \e\cdot A^T(y) \nabla\Gamma_{X}^{A^T}(y,0)\,dy,\\
\s_A g(X) &= \int_{\RR^n} g(y) \, \Gamma_{X}^{A^T}(y,0)\,dy.
\end{align*}
Notice that by \eqref{eqn:fundsoln:symm}, $\Div A\nabla \D_A f=0$ and 
$\Div A\nabla\partial_{n+1}^k \s_A g=0$ in $\RR^{n+1}\setminus\RR^n$.

We will be most concerned with the case where $f$, $g\in L^2(\RR^n)$.
If $A$ is $t$-independent, then by \eqref{eqn:2slabs} and \eqref{eqn:gradsize:far}, the integral in the definition of $\D_A f(X)$ converges absolutely for all $X\in \RR^{n+1}\setminus\RR^n$ and all $f\in L^2(\RR^n)$. Similarly, by \eqref{eqn:fundsolnsize}, if $t\neq 0$ and if $g\in L^p(\RR^n)$ for some $1\leq p<n$ then the integral in the definition of $\s_A g(x,t)$ converges absolutely.
If $n\geq 3$ this implies that $\s_A g(x,t)$ is well-defined for all $g\in L^2(\RR^n)$. If $n=2$, so the ambient dimension $n+1=3$, then $\s_A g$ is well-defined \emph{up to an additive constant} for $g\in L^2(\RR^2)$. 
That is, if $n=2$ and if $g\in L^2(\RR^2)$, or more generally if $g\in L^p(\RR^n)$ for some $1\leq p<n/(1-\alpha)$, then by \eqref{eqn:fundsolnholder:k}, the integral
\begin{equation*}\s_A g(X)-\s_A g(X_0) = \int_{\RR^n} g(y) \, (\Gamma_{(y,0)}^A(X)-\Gamma_{(y,0)}^A(X_0))\,dy \end{equation*}
converges absolutely for all $X$, $X_0\in \RR^{n+1}\setminus\RR^n$. We remark that by \eqref{eqn:fundsolnsize:k}, if $k\geq 1$ then
\begin{equation*}\int_{\RR^n} g(y) \, \partial_t^k\Gamma_{(y,0)}^A(x,t)\,dy\end{equation*}
converges absolutely provided $g\in L^p(\RR^n)$, $1\leq p <\infty$, and if $k\geq 2$ then the integral also converges if $g\in L^\infty(\RR^n)$. We will write
\begin{equation*}
\partial_t^k \s_A g(x,t)=\int_{\RR^n} g(y) \, \partial_t^k\Gamma_{(y,0)}^A(x,t)\,dy
\end{equation*}
for all such~$g$, even if $\s_A g$ does not converge absolutely.

By \eqref{eqn:fundsolnsize:k} we have a pointwise bound on $\partial_t^k\s_A g(x,t)$. If $1\leq p\leq \infty$ then for all integers $k\geq 2$ we have that
\begin{equation}
\label{eqn:Shdecay}
\abs{\partial_t^k \s_{A} g(x,t)}\leq 
\frac{C(p)}{(\abs{t}+\dist(x,\supp g))^{n/p+k-1}} \doublebar{g}_{L^p(\RR^n)}.
\end{equation}
This also holds for $k=1$ provided $1\leq p<\infty$ and for $k=0$ provided $1\leq p<n$.

We now establish that if $(R)^A_2$ and $(R)^{A^*}_2$ are well-posed, then certain properties of the single layer potential follow. Recall from \lemref{regularitytobound} that under these conditions the square-function estimate \eqref{eqn:Svertsquare} is valid.
By \cite[Formula~(5.5)]{AlfAAHK11}, if the square-function estimate \eqref{eqn:Svertsquare} is valid then it may be strengthened to the following estimate on the whole gradient:
\begin{align}
\label{eqn:Ssquare}
\triplebar{t\,\nabla \partial_t\s_A g}^2
&=
\int_{\RR^{n+1}} \abs{\nabla \partial_t \s_A g(x,t)}^2 \abs{t}\,dx\,dt
\leq C \doublebar{g}_{L^2(\RR^n)}^2
.\end{align}

Suppose that $A$ is $t$-independent and has bounded layer potentials, meaning that $\doublebar{\widetilde N_\pm(\nabla \s_A g)}_{L^2(\RR^n)}\leq C\doublebar{g}_{L^2(\RR^n)}$. (This is Formula~\eqref{eqn:NnablaS}; recall that it follows from \eqref{eqn:Svertsquare}.)
By \thmref{2L2limits},  the operators $\s_A^{\perp,\pm}$ and $(\nabla \s_A)^\pm$  are well-defined and bounded on~$L^2(\RR^n)$.
Observe that $\s_A^{\perp,+}$ is a Calder\'on-Zygmund operator by \eqref{eqn:fundsolnsize:k} and \eqref{eqn:fundsolnholder:k}. Thus,  by standard Calder\'on-Zygmund theory, $\s_A^{\perp,+}$ is bounded on $L^p(\RR^n)$ for any $1<p<\infty$.
By a standard argument  (see \cite[Proposition 4.3]{DahV90}), we may strengthen this to a nontangential bound: if $g\in L^p(\RR^n)$ for any $1<p<\infty$, then
\begin{align}
\label{eqn:NSt}\doublebar{N_\pm(\partial_{n+1}\s_{A}g)}_{L^p(\RR^n)}
&\leq C(p)\doublebar{g}_{L^p(\RR^n)}.
\end{align}

The following formulas come from \cite{AlfAAHK11}.
Again suppose that $A$ has bounded layer potentials.
If $g\in L^2(\RR^n)$, then by the proof of \cite[Lemma 4.18]{AlfAAHK11} we have that
\begin{align}
\label{eqn:Sjump}
\e\cdot A(\nabla \s_A)^+g - \e\cdot A(\nabla \s_A)^-g&= -g
,\\
\label{eqn:Scts}
\nabla_\parallel \s_A^+ g - \nabla_\parallel \s_A^- g &= 0
.\end{align}
In particular $\s_A^+=\s_A^-$ regarded as operators $L^2(\RR^n)\mapsto\dot W^2_1(\RR^n)$.

We now establish invertibility of~$\s_A^+$ and $\s_A^{\perp,\pm}$.
\begin{thm} \label{thm:invertiblelayer}
Let $A$ and $A^*$ be $t$-independent and satisfy the De Giorgi-Nash-Moser condition. Suppose that $A$ has bounded layer potentials in the sense that \eqref{eqn:NnablaS} is valid.

If $(R)^A_2$ is well-posed in $\RR^{n+1}_+$ and $\RR^{n+1}_-$, then the operator $\s_A^+$ is invertible $L^2(\RR^n)\mapsto \dot W^2_1(\RR^n)$.

If in addition $(N^\perp)^A_2$ is well-posed in $\RR^{n+1}_\pm$, then $\s_A^{\perp,\pm}$ is invertible on $L^2(\RR^n)$.

\end{thm}

\begin{proof}
The proof exploits extensively the jump relations  for the single layer potential. By \eqref{eqn:Sjump}, \eqref{eqn:NnablaS} and \thmref{2L2limits}, if $g\in L^2(\RR^n)$ then 
\begin{align*}\doublebar{g}_{L^2(\RR^n)}
&=
\doublebar{e\cdot A(\nabla \s_A)^+ g-\e\cdot A(\nabla \s_A)^- g}_{L^2(\RR^n)}
\\&\leq C\doublebar{\widetilde N_+(\nabla \s_A g)}_{L^2(\RR^n)}
+C\doublebar{\widetilde N_-(\nabla \s_A g)}_{L^2(\RR^n)}
.\end{align*}
But if $(R)^A_2$ is well-posed in $\RR^{n+1}_\pm$, then $\doublebar{\widetilde N_\pm(\nabla \s_A g)}_{L^2(\RR^n)}\leq C\doublebar{\nabla_\parallel \s_A^\pm g}_{L^2(\RR^n)}$. By \eqref{eqn:Scts} we have that $\nabla_\parallel \s_A^+ g=\nabla_\parallel \s_A^- g$ and so
\begin{equation*}\doublebar{g}_{L^2(\RR^n)}
\leq C\doublebar{\nabla_\parallel \s_A^\pm g}_{L^2(\RR^n)}
.\end{equation*}

We need only show that $\s_A^+$ is surjective. Choose some $f\in \dot W^2_1(\RR^n)$. Let $u_\pm$ be the solutions to $(R)^A_2$ with boundary data~$f$ in $\RR^{n+1}_\pm$. By \thmref{2L2limits}, the functions $g_\pm = \e\cdot A\nabla u_\pm\vert_{\partial\RR^{n+1}_\pm}$ exist and lie in $L^2(\RR^n)$.

Now, let $v=\s_A(g_+-g_-)$, and let $v_\pm=v\vert_{\RR^{n+1}_\pm}$. By \eqref{eqn:NnablaS}, $\widetilde N_\pm(\nabla v_\pm)\in L^2(\RR^n)$. Consider $w_\pm=u_\pm+v_\pm$. By definition of $u_\pm$ and by the continuity relation \eqref{eqn:Scts}, we have that
\begin{equation*}\nabla_\parallel w_+\big\vert_{\partial\RR^{n+1}_+}
=\nabla_\parallel w_-\big\vert_{\partial\RR^{n+1}_-}.\end{equation*}
Consider the conormal derivative. We have that
\begin{equation*}\e\cdot A\nabla w_+\big\vert_{\partial\RR^{n+1}_+}
=
	\e\cdot A\nabla u_+\big\vert_{\partial\RR^{n+1}_+}
	+\e\cdot A\nabla v_+\big\vert_{\partial\RR^{n+1}_+}
=
	g_+
	+\e\cdot A(\nabla \s_A)^+(g_+-g_-)
.\end{equation*}
But by \eqref{eqn:Sjump},
\begin{align*}\e\cdot A\nabla w_+\big\vert_{\partial\RR^{n+1}_+}
&=
	g_+
	+\e\cdot A(\nabla \s_A)^-(g_+-g_-)
	-(g_+ - g_-)
\\&=
	\e\cdot A\nabla u_-\big\vert_{\partial\RR^{n+1}_-}
	+\e\cdot A\nabla v_-\big\vert_{\partial\RR^{n+1}_-}
\\&=\e\cdot A\nabla w_-\big\vert_{\partial\RR^{n+1}_-}
.\end{align*}
Thus by \lemref{jumpunique}, we have that $w_\pm$ is constant in $\RR^{n+1}_\pm$. In particular, $\s_A^\pm (g_+-g_-)=f$ in $\dot W^2_1(\RR^n)$, as desired.

If in addition $(N^\perp)^A_2$ is well-posed in $\RR^{n+1}_\pm$, then
\begin{equation*}\doublebar{g}_{L^2(\RR^n)}
\leq C\doublebar{\nabla_\parallel \s_A^\pm g}_{L^2(\RR^n)}
\leq C\doublebar{\widetilde N_\pm(\nabla \s_A  g)}_{L^2(\RR^n)}
\leq C\doublebar{\s_A^{\perp,\pm} g}_{L^2(\RR^n)}
\end{equation*}
and so we need only show that $\s_A^{\perp,\pm}$ is surjective on $L^2(\RR^n)$. 

If $(N^\perp)^A_2$ is well-posed in $\RR^{n+1}_\pm$, then for each $g\in L^2(\RR^n)$ there exists some $u$ with $\widetilde N_\pm(\nabla u)\in L^2(\RR^n)$ and with $\partial_{n+1} u=g$ on $\partial\RR^{n+1}_\pm$. Let $F=u\vert_{\partial\RR^{n+1}_\pm}$; by \thmref{2L2limits} $F$ exists and $F\in \dot W^2_1(\RR^n)$. Because $\s_A^\pm$ is invertible, we have that $F=\s_A^\pm f$ for some $f\in L^2(\RR^n)$; by uniqueness of regularity solutions, $\s_A^{\perp,\pm} f=g$.
\end{proof}

We will need some bounds on the double layer potential as well. Suppose that $(R)^A_2$ and $(R)^{A^*}_2$ are both well-posed in $\RR^{n+1}_\pm$, so \eqref{eqn:Svertsquare} is valid.
By \cite[Corollary~4.28]{AlfAAHK11}, if $t>0$ then for all $g\in C^\infty_0(\RR^n)$ we have that
\begin{equation}
\label{eqn:DSformula}
\D_A (\s_A^+ g)(x,\pm t)=\mp \s_A(\e\cdot A(\nabla \s_A )^\mp g)(x,\pm t).
\end{equation} 

We have the following consequences of~\eqref{eqn:DSformula}. By \thmref{invertiblelayer}, $\s_A^+:L^2(\RR^n)\mapsto \dot W^2_1(\RR^n)$ is invertible. Let $f\in W^2_1(\RR^n)$, and let $g=(\s_A^+)^{-1}f$. By \eqref{eqn:NnablaS} and \eqref{eqn:Ssquare},
\begin{align}
\label{eqn:NnablaD}
\doublebar{\widetilde N_\pm(\nabla \D_A f)}_{L^2(\RR^n)}&\leq C\doublebar{\nabla f}_{L^2(\RR^n)}
,\\
\label{eqn:Dtsquare}
\triplebar{t\,\nabla \partial_t\D_A f}^2&=
\int_{\RR^{n+1}} \abs{\nabla \partial_t\D_A f(x,t)}^2 \abs{t}\,dx\,dt
\leq C \doublebar{\nabla f}_{L^2(\RR^n)}^2.
\end{align}
Combining \eqref{eqn:DSformula} with the jump relations \eqref{eqn:Sjump} and \eqref{eqn:Scts}, we have that
\begin{align}
\label{eqn:Dcts}
\e\cdot A\nabla \D_A f\big\vert_{\partial\RR^{n+1}_+}-\e\cdot A\nabla \D_A f\big\vert_{\partial\RR^{n+1}_-} &= 0
,\\
\label{eqn:Djump}
\nabla_\parallel\D_A^+ f-\nabla_\parallel\D_A^- f &= -f
\end{align}
for all $f\in W^2_1(\RR^n)$. 


\section{The Caccioppoli inequality and related results}
\label{sec:Caccioppoli}
The Caccioppoli inequality for second-order elliptic equations (\lemref{Caccioppoli2}) is well known. 
In this section, we will prove a similar inequality for weak solutions to the fourth-order equation $L_B^*(a\,L_Au)=0$. We will also prove fourth-order analogs to some other basic results of the second-order theory.

\begin{thm}
\label{thm:Caccioppoli}
Suppose that $L_B^* (a\,L_A u)=0$ in $B(X,2r)$ in the sense of \dfnref{weaksoln}. Suppose that $A$ and $B$ are elliptic in the sense of \eqref{eqn:elliptic} and that $a$ is accretive in the sense of \eqref{eqn:accretive}. Then
\begin{equation*}\int_{B(X,r)} \abs{L_A u}^2 + \frac{1}{r^2}\int_{B(X,r)}\abs{\nabla u}^2\leq \frac{C}{r^4}\int_{B(X,2r)} \abs{u}^2\end{equation*}
where $C$ depends only on the constants $\lambda$, $\Lambda$ in \eqref{eqn:elliptic} and \eqref{eqn:accretive}.
\end{thm}

\begin{proof}
Let $\varphi$ be a real smooth cutoff function, so that $\varphi=1$ on $B(X,r)$,
$\varphi$ is supported in $B(X,2r)$, and $\abs{\nabla\varphi}\leq C/r$.

Then, for any constant $c_1$, we have that
\begin{align*}
\int \varphi^2 \abs{\nabla u}^2
&\leq \frac{1}{\lambda}\re\int \varphi^2 \nabla \bar u\cdot A\nabla u
\\&=
\frac{1}{\lambda}\re\int \varphi^2\,\bar u\,L_A u
-\frac{1}{\lambda}\re\int 2 \bar u \varphi\,\nabla\varphi\cdot A\nabla u
\\&\leq
\frac{C c_1 }{2 r^2} \int_{B(X,2 r)} \abs{u}^2
+\frac{r^2}{2c_1} \int\varphi^4 \abs{L_A u}^2
+\frac{C}{2}\int \abs{u}^2\abs{\nabla\varphi}^2
+\frac{1}{2}\int \varphi^2\abs{\nabla u}^2
\end{align*}
and so
\begin{align*}
\int \varphi^2 \abs{\nabla u}^2
&\leq
\frac{C(c_1+1)}{2 r^2} \int_{B(X,2 r)} \abs{u}^2
+\frac{r^2}{c_1} \int\varphi^4 \abs{L_A u}^2.
\end{align*}

Now, recall that $a L_A u=v$ for some $v\in W^2_{1,loc}(B(X,2r))$. Therefore,
\begin{align}
\label{eqn:LAuv}
\abs{L_A u}^2
&\leq \frac{\re \bar a}{\lambda}\,\overline{L_A u} L_A u
=\frac{1}{\lambda}\re(\bar v\, L_A u),
\\
\label{eqn:vLAu}
\abs{v}^2
&\leq \frac{1}{\re (1/a)}\re\left(\bar v\, \frac{1}{a} v\right)
\leq \frac{\Lambda^2}{\lambda} \re(\bar v\,L_A u).
\end{align}

Because $\varphi$ is compactly supported, the weak definition of $L_A u$ implies that
\begin{align*}
\int\varphi^4 \,\bar v\,L_A u
&=
\int \nabla(\varphi^4 \bar v)  \cdot A\nabla u
=
\int 4\varphi^3 \bar v\, \nabla\varphi  \cdot A\nabla u
+\int \varphi^4 \, \nabla \bar v  \cdot A\nabla u
\end{align*}
and so for any $c_2>0$ we have that
\begin{align*}
\re \int\varphi^4 \,\bar v\,L_A u
&\leq
\frac{\lambda}{2\Lambda^2}\int \varphi^4 \abs{v}^2 
+C\int \varphi^2 \abs{\nabla \varphi}^2\abs{\nabla u}^2
\\&\qquad
+\frac{r^2}{2c_2}\int \varphi^6\abs{\nabla v}^2
+\frac{Cc_2}{r^2}\int \varphi^2\abs{\nabla u}^2
\\&\leq
\frac{1}{2}\re\int \varphi^4 \bar v L_A u 
+\frac{r^2}{2c_2}\int \varphi^6\abs{\nabla v}^2
+\frac{C(1+c_2)}{r^2}\int \varphi^2\abs{\nabla u}^2
.\end{align*}
This implies that
\begin{align*}
\re\int\varphi^4 \,\bar v\,L_A u
&\leq
\frac{C(1+c_2)}{r^2}\int \varphi^2\abs{\nabla u}^2
+\frac{r^2}{c_2}\int \varphi^6\abs{\nabla v}^2.
\end{align*}

But $\Div B^*\nabla v=0$ in the weak sense, so 
\begin{equation*}0=\int \nabla(\varphi^6 \bar v)\cdot B^*\nabla v.\end{equation*}
As in the proof of the second-order Caccioppoli inequality, and by \eqref{eqn:vLAu},
\begin{equation*}\int \varphi^6\abs{\nabla v}^2
\leq C\int \varphi^4\abs{\nabla\varphi}^2\abs{v}^2
\leq C\re\int \varphi^4\abs{\nabla\varphi}^2\bar v\, L_A u
.\end{equation*}
So
\begin{align*}
\re\int\varphi^4 \,\bar v\,L_A u
&\leq
\frac{C(1+c_2)}{r^2}\int \varphi^2 \abs{\nabla u}^2
+\frac{C}{c_2}\re\int \varphi^4 \bar v L_A u
.\end{align*}
Choosing $c_2$ large enough, we see that
\begin{align*}
\int\varphi^4 \abs{L_A u}^2
&\leq \frac{1}{\lambda}\re\int\varphi^4 \,\bar v\,L_A u
\leq
\frac{C}{r^2}\int \varphi^2 \abs{\nabla u}^2
\end{align*}
and so
\begin{align*}
\int \varphi^2 \abs{\nabla u}^2
&\leq
\frac{C(c_1+1)}{2 r^2} \int_{B(X,2r)} \abs{u}^2
+\frac{C}{c_1} \int \varphi^2 \abs{\nabla u}^2
.\end{align*}
Choosing $c_1$ large enough lets us conclude that
\begin{equation*}
\int\varphi^4 \abs{L_A u}^2
\leq
\frac{C}{r^2}\int \varphi^2 \abs{\nabla u}^2
\leq\frac{C}{r^4}\int_{B(X,2r)} \abs{u}^2
\end{equation*}
and so
\begin{align*}
\int_{B(X,r)}\abs{L_A u}^2
+
\frac{1}{r^2}\int_{B(X,r)} \abs{\nabla u}^2
\leq\frac{C}{r^4}\int_{B(X,2r)} \abs{u}^2
\end{align*}
as desired.
\end{proof}

We now prove H\"older continuity of solutions under the assumption that $A$, $A^*$ and $B^*$ satisfy the De Giorgi-Nash-Moser condition. We begin with the following De Giorgi-Nash estimate for solutions to \emph{inhomogeneous} second-order problems.
This estimate is well-known in the case of real coefficients; see, for example, \cite[Theorem~8.24]{GilT01}. Given the fundamental solution of \cite{HofK07} it is straightforward to generalize to complex coefficients.

\begin{thm}
	\label{thm:DGNMinhomogeneous}
	Suppose that $A$, $A^*$ satisfy the De Giorgi-Nash-Moser condition. Let $(n+1)/2<p\leq \infty$ and let $\beta=\min(\alpha,2-(n+1)/p)$.
	
	Then if $\Div A\nabla u=f$ in $B(X_0,2r)$ for some $f\in L^p(B(X_0,2r))$, then
	\begin{align*}
	\abs{u(X)-u(X')}
	&\leq 
	C\frac{\abs{X-X'}^\beta}{r^\beta} r^{2-(n+1)/p}\doublebar{f}_{L^p(B(X_0,2r))}
	\\&\qquad
	+C\frac{\abs{X-X'}^\alpha}{r^\alpha}\left(\fint_{B(X_0,2r)} \abs{u}^2\right)^{1/2}
	\end{align*}
	for all $X$, $X'\in B(X_0,r)$.
\end{thm}
\begin{proof} 
	Let 
	\begin{equation*}v(X)=\int_{B(X_0,2r)} \Gamma_Y^A(X)\,f(Y)\,dY.\end{equation*}
	Let $w(X)=u(X)-v(X)$. By \thmref{fundsoln:properties}, $\Div A\nabla w=0$ in $B(X_0,2r)$ and so because $A$ satisfies the De Giorgi-Nash-Moser condition,
	\begin{align*}
	\abs{w(X)-w(X')}&\leq C \frac{\abs{X-X'}^\alpha}{r^\alpha}
	\left(\fint_{B(X_0,2r)} \abs{w}^2\right)^{1/2}
	\\&\leq C \frac{\abs{X-X'}^\alpha}{r^\alpha}
	\left(\fint_{B(X_0,2r)} \abs{u}^2+\abs{v}^2\right)^{1/2}
	.\end{align*}

	By \eqref{eqn:fundsolnsize}, if $X\in\RR^n$ then
	\begin{align*}
	\abs{v(X)}&\leq \int_{B(X_0,2r)} \abs{\Gamma_Y^A(X)}\,\abs{f(Y)}\,dY
	\leq C r^{2-(n+1)/p} \doublebar{f}_{L^p(B(X_0,2r)},
	\end{align*}
	and by applying \eqref{eqn:fundsolnsize} in $B(X_0,2r)\cap B(X,2\abs{X-X'})$ and \eqref{eqn:fundsolnholder:k} in $B(X_0,2r)\setminus B(X,2\abs{X-X'})$, we see that if $X$, $X'\in\RR^{n+1}$ then
	\begin{align*}
	\abs{v(X)-v(X')}
	&\leq
	C \left(\frac{\abs{X-X'}^\alpha}{r^\alpha}r^{2-(n+1)/p}+
	\abs{X-X'}^{2-(n+1)/p}
	\right)\doublebar{f}_{L^p(B(X_0,2r))}
	.\end{align*}
	Summing these two bounds completes the proof.
\end{proof}

The following corollary follows immediately from Theorems~\ref{thm:Caccioppoli} and~\ref{thm:DGNMinhomogeneous}.
\begin{cor}\label{cor:4holder}
	Suppose that $L_B^* (a\,L_A u)=0$ in $B(X_0,2r)$ in the sense of \dfnref{weaksoln}, where $a$, $A$, $B$ are as in \thmref{Caccioppoli}, and where $A$, $A^*$ and $B^*$ satisfy the De Giorgi-Nash-Moser condition. Then
	\begin{equation}\label{eqn:4holder}
	\abs{u(X)-u(X')}
	\leq C\frac{\abs{X-X'}^\alpha}{r^\alpha} \left(\fint_{B(X_0,2r)} \abs{u}^2\right)^{1/2}\end{equation}
	provided $X$, $X'\in B(X_0,r)$.
\end{cor}

We now prove the following pointwise estimate for solutions~$u$ in terms of their $L^1$ norms. (The bound \eqref{eqn:localbound} is essentially the same estimate in terms of the $L^2$ norm.) This estimate is known for solutions to second-order equations (see, for example, \cite[Theorem~4.1]{HanL97}) and may be proven for solutions to higher order equations using the same techniques.
\begin{cor}\label{cor:reverseholder} Suppose that $u$ is as in \crlref{4holder}. Then
\begin{equation*}
\sup_{B(X_0,r)}\abs{u}\leq C\fint_{B(X_0,2r)}\abs{u}.\end{equation*}
\end{cor}

\begin{proof}
Let $f(\rho)=\doublebar{u}_{L^\infty(B(X_0,\rho))}$ for $0<\rho<2r$. By \crlref{4holder}, if $0<\rho<\rho'<2r$ and if $X\in B(X_0,\rho)$, then
\begin{equation*}\abs{u(X)}
\leq C \left(\fint_{B(X,\rho'-\rho)} \abs{u}^2\right)^{1/2}
\leq \frac{C\doublebar{u}_{L^\infty(B(X_0,\rho'))}^{1/2}} {(\rho'-\rho)^{(n+1)/2}} 
\left(\int_{B(X_0,2r)} \abs{u}\right)^{1/2}\end{equation*}
and so
\begin{align*}
f(\rho)
&\leq 
	\frac{1}{2} f(\rho')+
	\frac{C}{(\rho'-\rho)^{n+1}} 
	\int_{B(X_0,2r)} \abs{u}.
\end{align*}
We eliminate the $f(\rho')$ as follows.
Let $\rho_0=r$ and let $\rho_{k+1} = \rho_k + \frac{1}{2}r(1-\tau) \tau^k$, where ${1/2}<\tau^{n+1}<1$. Then by induction
\begin{equation*}\sup_{B(X_0,r)}\abs{u}
=f(\rho_0)
\leq \frac{1}{2^k} f(\rho_k)
+\sum_{j=0}^{k-1} \frac{C}{(1-\tau)^{n+1}(2\tau^{n+1})^j} 
\fint_{B(X_0,2r)} \abs{u}
.\end{equation*}
Observe that $\lim_{k\to \infty} \rho_k=\frac32 r$ and so $f(\rho_k)$ is bounded uniformly in~$k$. Thus, we may take the limit as $k\to\infty$; this completes the proof.
\end{proof}

We conclude this section with the higher order analogue of \eqref{eqn:2slabs} and some similar results from \cite{AlfAAHK11}, that is, with Caccioppoli-type inequalities valid in horizontal slices.

\begin{lem}\label{lem:slabs} Suppose that $u$, $\partial_{n+1} u$, and $\partial_{n+1}^2 u$ satisfy the Caccioppoli inequality in $\RR^{n+1}_+$, that is, that whenever $B(X,2r)\subset \RR^{n+1}_+$ we have that
\begin{equation*}
\fint_{B(X,r)} \abs{\nabla \partial_{n+1}^k u}^2
\leq \frac{C}{r^2}\fint_{B(X,2r)} \abs{\partial_{n+1}^k u}^2
\qquad\text{for $k=0$, $1$, $2$.}
\end{equation*}

If $\widetilde N_+(\nabla u)\in L^2(\RR^n)$ then 
\begin{equation}
\label{eqn:slabsL2}
\sup_{t>0}\doublebar{\nabla u(\,\cdot\,,t)}_{L^2(\RR^n)} \leq C\doublebar{\widetilde N_+(\nabla u)}_{L^2(\RR^n)}.
\end{equation}
If $t>0$ then
\begin{equation}
\label{eqn:slabsL2Caccioppoli}
\doublebar{\nabla u(\,\cdot\,,t)}_{L^2(\RR^n)} \leq \frac{C}{t}\fint_{t/2}^{2t} \doublebar{u(\,\cdot\,,s)}^2_{L^2(\RR^n)}\,ds
\end{equation}
provided the right-hand side is finite.

Finally, if $0<s<t<2s$, then
\begin{equation}
\label{eqn:slabscts}
\doublebar{\nabla u(\,\cdot\,,t)-\nabla u(\,\cdot\,,s)}_{L^2(\RR^n)}
\leq C\frac{t-s}{s} \biggl(\fint_{s/2}^{3s}
\doublebar{ \nabla u(\,\cdot\,,r)}_{L^2(\RR^n)}^2 \,dr\biggr)^{1/2}
\end{equation}
provided the right-hand side is finite.
\end{lem}
	
	\begin{proof}	
	First, we have that
	\begin{align*} 
	\int_Q \abs{\nabla u(x,t)}^2\,dx
	&\leq
	C\int_Q \biggabs{\nabla u(x,t)-\fint_t^{t+l(Q)/4} \nabla u(x,s)\,ds}^2\,dx
	\\&\qquad+
	\frac{C}{l(Q)}\int_Q \int_t^{t+l(Q)/4} \abs{\nabla u(x,s)}^2\,ds\,dx
	.\end{align*}
	But we may bound $\abs{\nabla u(x,t)-\nabla u(x,s)}$ by $\int \abs{\nabla\partial_r u(x,r)\,dr}$. Applying the Caccioppoli inequality to $\nabla \partial_r u(x,r)$ yields that
	\begin{equation}
	\label{eqn:slabsgradient}
	\int_Q \abs{\nabla u(x,t)}^2\,dx
	\leq 
	\frac{C}{l(Q)} \int_{(3/2)Q} \int_{t-l(Q)/3}^{t+l(Q)/3} \abs{\nabla u(x,t)}^2\,dt\,dx.
	\end{equation}	
	Applying the Caccioppoli inequality to $\nabla u$, this yields that
	\begin{equation}
	\label{eqn:slabsCaccioppoli}
	\int_Q \abs{\nabla u(x,t)}^2\,dx
	\leq 
	\frac{C}{l(Q)^3} \int_{2Q} \int_{t-l(Q)/2}^{t+l(Q)/2} \abs{u(x,t)}^2\,dt\,dx.
	\end{equation}
	Note that \eqref{eqn:slabsCaccioppoli} is valid with $u$ replaced by~$\partial_{n+1} u$.
	
	Dividing $\RR^{n}$ into cubes of side-length $t/2$, we have that \eqref{eqn:slabsL2Caccioppoli} follows from \eqref{eqn:slabsCaccioppoli}.
	Suppose that $t>0$. By definition of $\widetilde N_+$, 
	\begin{equation*}\int_{\RR^n}\int_{t/2}^{2t} \abs{\nabla u(x,s)}^2\,dx\,ds \leq C t \int_{\RR^n} \widetilde N_+(\nabla u)(x)^2\,dx\end{equation*}
	and applying \eqref{eqn:slabsgradient} in cubes of side-length $t$ completes the proof of \eqref{eqn:slabsL2}. 
	
	Finally, let $0<s<t<2s$ and let $Q$ be a cube of side-length~$t-s$. Then
	\begin{align*}
	\int_{Q} \abs{\nabla u(x,t)-\nabla u(x,s)}^2\,dx
	&=
	\int_{Q} \biggabs{\int_s^t\nabla \partial_r u(x,r)\,dr}^2\,dx
	\\&\leq
	(t-s)\int_{Q} \int_s^t\abs{\nabla \partial_r u(x,r)}^2\,dr\,dx.
	\end{align*}
	We have that $0<t-s<s$. 
	Applying \eqref{eqn:slabsCaccioppoli} yields that
	\begin{align*}
	\int_{\RR^n} \abs{\nabla u(x,t)-\nabla u(x,s)}^2\,dx
	&\leq
	C\frac{(t-s)^2}{s^3}\int_{\RR^n} \int_{s/2}^{3s}\abs{\partial_r u(x,r)}^2\,dr\,dx
	\end{align*}
	and so \eqref{eqn:slabscts} is valid.
	\end{proof}


\section{The potentials $\E h$ and $\F h$}
\label{sec:potentials}
We will construct solutions to the fourth-order Dirichlet problem \eqref{eqn:Dirichletprob} using layer potentials. Specifically, our solution $u$ will be given by $u=-\D_A f - \s_A g + \E_{B,a,A} h$ for appropriate functions $f$, $g$ and~$h$. The behavior of the second-order potentials $\D_A f$ and $\s_A g$ is by now well understood (see \secref{secondorder} or the extensive literature on the subject). It remains to investigate $\E=\E_{B,a,A}$. Observe that by the definition \eqref{eqn:E} of~$\E_{B,a,A}$, when convenient we may instead investigate the potential $\F=\F_{B,a,A}$. 

In this section, we will show that $\F h$ (and thus $\E h$) are well-defined in $\RR^{n+1}_+$ for appropriate~$h$ and will establish a few useful preliminary bounds on $\F h$. We will also investigate the behavior of $\E$ across the boundary; that is, we will prove analogues to the jump relations \eqref{eqn:Sjump} and~\eqref{eqn:Scts}. In Sections~\ref{sec:square} and~\ref{sec:nontangential}, we will establish somewhat more delicate bounds on $\F$ (and~$\E$); specifically, our goal in these three sections is to show that $\doublebar{\widetilde N_\pm(\nabla \E h)}_{L^2(\RR^n)}\leq C \doublebar{h}_{L^2(\RR^n)}$. 

In \secref{final}, we will show that the map $h\mapsto \partial_{n+1} \E h\vert_{\partial\RR^{n+1}_\pm}$ is invertible $L^2(\RR^n)\mapsto L^2(\RR^n)$. 
We will need the assumptions that $a$ is real-valued and $A=B$, or that $\doublebar{\im a}_{L^\infty}$ and $\doublebar{A-B}_{L^\infty}$ are small, only in \secref{final}; the bounds of Sections~\ref{sec:potentials}, \ref{sec:square} and~\ref{sec:nontangential} require only that $a$ be accretive and that $A$ and $B^*$ satisfy the single layer potential requirements of \dfnref{goodlayer}. We will conclude this paper by using these boundedness and invertibility results to prove existence of solutions to the fourth-order Dirichlet problem.


We begin by establishing conditions under which $\F h$ exists. Like the single layer potential $\s_A g$, in dimensions $n+1\geq 4$, the integral in the definition of $\F h(x,t)$ converges absolutely whenever $t\neq 0$ and $h\in L^2(\RR^n)$; in dimension $n+1=3$, $\F h$ is only well-defined up to an additive constant if $h\in L^2(\RR^2)$.

More precisely, we have the following.

\begin{lem} \label{lem:Fexists}
Suppose that $a$, $A$ and~$B$ are $t$-independent, $a$ is accretive, $A$ and $A^*$ satisfy the De Giorgi-Nash-Moser condition, and $\s_{B^*}$ satisfies the square-function estimate \eqref{eqn:NSt}.

If $h\in L^p(\RR^n)$ for some $1<p<n$, and if $t\neq 0$, then the integral in the definition \eqref{eqn:F} of $\F h(x,t)$ converges absolutely. Furthermore, 
\begin{align}
\label{eqn:Fdecay}
\abs{\F h(x,t)}
&\leq
C \abs{t}^{1-n/p}\doublebar{h}_{L^p(\RR^n)}.
\end{align}

If $h\in L^p(\RR^n)$ for some $n\leq p<n/(1-\alpha)$, then $\F h$ is well-defined in $\RR^{n+1}_+$ and $\RR^{n+1}_-$ up to an additive constant; that is, if we write
\begin{multline*}
\F h(x,t)-\F h(x',t')=\\
\int_{\RR^n}\int_0^\infty
\bigl(\partial_s \Gamma_{(y,s)}^A(x',t') - \partial_s \Gamma_{(y,s)}^A(x,t)\bigr) \,\frac{1}{a(y)}
\,\partial_s\s_{B^*} h(y,s)\,ds\,dy
\end{multline*}
then the right-hand integral converges absolutely. If $(x',t')\in B((x,t),\abs{t}/4)$, or if $(x',t')=(x,-t)$, then
\begin{align}
\label{eqn:Fdiffdecay}
\abs{\F h(x,t)-\F h(x',t')}
&\leq
C(p) \abs{t}^{1-n/p}\doublebar{h}_{L^p(\RR^n)}.
\end{align}
\end{lem}

\begin{proof}
Recall that
\begin{equation*}
\F h(x,t) 
= -\int_{\RR^n}\int_0^\infty
\partial_s \Gamma_{(y,s)}^A(x,t) \,\frac{1}{a(y)}
\,\partial_s \s_{B^*} h(y,s)\,ds\,dy.
\end{equation*}
Choose some $(x,t)\in \RR^{n+1}$ with $t\neq 0$. Let $B=B((x,t),\abs{t}/2)$, $\Delta=\Delta(x,\abs{t}/2)$. 
By \eqref{eqn:fundsolnsize:k}, $\abs{\partial_t \Gamma_{(y,s)}^A(x,t)}\leq C/\abs{(x,t)-(y,s)}^n$.
If $t>0$ and the aperture $a$ in the definition of nontangential maximal function is large enough, and if $1\leq p\leq\infty$, then
\begin{multline}
\label{eqn:SagainstCZ:local}
\int_B
\frac{C}{\abs{(x,t)-(y,s)}^{n}}
\,\abs{\partial_s \s_{B^*} h(y,s)}\,ds\,dy
\\\leq
Ct \fint_\Delta N_+(\partial_{n+1}\s_{B^*} h)(y)\,dy
\leq Ct^{1-n/p}\doublebar{N_+(\partial_{n+1}\s_{B^*} h)}_{L^p(\RR^n)}.
\end{multline}

Observe that
\begin{multline}
\label{eqn:SagainstCZ}
\int_{\RR^{n+1}_+\setminus B}
\frac{C}{\abs{(x,t)-(y,s)}^{m}}
\,\abs{\partial_s \s_{B^*} h(y,s)}\,ds\,dy
\\
\begin{aligned}
&\leq
\int_{\RR^n} N_+(\partial_{n+1}\s_{B^*} h)(y)
\int_{0}^\infty
\frac{C}{\abs{x-y}^m+(s+\abs{t})^m}
\,ds\,dy
.\end{aligned}
\end{multline}
If $m=n$ and $1\leq p<n$, then these integrals converge and are at most \begin{equation*}C(p) \abs{t}^{1-n/p}\doublebar{N_+(\partial_{n+1}\s_{B^*} h)}_{L^p(\RR^n)}.\end{equation*}
By \eqref{eqn:NSt} we may bound $\doublebar{N_+(\partial_{n+1}\s_{B^*} h)}_{L^p(\RR^n)}$ and \eqref{eqn:Fdecay} is proven.

To establish \eqref{eqn:Fdiffdecay}, recall that by \eqref{eqn:fundsolnsize:k} and \eqref{eqn:fundsolnholder:k}, if  $\abs{(x',t')-(x,t)}<\abs{t}/4$, or if $t>0$ and $(x',t')=(x,-t)$, then for all $(y,s)\in\RR^{n+1}_+\setminus B$ we have that
\begin{equation*}\abs{\partial_t \Gamma_{(y,s)}^A(x,t) - \partial_t \Gamma_{(y,s)}^A(x',t')}
\leq \frac{C\abs{t}^\alpha}{\abs{(x,t)-(y,s)}^{n+\alpha}}
\end{equation*}
Letting $m=n+\alpha$, we see that if $1<p<n/(1-\alpha)$ then the right-hand side of \eqref{eqn:SagainstCZ} converges and is at most $C(p)\abs{t}^{1-n/p}\doublebar{N_+(\partial_{n+1}\s_{B^*} h)}_{L^p(\RR^n)}$. By \eqref{eqn:NSt}, and since \eqref{eqn:SagainstCZ:local} is valid for all $1\leq p\leq\infty$, this completes the proof of \eqref{eqn:Fdiffdecay}.
\end{proof}

Next, we consider $\nabla \F h$ and $L_A \F h$.
\begin{lem}
\label{lem:Fgradientexists} 
Let $a$, $A$ and $B$ be as in \lemref{Fexists}.
If $h\in L^p(\RR^n)$ for some $1<p<n/(1-\alpha)$, then
$\nabla \F h\in L^2_{loc}(\RR^{n+1}_\pm)$, and in particular, $\Div A\nabla\F h$ is a well-defined element of $W^{2}_{-1,loc}(\RR^{n+1}_\pm)$.

Furthermore,
\begin{alignat}{2}
\label{eqn:LFupper}
-a\,\Div A\nabla \F h &= \partial_{n+1}^2\s_{B^*} h
&\quad&\text{in }\RR^{n+1}_+,
\\
\label{eqn:LFlower}
\Div A\nabla \F h &= 0
&&\text{in }\RR^{n+1}_-
\end{alignat}
in the weak sense.
\end{lem}
As an immediate corollary, $\E h$ is well-defined and also lies in $W^2_{1,loc}(\RR^{n+1}_\pm)$, and $\Div A\nabla\E h=\Div A \nabla \F h$ in $\RR^{n+1}_\pm$.

\begin{proof}
Fix some $(x_0,t_0)\in\RR^{n+1}$ with $t_0\neq 0$, and let $B_r=B((x_0,t_0),r)$.
Let $\eta$ be a smooth cutoff function, supported in $B_{\abs{t_0}/2}$ and identically equal to 1 in $B_{\abs{t_0}/4}$, with $0\leq\eta\leq 1$, $\abs{\nabla\eta}\leq C/\abs{t_0}$. For all $(x,t)\in B_{\abs{t_0}/8}$, we have that by definition of $\F h$ and by \eqref{eqn:fundsoln:vert},
\begin{multline*}
\F h(x,t)-\F h(x_0,t_0)
\\\begin{aligned}&=
	\int_{\RR^{n+1}_+}
	\Gamma_{(y,s)}^A(x,t)\,\frac{1}{a(y)}
	\,\partial_s\bigl(\eta(y,s) \partial_s \s_{B^*} h(y,s)\bigr)\,ds\,dy
	\\&\qquad
	-\int_{\RR^{n+1}_+}
	\Gamma_{(y,s)}^A(x_0,t_0)\,\frac{1}{a(y)}
	\,\partial_s\bigl(\eta(y,s) \partial_s \s_{B^*} h(y,s)\bigr)\,ds\,dy
	\\&\qquad
	-\int_{\RR^{n+1}_+}
	\bigl(\partial_s \Gamma_{(y,s)}^A(x,t) -\partial_s\Gamma_{(y,s)}^A(x_0,t_0)\bigr)
	\,\frac{1-\eta(y,s)}{a(y)}
	\,\partial_s \s_{B^*} h(y,s)\,ds\,dy
\\&= I(x,t)-I(x_0,t_0)+II(x,t)
.\end{aligned}\end{multline*}

If $t_0<0$ then $I\equiv 0$. Otherwise, by \eqref{eqn:Shdecay}, the function $\partial_{n+1}(\eta\,\partial_{n+1}\s_{B^*} h)$ is bounded and compactly supported, and so by \thmref{fundsoln:properties}, $I\in Y^{1,2}(\RR^{n+1})\subset W^2_{1,loc}(\RR^{n+1})$. Furthermore, $-\Div A\nabla I = (1/a)\partial_{n+1}(\eta\,\partial_{n+1}\s_{B^*} h)$, and so \begin{equation*}-a \,\Div A\nabla I = \partial_{n+1}^2 \s_{B^*} h\quad\text{in }B_{\abs{t_0}/8}.\end{equation*}

We must show that $\nabla II \in L^2(B_{\abs{t_0}/8})$ and that $\Div A\nabla II=0$ in $B_{\abs{t_0}/8}$. By \eqref{eqn:fundsolnholder:k} and \lemref{Caccioppoli2}, we have that if $(y,s)\notin B_{\abs{t_0}/4}$ then
\begin{equation*} \int_{B_{\abs{t_0}/8}}\abs{\nabla_{x,t}\partial_s \Gamma_{(y,s)}^A(x,t)}\,dx\,dt
\leq \frac{C\abs{t_0}^{n+\alpha}}{\abs{(x_0,t_0)-(y,s)}^{n+\alpha}}. \end{equation*}
Thus by \eqref{eqn:SagainstCZ}, if $1<p<n/(1-\alpha)$ then
\begin{multline*}
\int_{B_{\abs{t_0}/8}} \int_{\RR^{n+1}_+}
	\abs{\nabla_{x,t} \partial_s \Gamma_{(y,s)}^A(x,t)} \,\frac{1-\eta(y,s)}{\abs{a(y)}}
	\,\abs{\partial_s \s_{B^*} h(y,s)}\,ds\,dy\,dx\,dt
\\\leq
C(p) \abs{t_0}^{n-n/p}\doublebar{N_+(\partial_{n+1}\s_{B^*} h)}_{L^p(\RR^n)}.
\end{multline*}
Thus by Fubini's theorem,
\begin{equation}
\label{eqn:nablaII}
\nabla II(x,t)=\int_{\RR^{n+1}_+}
	\nabla_{x,t}\partial_s \Gamma_{(y,s)}^A(x,t) \,\frac{1-\eta(y,s)}{a(y)}
	\,\partial_s \s_{B^*} h(y,s)\,ds\,dy.\end{equation}
If $\varphi$ is a test function supported in $B_{\abs{t_0}/8}$, then again by \eqref{eqn:fundsolnholder:k}, \eqref{eqn:SagainstCZ} and the Caccioppoli inequality,
\begin{multline*}
\biggabs{\int_{\RR^{n+1}} \varphi(x,t) \, \nabla II(x,t)\,dx\,dt}
\\\leq C\doublebar{\varphi}_{L^2(\RR^{n+1})}
\abs{t_0}^{1/2+n/2-n/p}\doublebar{N_+(\partial_{n+1}\s_{B^*} h)}_{L^p(\RR^n)}
\end{multline*}
and so $\nabla II(x,t)\in L^2(B_{\abs{t_0}/8})$. 

Finally, by the weak definition \eqref{eqn:2weaksoln} of $\Div A\nabla$ and by the formula \eqref{eqn:nablaII} for $\nabla II$, we have that $\Div A\nabla II=0$ in $B_{\abs{t_0}/8}$, as desired.
\end{proof}

We conclude this section by proving the continuity of $\nabla \E h$ across the boundary. This property is analogous to the continuity relations \eqref{eqn:Scts} for the single layer potential, and is the reason we will eventually prefer the operator $\E$ to~$\F$.
\begin{lem}
\label{lem:gradEcts} 
Suppose that $a$, $A$ and $B$ are $t$-independent, $a$ is accretive, $A$ satisfies the square-function estimate \eqref{eqn:NnablaS}, and $B^*$ satisfies the single layer potential requirements of \dfnref{goodlayer}.
Then there is a dense subset $S\subset L^2(\RR^n)$ such that if $h\in S$, then
\begin{equation*}\lim_{t\to 0^+} \doublebar{\nabla\E h(\,\cdot\,,t)-\nabla \E h(\,\cdot\,,-t)}_{L^2(\RR^n)}=0.\end{equation*}
\end{lem}

\begin{proof} 
We define the set $S\subset L^2(\RR^n)$ as follows.
By assumption, $\s_{B^*}^{\perp,+}$ is invertible. Let $S=(\s_{B^*}^{\perp,+})^{-1}(S')$, where $h\in S'$ if $h(x)=u(x,t)$
for some $t>0$ and some $u$ with $\Div B^*\nabla u=0$ in $\RR^{n+1}_+$ and $N_+ u\in L^2(\RR^n)$.

We first show that $S$ is dense. It suffices to show that $S'$ is dense. Choose some $f\in L^2(\RR^n)$. By well-posedness of $(D)^{B^*}_2$, there is some $u$ with $\Div B^*\nabla u=0$ in $\RR^{n+1}_+$ and $u=f$ on $\partial\RR^{n+1}_+$. Define $f_k\in L^2(\RR^n)$ by $f_k(x) = u(x,1/k)$; then $f_k\in S'$.
By \thmref{2L2limits}, $f_k\to f$ in $L^2(\RR^n)$, and so $S'$ is dense in $L^2(\RR^n)$.

Suppose that $h\in S$. Then $\s_{B^*}^{\perp,+} h(x)=u(x,\tau)$ for some $\tau>0$ and some solution~$u$. Let $v(x,t)=\partial_t u(x,t+\tau)$. By \eqref{eqn:NSt}, and by uniqueness of solutions to $(D)^{B^*}_2$, we have that $\partial_t \s h=v$ in $\RR^{n+1}_+$. But by the Caccioppoli inequality and the De Giorgi-Nash-Moser estimates, we have that $N_+(\partial_{n+1} v)(x)\leq \frac{C}{\tau} N_+ u(x)$ (possibly at the cost of increasing the apertures of the nontangential cones). 

So if $h\in S$, then
\begin{equation*}\widetilde N_+(\partial_{n+1}^2\s_{B^*} h)\in L^2(\RR^n).\end{equation*}

Recall that by \eqref{eqn:E} and \lemref{Fexists} if $h\in L^2(\RR^n)$ and $n\geq 3$ then
\begin{equation*}
\E h(x,t)
=
	-\int_0^\infty \int_{\RR^n} \partial_s \Gamma_{(y,s)}^A(x,t)
	\frac{1}{a(y)} \partial_s \s_{B^*} h(y,s)\,dy\,ds
	- \s_A \biggl(\frac{1}{a}\s_{B^*}^{\perp,+} h\biggr)(x,t)
.\end{equation*}
If $n=2$ then we must instead work with $\E h(x,t)-\E h(x,-t)$.

If $h\in S$, so that $N_+(\partial_{n+1}^2 \s_{B^*} h)\in L^2(\RR^n)$, then we may integrate by parts in the region $0<s<\tau$ for any fixed~$\tau$. Observe that the boundary term at $s=0$ precisely cancels the term $\s_A((1/a)\s_{B^*}^{\perp,+}h)$.
We conclude that
\begin{align*}
 \E h(x,t)
&=
	\int_0^{\tau} \int_{\RR^n} \Gamma_{(y,s)}^A(x,t)
	\frac{1}{a(y)} \partial_s^2 \s_{B^*} h(y,s)\,dy\,ds
\\&\qquad
	-\int_{\RR^n}  \Gamma_{(y,\tau)}^A(x,t)
	\frac{1}{a(y)} \partial_{n+1} \s_{B^*} h(y,\tau)\,dy
\\&\qquad
	-\int_{\tau}^\infty \int_{\RR^n} \partial_s \Gamma_{(y,s)}^A(x,t)
	\frac{1}{a(y)} \partial_s \s_{B^*} h(y,s)\,dy\,ds
.\end{align*}
Applying \eqref{eqn:fundsoln:vert} and the definition of $\s_A$, we see that
\begin{align*}
 \E h(x,t)
&=
	\int_0^{\tau} 
	\s_A\biggl(\frac{1}{a} \partial_s^2 \s_{B^*} h(s)\biggr)(x,t-s)
	\,ds
	-\s_A\biggl(\frac{1}{a} \partial_\tau \s_{B^*} h(\tau)\biggr)(x,t-\tau)
\\&\qquad
	-\int_{\tau}^\infty 
	\partial_s \s_A\biggl(\frac{1}{a} \partial_s \s_{B^*} h(s)\biggr)(x,t-s)
	\,ds
.\end{align*}
Here we have adopted the notation that $\partial_s\s_{B^*} h(s)(y)=\partial_s\s_{B^*} h(y,s)$.

Thus $\nabla \E h(x,t)-\nabla \E h(x,-t)$ has three terms.
Suppose that $s\in\RR$, $s\neq 0$.
By \eqref{eqn:NSt}, $\doublebar{\partial_{n+1} \s_{B^*} h(s)}_{L^2(\RR^n)}\leq C\doublebar{h}_{L^2(\RR^n)}$.
Suppose that $g\in L^2(\RR^n)$.
By \eqref{eqn:slabsL2} and by \eqref{eqn:NnablaS},
\begin{equation*}\doublebar{\nabla \s_A g(s)}_{L^2(\RR^n)} \leq C \doublebar{\widetilde N_\pm(\nabla \s_A g)}_{L^2(\RR^n)}
\leq C\doublebar{g}_{L^2(\RR^n)}.\end{equation*}
Furthermore, if $0<t<\abs{s}/2$, then by \eqref{eqn:slabscts}, \eqref{eqn:slabsL2} and \eqref{eqn:NnablaS},
\begin{equation*}\doublebar{\nabla \s_A g(s+t)-\nabla \s_A g(s-t)}_{L^2(\RR^n)} \leq C \frac{t}{\abs{s}}\doublebar{g}_{L^2(\RR^n)}.\end{equation*}
Finally by \eqref{eqn:slabscts}, \eqref{eqn:slabsL2}, \eqref{eqn:NnablaS} and applying the Caccioppoli inequality to $\partial_{n+1}\s_A$,
\begin{equation*}\doublebar{\nabla \partial_s\s_A g(s+t)-\nabla \partial_s\s_A g(s-t)}_{L^2(\RR^n)} \leq C \frac{t}{\abs{s}^2}\doublebar{h}_{L^2(\RR^n)}.\end{equation*}

Thus, if $t<\tau$, then
\begin{align*}
\doublebar{\nabla \E h(t) - \nabla \E h(-t)}_{L^2(\RR^n)}
&\leq 
	C\int_0^{\tau} 
	\doublebar{N_+(\partial_s^2 \s_{B^*} h)}_{L^2(\RR^n)}\,ds
\\&\qquad
	+C\frac{t}{\tau} \doublebar{h}_{L^2(\RR^n)}
	+C\int_{\tau}^\infty 	
	\frac{t}{s^2} \doublebar{h}_{L^2(\RR^n)}\,ds
.\end{align*}
Choosing $\tau=\sqrt{t}$ and recalling that $h\in S$, we see that the right-hand side goes to zero as $t\to 0^+$, as desired.
\end{proof}

\section{A square-function bound}
\label{sec:square}
Recall that we intend to construct solutions to the Dirichlet problem \eqref{eqn:Dirichletprob} by letting $u=u_+ +\E h$ for some appropriately chosen function $h\in L^2(\RR^n)$. To prove Theorems~\ref{thm:Dirichletexists} and~\ref{thm:perturb} we must have that the norms $\doublebar{\widetilde N_+(\nabla \E h)}_{L^2(\RR^n)}$ and $\triplebar{t\,\nabla\partial_t\E h}_+$ are appropriately bounded. 

In this section, we will prove a preliminary square-function estimate; we will prove the full estimate $\doublebar{\widetilde N_+(\nabla \E h)}_{L^2(\RR^n)}+\triplebar{t\,\nabla\partial_t\E h}_+\leq C\doublebar{h}_{L^2(\RR^n)}$ in the next section.

\begin{thm} \label{thm:squarebound}
Suppose that $a$, $A$ and~$B$ are $t$-independent, $a$ is accretive, $A$, $A^*$, $B$ and $B^*$ satisfy the De Giorgi-Nash-Moser condition, and $\s_A$, $\s_{B^*}$ satisfy the square-function estimate \eqref{eqn:Svertsquare}.

Then for all $h\in L^2(\RR^n)$, we have the bound
\begin{equation*}\int_{-\infty}^\infty \int_{\RR^n} \abs{\partial_t^2 \F h(x,t)}^2 \,\abs{t}\,dx\,dt
\leq C \doublebar{h}_{L^2(\RR^n)}^2.\end{equation*}
\end{thm}

In the remainder of this section, let $u=\partial_{n+1} \s_{B^*} h$. Observe that
\begin{align*}
\partial_t^2 \F h(x,t)
&=
	\partial_t^2\int_{\RR^{n+1}} \partial_t\Gamma_{(x,t)}^{A^T}(y,s)\,\frac{1}{a(y)}\,u(y,s)\,dy\,ds
	.\end{align*}
By \eqref{eqn:NnablaS} and \eqref{eqn:Ssquare} it suffices to prove that
\begin{equation}
\label{eqn:triple}
\triplebar{t\,\partial_t^2 \F h}_\pm
\leq
C\doublebar{N_+ u}_{L^2(\RR^n)} + 
C\triplebar{t\,\nabla u}_+
.\end{equation}
This theorem is the technical core of the paper. The proof is inspired by the $T(1)$ theorems of \cite{DavJ84} and \cite{Sem90}. We may think of $\partial_t^2 \F h(x,t)$ as $T u(x,t)$ for a singular integral operator $T$ with kernel $(1/a(y))\,\partial_t^3\Gamma_{(x,t)}^{A^T}(y,s)$. The $T(1)$ theorem of Semmes was proven by analyzing $Tu-T(1) P_t u$ and $T(1)P_t u$ for an averaging operator~$P_t$. We will use the same argument; our operator $P_t$ will be averages over dyadic cubes. Our bound on $T(1)P_t u$ will follow from well-known Carleson-measure properties of solutions to second-order equations. To bound $Tu-T(1) P_t u$ in terms of $\nabla u$, we develop an argument ultimately allowing us to exploit the Poincar\'e inequality. It bears  some resemblance to the arguments of \cite{AlfAAHK11} (see, in particular, Lemma~3.5(i)), but the particular singular integral operator at hand is different form those in \cite{AlfAAHK11}, and new ideas are required.

For ease of notation we will prove \eqref{eqn:triple} only in the upper half-space (that is, only for $\triplebar{t\,\partial_t^2 \F h}_+$); the argument in the lower half-space is similar and simpler.

Let $W$ be the grid of dyadic Whitney cubes in $\RR^{n+1}_+$. That is, 
\begin{equation*}W=\{Q=\widetilde Q\times[\ell(\widetilde Q), 2\ell(\widetilde Q)):\widetilde Q\subset\RR^n\text{ is a dyadic cube}\}.\end{equation*}
Then $\RR^{n+1}_+=\cup_{Q\in W} Q$, and any two distinct cubes $Q$, $R\in W$ have disjoint interiors. For any $Q\in W$ we let 
$F(Q)=\{R\in W:\dist(R,Q)>0\}$ be the set of cubes a positive distance from $Q$ and let $N(Q)=\cup_{R\in W\setminus F(Q)} R$ be the union of cubes adjacent to~$Q$.

Observe that
\begin{align*}
\triplebar{t\,\partial_t^2 \F h}_+
&=
	\sum_{Q\in W} \int_Q \biggabs{
	\partial_t^2\int_{\RR^{n+1}} \partial_t\Gamma_{(x,t)}^{A^T}(y,s)\,\frac{1}{a(y)}\,u(y,s)\,dy\,ds
	}^2\,t\,dx\,dt
	.\end{align*}
If $R\in W$, let $u_R=\fint_R u$. Define the four quantities
\begin{align*}
{I}&= \sum_{Q\in W}
	\int_{Q} \biggl(\sum_{R\in F(Q)} \int_{R} 
	\Bigl\lvert
	\partial_t^3\Gamma_{(x,t)}^{A^T}(y,s) \,\frac{1}{a(y)}\,\bigl(u(y,s)-u_R\bigr)
	\Bigr\rvert\,dy\,ds\biggr)^2 t\,dx\,dt
,\\\displaybreak[0]
{II} &=\sum_{Q\in W}
	\int_{Q} \biggl(\sum_{R\in F(Q)} \abs{u_R-u_Q}\int_{R} \Bigl\lvert\partial_t^3\Gamma_{(x,t)}^{A^T}(y,s)\,\frac{1}{a(y)}
	\Bigr\rvert\,dy\,ds\biggr)^2 t\,dx\,dt
,\\\displaybreak[0]
{III} &=\sum_{Q\in W}
	\int_{Q} \biggabs{\partial_t^2\int_{N(Q)} \partial_t\Gamma_{(x,t)}^{A^T}(y,s)\,\frac{1}{a(y)}\,(u(y,s)-u_Q)\,dy\,ds}^2 t\,dx\,dt
,\\\displaybreak[0]
{IV} &=\sum_{Q\in W}\abs{u_Q}^2
	\int_{Q} \biggl\lvert
	\partial_t^2\int_{N(Q)} \partial_t\Gamma_{(x,t)}^{A^T}(y,s)\,\frac{1}{a(y)}\,dy\,ds
	\\&\quad\phantom{=\sum_{Q\in W}\abs{u_Q}^2
	\int_{Q} \biggl\lvert}+
	\int_{\RR^{n+1}\setminus N(Q)} \partial_t^3\Gamma_{(x,t)}^{A^T}(y,s)\,\frac{1}{a(y)}\,dy\,ds
	\biggr\rvert^2 t\,dx\,dt
.\end{align*}
We have that
\begin{equation*}\triplebar{t\,\partial_t^2 \F h}_+\leq C(I+II+III+IV).\end{equation*}

We will bound each of the terms $I$, $II$, $III$ and~$IV$. We begin with term~${IV}$.

\begin{lem} \label{lem:termIV}
If $a$ and $A$ are $t$-independent, $A$ and $A^*$ satisfy the De Giorgi-Nash-Moser condition and $\s_A$ satisfies the square-function bound \eqref{eqn:Ssquare}, then
\begin{equation*}{IV}\leq C \doublebar{N_+ u}_{L^2(\RR^n)}^2 \doublebar{1/a}_{L^\infty(\RR^n)}^2.\end{equation*}
\end{lem}
\begin{proof}
Recall that
\begin{align*}
{IV} &=\sum_{Q\in W}\abs{u_Q}^2
	\int_{Q} \biggl\lvert
	\partial_t^2\int_{N(Q)} \partial_t\Gamma_{(x,t)}^{A^T}(y,s)\,\frac{1}{a(y)}\,dy\,ds
	\\&\quad\phantom{=\sum_{Q\in W}\abs{u_Q}^2
	\int_{Q} \biggl\lvert}+
	\int_{\RR^{n+1}\setminus N(Q)} \partial_t^3\Gamma_{(x,t)}^{A^T}(y,s)\,\frac{1}{a(y)}\,dy\,ds
	\biggr\rvert^2 t\,dx\,dt
.\end{align*}	
By the decay estimate \eqref{eqn:fundsolnsize:k}, both of the innermost integrals converge absolutely.
By \eqref{eqn:fundsoln:vert}, $\partial_t\Gamma_{(x,t)}^{A^T}(y,s)=-\partial_s\Gamma_{(x,t)}^{A^T}(y,s)$, and so
\begin{multline*}
\partial_t^2\int_{N(Q)} \partial_t\Gamma_{(x,t)}^{A^T}(y,s)\,\frac{1}{a(y)}\,dy\,ds
	+\int_{\RR^{n+1}\setminus N(Q)} \partial_t^3\Gamma_{(x,t)}^{A^T}(y,s)\,\frac{1}{a(y)}\,dy\,ds
\\=\int_{\RR^n} \partial_t^2 \Gamma_{(x,t)}^{A^T}(y,0)\,\frac{1}{a(y)}\,dy
=\partial_t^2\s_A(1/a)(x,t).
\end{multline*}	

We claim that because $1/a\in L^\infty(\RR^n)$, we have that 
\begin{equation*}
d\mu(x,t) = \abs{\partial_t^2 \s_A\left(1/a\right)(x,t)}^2\, t\,dx\,dt
\end{equation*} 
is a Carleson measure with Carleson norm $\doublebar{\mu}_{\mathcal{C}}$ at most $C\doublebar{1/a}_{L^\infty(\RR^n)}^2$. This follows from the square-function bound \eqref{eqn:Ssquare} and the decay estimate \eqref{eqn:fundsolnsize:k} by a simple argument due to Fefferman and Stein (see the proof of Theorem~3 in \cite{FefS72}).
Let $\widetilde R\subset\RR^n$ be a cube; then
\begin{multline*}
\int_0^{\ell(\widetilde R)}\int_{\widetilde R} \abs{\partial_t^2 \s_A\left(1/a\right)(x,t)}^2\, t\,dx\,dt
\\\begin{aligned}&\leq
2\int_0^{\ell(\widetilde R)}\int_{\widetilde R} \abs{\partial_t^2 \s_A\left(\1_{2\widetilde R}(1/a)\right)(x,t)}^2\, t\,dx\,dt
\\&\qquad+
2\int_0^{\ell(\widetilde R)}\int_{\widetilde R} \abs{\partial_t^2 \s_A\left((1-\1_{2\widetilde R})(1/a)\right)(x,t)}^2\, t\,dx\,dt.
\end{aligned}\end{multline*}
The first integral is at most $C\doublebar{1/a}_{L^2(2\widetilde R)}^2\leq C\abs{\widetilde R}\,\doublebar{1/a}_{L^\infty(\RR^n)}^2$ by \eqref{eqn:Ssquare}, while the second is at most $C\abs{\widetilde R}\,\doublebar{1/a}_{L^\infty(\RR^n)}^2$ by \eqref{eqn:Shdecay}.

Let $u_W(x,t)=u_Q$ whenever $(x,t)\in Q$. Observe that $N_+u_W(x)\leq N_+u(x)$ (possibly at a cost of increasing the aperture of nontangential cones). Then
\begin{equation*}{IV}=\int_{\RR^{n+1}_+} \abs{u_W(x,t)}^2 \,\abs{\partial_t^2 \s_A(1/a)(x,t)}^2 \,t\,dx\,dt.\end{equation*}	
Applying duality between Carleson measures and nontangentially bounded functions, we see that
\begin{equation*}{IV}\leq \doublebar{N_+((u_W)^2)}_{L^1(\RR^n)} \doublebar{\mu}_{\mathcal{C}}
\leq C \doublebar{N_+ u}_{L^2(\RR^n)}^2 \doublebar{1/a}_{L^\infty(\RR^n)}^2\end{equation*}
as desired.
\end{proof}

Next, we bound the term ${I}$.

\begin{lem} 
\label{lem:termI}
If $a$ and $A$ are $t$-independent, $a$ is accretive, and $A$ and $A^*$ satisfy the De Giorgi-Nash-Moser condition, then
\begin{equation*}{I}\leq C\int_{\RR^{n+1}_+} \abs{\nabla u(y,s)}^2\,s\,dy\,ds.\end{equation*}
\end{lem}

\begin{proof}
Recall that
\begin{align*}
{I} &=
	\sum_{Q\in W}
	\int_{Q} \biggl(\sum_{R\in F(Q)} \int_{R} 
	\Bigl\lvert
	\partial_t^3\Gamma_{(x,t)}^{A^T}(y,s) \,\frac{1}{a(y)}\,\bigl(u(y,s)-u_R\bigr)
	\Bigr\rvert\,dy\,ds\biggr)^2 t\,dx\,dt
.\end{align*}
By the Poincar\'e inequality,
\begin{equation*}\int_{R} \abs{u(y,s)-u_R}\,dy\,ds\leq C\ell(R)\int_R \abs{\nabla u(y,s)}\,dy\,ds.\end{equation*}
By \eqref{eqn:fundsolnsize:k}, and because $\ell(R)\leq s\leq 2\ell(R)$ for any $(y,s)\in R$, we have that
\begin{align*}
{I} &\leq	
	\sum_{Q\in W}
	\int_{Q} \biggl(\sum_{R\in F(Q)} 
	\frac{C\ell(R)\doublebar{1/a}_{L^\infty(\RR^n)}}{\dist(Q,R)^{n+2}}
	\int_{R} \abs{\nabla u(y,s)}\,dy\,ds\biggr)^2 t\,dx\,dt
\\&\leq
	C\doublebar{1/a}_{L^\infty(\RR^n)}^2\int_{\RR^{n+1}_+} \biggl(
	\int_{\RR^{n+1}_+} \frac{s}{(s+t+\abs{x-y})^{n+2}}\abs{\nabla u(y,s)}\,dy\,ds\biggr)^2 t\,dx\,dt
.\end{align*}
But by H\"older's inequality,
\begin{multline}
\label{eqn:squarenorm}
\int_{\RR^{n+1}_+} \biggl(
	\int_{\RR^{n+1}_+} \frac{s}{(s+t+\abs{x-y})^{n+2}}\abs{\nabla u(y,s)}\,dy\,ds\biggr)^2 t\,dx\,dt
\\\begin{aligned}
&\leq
	\int_{\RR^{n+1}_+} 
	\int_{\RR^{n+1}_+} \frac{s^2\,\abs{\nabla u(y,s)}^2\,dy\,ds}{(s+t+\abs{x-y})^{n+2}}
	\int_{\RR^{n+1}_+}
	\frac{ \,dy\,ds}{(s+t+\abs{x-y})^{n+2}}
	\,t\,dx\,dt
\\&=
	C\int_{\RR^{n+1}_+} 
	\abs{\nabla u(y,s)}^2s\,dy\,ds.
\end{aligned}\end{multline}
This completes the proof.
\end{proof}

We may bound ${III}$ similarly.

\begin{lem} 
\label{lem:termIII}
Suppose that $a$ and $A$ are as in \lemref{termI}. Suppose in addition that $B^*$ is $t$-independent and satisfies the De Giorgi-Nash-Moser condition, and $\Div B^*\nabla u=0$ in $\RR^{n+1}_+$. Then
\begin{equation*}
{III}
\leq
	C\int_{\RR^{n+1}_+} \abs{\nabla u(y,s)}^2\,s\,dy\,ds.\end{equation*}
\end{lem}

\begin{proof}
Recall that
\begin{equation*}
{III}=\sum_{Q\in W}
	\int_{Q} \biggabs{\partial_t^2\int_{N(Q)} \partial_t\Gamma_{(x,t)}^{A^T}(y,s)\,\frac{1}{a(y)}\,(u(y,s)-u_Q)\,dy\,ds}^2 t\,dx\,dt.\end{equation*}

Observe that $2Q\subset N(Q)\subset 5Q$. Let 
\begin{equation*}w(x,t)= \int_{N(Q)}\Gamma_{(x,t)}^{A^T}(y,s)\,\frac{1}{a(y)}\,(u(y,s)-u_Q)\,dy\,ds.\end{equation*}
By \eqref{eqn:fundsolnsize}, $w$ satisfies
$\abs{w(x,t)}\leq C\ell(Q)^2 \doublebar{1/a}_{L^\infty(\RR^n)} \doublebar{u-u_Q}_{L^\infty(N(Q))}$.

By \thmref{fundsoln:properties}, $a\,L_A w=(u-u_Q)$ in~$N(Q)$.
If $\Div B^*\nabla u = 0$ in $N(Q)$, then $L_B^*(a\,L_A w)=0$ in~$N(Q)$. Thus, we may use \thmref{Caccioppoli} and \crlref{4holder} twice to show that
\begin{equation*}\sup_{(x,t)\in Q}\abs{\partial_t^3 w(x,t)}
\leq \frac{C}{\ell(Q)} \doublebar{1/a}_{L^\infty(\RR^n)} \doublebar{u-u_Q}_{L^\infty(N(Q))}
\end{equation*}
and so
\begin{equation*}{III}
\leq
\sum_{Q\in W}
	\int_{Q} \biggl(\frac{C}{\ell(Q)} \doublebar{1/a}_{L^\infty(\RR^n)} \doublebar{u-u_Q}_{L^\infty(N(Q))}\biggr)^2 t\,dx\,dt.\end{equation*}
Let 
\begin{equation*}\pi(Q)=\{x:(x,t)\in Q\text{ for some }t>0\}\end{equation*}
be the projection of $Q$ onto~$\RR^n$. Notice that $\pi(Q)$ is also a cube. Let $U(Q)=6\pi(Q) \times (\ell(Q)/4, 5\ell(Q))$. Observe that if $X\in N(Q)$, then $B(X, \ell(Q)/4)\subset U(Q)$.
If $\Div B^*\nabla u=0$ in~$U(Q)$, then by the De Giorgi-Nash-Moser condition and the Poincar\'e inequality,
\begin{equation*}\sup_{N(Q)} \abs{u-u_Q} 
\leq C\biggl(\fint_{U(Q)} \abs{u-u_Q}^2\biggr)^{1/2}
\leq C\ell(Q)\biggl(\fint_{U(Q)} \abs{\nabla u}^2\biggr)^{1/2}
.\end{equation*}
So
\begin{align*}{III}
&\leq
	C\doublebar{1/a}_{L^\infty(\RR^n)}^2
	\sum_{Q\in W}
	\int_{Q} \fint_{U(Q)} \abs{\nabla u(y,s)}^2\,dy\,ds\, t\,dx\,dt
.\end{align*}
Because each $(y,s)\in\RR^{n+1}$ lies in $U(Q)$ for at most $C$ cubes $Q\in W$, this implies that
\begin{align*}{III}
&\leq
	C\doublebar{1/a}_{L^\infty(\RR^n)}^2
	\int_{\RR^{n+1}_+} \abs{\nabla u(y,s)}^2\,s\,dy\,ds
\end{align*}
as desired.\end{proof}

Finally, we come to the term ${II}$.

\begin{lem} \label{lem:termII}
If $a$ is accretive, and if $A$ and $A^*$ are $t$-independent and satisfy the De Giorgi-Nash-Moser condition, then
\begin{equation}
\label{eqn:termII}
{II}\leq C\int_{\RR^{n+1}_+} \biggl(
	\int_{\RR^{n+1}_+} \frac{s}{(s+t+\abs{x-y})^{n+2}}\abs{\nabla u(y,s)}\,dy\,ds\biggr)^2 t\,dx\,dt.
\end{equation}
\end{lem}
Note that the estimate \eqref{eqn:termII} together with  \eqref{eqn:squarenorm} above imply that
\begin{equation*}{II} \leq C\int_{\RR^{n+1}_+} \abs{\nabla u(y,s)}^2s\,dy\,ds.\end{equation*}

\begin{proof}

Recall that 
\begin{equation*}
{II}=\sum_{Q\in W}
	\int_{Q} \biggl(\sum_{R\in F(Q)} \abs{u_R-u_Q}\int_{R} \Bigl\lvert\partial_t^3\Gamma_{(x,t)}^{A^T}(y,s)\,\frac{1}{a(y)}
	\Bigr\rvert\,dy\,ds\biggr)^2 t\,dx\,dt.\end{equation*}
	
To analyze the inner sum, we establish some notation. As in the proof of \lemref{termIII}, if $R\in W$ is a cube, we let $\pi(R)$ be the projection of $R$ onto~$\RR^n$. 
We let $P(R)$ denote the Whitney cube directly above~$R$, so $\pi(P(R))$ is the dyadic parent of $\pi(R)$.
Let $\delta(Q,R)=\ell(Q)+\ell(R)+\dist(\pi(Q),\pi(R))$. If $(x,t)\in Q$ and $(y,s)\in R$ for some $R\in F(Q)$, then
\begin{equation*}\abs{(y,s)-(x,t)} \approx \abs{x-y} + \max(s,t)\approx\delta(Q,R).\end{equation*}
Here $U\approx V$ if $U\leq CV$ and $V\leq CU$.
Applying \eqref{eqn:fundsolnsize:k}, we see that
\begin{align*}
{II} &\leq 
	C\sum_{Q\in W} 
	\int_{Q} \biggl(\sum_{R\in W} \abs{u_R-u_Q}
	\frac{\abs{R}}{\delta(Q,R)^{n+2}}
	\biggr)^2 t\,dx\,dt
.\end{align*}

To analyze the sum over $R\in W$, we will divide $W$ into a ``discretized cone'' over $Q$ and a leftover  region. 
Let $G'(Q)$ be the discretized cone given by
\begin{equation*}G'(Q)=\{R\in W:\ell(R)> \ell(Q),\>\ell(R)>\dist(\pi(R),\pi(Q))\}.\end{equation*}
Let $G(Q)$ be the larger discretized cone given by
\begin{equation*}G(Q)=\{R\in W:P(R)\in G'(Q).\} \end{equation*}
Let $B(Q)=G(Q)\setminus G'(Q)$ be the lower boundary of $G(Q)$. 
If $R$ is a cube, let $T(R)$ be the set of cubes below $R$, that is,
\begin{equation*}T(R)=\{S\in W: \pi(S)\subsetneq \pi(R)\}.\end{equation*}
Observe that if $R$, $S\in B(Q)$ are distinct cubes then their projections $\pi(R)$ and $\pi(S)$ are disjoint, and so if $R\in B(Q)$ then $T(R)\cap G(Q)$ is empty. Furthermore, 
if $S$ is a Whitney cube, then either $S\in G(Q)$ or $S\in T(R)$ for a unique cube $R\in B(Q)$.

Thus, we may write
\begin{align*}
{II} &\leq
	C\sum_{Q\in W}
	\int_{Q} \biggl(\sum_{R\in G(Q)} \abs{u_R-u_Q}
	\frac{\abs{R}}{\delta(Q,R)^{n+2}}
	\biggr)^2 t\,dx\,dt
	\\&\quad+
	C\sum_{Q\in W}
	\int_{Q} \biggl(\sum_{R\in B(Q)}\sum_{S\in T(R)} \abs{u_S-u_Q}
	\frac{\abs{S}}{\delta(Q,S)^{n+2}}
	\biggr)^2 t\,dx\,dt
.\end{align*}
We may write $\abs{u_S-u_Q}\leq \abs{u_S-u_R}+\abs{u_R-u_Q}$ in the second sum. Observe that if $R\in B(Q)=G(Q)\setminus G'(Q)$, then $R$ is not in $G'(Q)$ but the cube $P(R)$ above it is, and so either $\ell(R)\approx\ell(Q)$ or $\ell(R)\approx \dist (\pi(R),\pi(Q))$. Thus if $S\in T(R)$ then $\delta(S,Q)\approx \ell(R)\approx\delta(R,Q)$. So we may write
\begin{align*}
{II} &\leq
	C\sum_{Q\in W}
	\int_{Q} \biggl(\sum_{R\in G(Q)} \abs{u_R-u_Q}
	\frac{\abs{R}}{\delta(Q,R)^{n+2}}
	\biggr)^2 t\,dx\,dt
	\\&\quad+
	C\sum_{Q\in W}
	\int_{Q} \biggl(\sum_{R\in B(Q)}\sum_{S\in T(R)} \abs{u_S-u_R}
	\frac{\abs{S}}{\delta(Q,R)^{n+2}}
	\biggr)^2 t\,dx\,dt
\\&= C({V}+{VI})
.\end{align*}

We begin by analyzing the term ${VI}$. Choose some $R\in B(Q)$.
Let $L_0=\{R\}$, and let $L_j=\{S\in T(R):\ell(S)=2^{-j} \ell(R)\}$. Then
\begin{equation*}
\sum_{S\in T(R)}
	\abs{u_S-u_R}\,\abs{S}
=
	\abs{R}\sum_{j=1}^\infty 2^{-j(n+1)}
	\sum_{S\in L_j} \abs{u_S-u_R} .\end{equation*}
Let $U_j = \sum_{S\in L_j} \abs{u_S-u_R}$; observe that $U_0=0$. Then for each $j\geq 1$,
\begin{align*}
U_j
&=
	\sum_{S\in L_j} \abs{u_S-u_R}
\leq 
	\sum_{S\in L_j} \abs{u_S-u_{P(S)}}+\sum_{S\in L_j} \abs{u_{P(S)}-u_R}
.\end{align*}
The second term is equal to $2^n U_{j-1}$. To contend with the first term, we apply the Poincar\'e inequality in the set $S\cup P(S)$. Then
\begin{align*}
\abs{u_S-u_{P(S)}}
&\leq
\abs{u_S-u_{S\cup P(S)}} + \abs{u_{P(S)}-u_{S\cup P(S)}}
\\&=
\frac{1}{\abs{S}}\int_S \abs{u-u_{S\cup P(S)}}
+
\frac{1}{2^n\abs{S}}\int_{P(S)} \abs{u-u_{S\cup P(S)}}
\\&\leq
\frac{1}{\abs{S}}\int_{S\cup P(S)} \abs{u-u_{S\cup P(S)}}
\leq
C\frac{\ell(S)}{\abs{S}}\int_{S\cup P(S)} \abs{\nabla u(y,s)}\,dy\,ds
.\end{align*}
Thus,
\begin{align*}
U_j
&\leq
	2^n U_{j-1}
	+\sum_{S\in L_j} \frac{C}{\ell(S)^n} \int_{S} \abs{\nabla u}
	+2^{2n}\sum_{S\in L_{j-1}} \frac{C}{\ell(S)^n} \int_{S} \abs{\nabla u}
.\end{align*}
Since $U_0=0$, by induction
\begin{align*}
U_j
&\leq
	\sum_{k=0}^{j-1} (2^n+1) 2^{n(j-k)}\sum_{S\in L_k} \frac{C}{\ell(S)^n}\int_S\abs{\nabla u}
	+\sum_{S\in L_j} \frac{C}{\ell(S)^n}\int_S\abs{\nabla u}
\end{align*}
which may be simplified to 
\begin{align*}
U_j
&\leq
	C\sum_{k=0}^{j} 2^{n(j-k)}\sum_{S\in L_k} \frac{1}{\ell(S)^n}\int_S\abs{\nabla u}
.\end{align*}
Thus
\begin{align*}
\sum_{S\in T(R)}
	\abs{u_S-u_R}\, \abs{S}
&=
	\abs{R}\sum_{j=1}^\infty 2^{-j(n+1)} U_j
\\&\leq
	\abs{R}\sum_{j=1}^\infty 2^{-j(n+1)} \sum_{k=0}^j 2^{n(j-k)}\sum_{S\in L_k} \frac{C}{\ell(S)^n}\int_S \abs{\nabla u(y,s)}\,dy\,ds
\\&=
	\abs{R}
	\sum_{k=0}^\infty 2^{-k-n k}
	\sum_{S\in L_k} \frac{C}{\ell(S)^n}\int_S \abs{\nabla u(y,s)}\,dy\,ds
\\&\leq
	C
	\sum_{S\in T(R)\cup\{R\}} \ell(S)\int_S \abs{\nabla u(y,s)}\,dy\,ds
.\end{align*}
So
\begin{align*}
{VI} &=
	\sum_{Q\in W}
	\int_{Q} \biggl(\sum_{R\in B(Q)} \frac{1}{\delta(Q,R)^{n+2}} \sum_{S\in T(R)} \abs{u_S-u_R}\,\abs{S}	
	\biggr)^2 t\,dx\,dt
\\&\leq C
	\sum_{Q\in W}
	\int_{Q} \biggl(\sum_{R\in B(Q)} \frac{1}{\delta(Q,R)^{n+2}} \sum_{S\in T(R)\cup\{R\}} \ell(S)\int_S \abs{\nabla u(y,s)}\,dy\,ds
	\biggr)^2 t\,dx\,dt
\\&\leq C
	\int_{\RR^{n+1}_+} \biggl(\int_{\RR^{n+1}_+} \frac{s}{(t+s+\abs{x-y})^{n+2}} \abs{\nabla u(y,s)}\,dy\,ds
	\biggr)^2 t\,dx\,dt
\end{align*}
because $\delta(Q,R)\approx \delta(Q,S)\approx t+s+\abs{x-y}$ for any $S\in T(R)$, any $R\in B(Q)$ and any $(x,t)\in Q$ and $(y,s)\in S$.

Now, recall that
\begin{align*}
{V} &=
\sum_{Q\in W}
	\int_{Q} \biggl(\sum_{R\in G(Q)} \abs{u_R-u_Q}
	\frac{\abs{R}}{\delta(Q,R)^{n+2}}
	\biggr)^2 t\,dx\,dt
	.\end{align*}
Observe that if $R\in G(Q)$ then $\dist(\pi(R),\pi(Q))\leq C\ell(R)$ and so $\delta(R,Q)\approx \ell(R)$; thus
\begin{align*}
{V} &\leq
	C\sum_{Q\in W}
	\int_{Q} \biggl(\sum_{R\in G(Q)} \abs{u_R-u_Q}
	\frac{1}{\ell(R)}
	\biggr)^2 t\,dx\,dt	
	.\end{align*}
Let $G_j(Q)=\{R\in G(Q): \ell(R)=2^j\ell(Q)\}$. Let $x_Q$ be the midpoint of the cube $\pi(Q)$. There is some constant $c$ depending only on dimension such that if $R\in G_j(Q)$ then $R$ is contained in the cylinder $C_j(Q)= \Delta(x_Q, c\,2^j\ell(Q))\times (2^j\ell(Q),2^{j+1}\ell(Q))$.
Therefore,
\begin{align*}
\sum_{R\in G_j(Q)} \abs{u_R-u_Q}
&=
	\frac{1}{2^{j(n+1)}\abs{Q}} \sum_{R\in G_j(Q)} \int_R \abs{u(y,s)-u_Q}\,dy\,ds
\\&\leq
	C\fint_{C_j(Q)} \abs{u(y,s)-u_{C_j(Q)}}\,dy\,ds
	+
	C\abs{u_{C_j(Q)}-u_Q}
.\end{align*}
If $j\geq 1$, then 
\begin{align*}
\abs{u_{C_j(Q)}-u_Q}
&\leq
	\abs{u_{C_0(Q)}-u_Q}
	+\sum_{k=1}^j \abs{u_{C_k(Q)}-u_{C_{k-1}(Q)}}
.\end{align*}
Applying the Poincar\'e inequality in $C_k(Q)\cup C_{k-1}(Q)$, we see that 
\begin{align*}
\abs{u_{C_k(Q)}-u_{C_{k-1}(Q)}}
&\leq 
\frac{C}{2^{kn}\ell(Q)^n}\int_{C_{k}(Q)\cup C_{k-1}(Q)} \abs{\nabla u(y,s)}\,dy\,ds
.\end{align*}
Observe that $Q\subset C_0(Q)$, and so by the Poincar\'e inequality,
\begin{equation*}
\abs{u_{C_0(Q)}-u_Q} \leq C\,\ell(Q) \fint_{C_0(Q)} \abs{\nabla u(y,s)}\,dy\,ds.
\end{equation*}
Finally, $\fint_{C_j(Q)} \abs{u(y,s)-u_{C_j(Q)}}\,dy\,ds\leq C 2^j\ell(Q)\fint_{C_j(Q)} \abs{\nabla u(y,s)}\,dy\,ds$.
So
\begin{align*}
\sum_{R\in G_j(Q)} \abs{u_R-u_Q}
&\leq
	\sum_{k=0}^j \frac{C}{2^{kn}\ell(Q)^n}\int_{C_{k}(Q)} \abs{\nabla u(y,s)}\,dy\,ds
 .\end{align*}
Thus,
\begin{align*}
\sum_{R\in G(Q)} \frac{1}{\ell(R)} \abs{u_R-u_Q}
&=
	\sum_{j=0}^\infty \frac{1}{2^j\ell(Q)}
	\sum_{R\in G_j(Q)} \abs{u_R-u_Q}
\\&\leq 
	\sum_{j=0}^\infty \frac{1}{2^j\ell(Q)}
	\sum_{k=0}^j \frac{C}{2^{kn}\ell(Q)^n}\int_{C_{k}(Q)} \abs{\nabla u(y,s)}\,dy\,ds
\\&=
	C\sum_{k=0}^\infty
	\frac{2^k\ell(Q)}{2^{k(n+2)}\ell(Q)^{n+2}}\int_{C_{k}(Q)} \abs{\nabla u(y,s)}\,dy\,ds
.\end{align*}
But if $(y,s)\in C_k(Q)$ and $(x,t)\in Q$, then $\abs{x-y}+t\leq C 2^k\ell(Q)$ and $s\approx 2^k\ell(Q)$. Thus, 
\begin{align*}
{V} &\leq
	C\sum_{Q\in W}
	\int_{Q} \biggl(\sum_{R\in G(Q)} \abs{u_R-u_Q}
	\frac{1}{\ell(R)}
	\biggr)^2 t\,dx\,dt	
\\&\leq
	C\sum_{Q\in W}
	\int_{Q} \biggl(\sum_{k=0}^\infty
	\frac{2^k\ell(Q)}{2^{k(n+2)}\ell(Q)^{n+2}}\int_{C_{k}(Q)} \abs{\nabla u(y,s)}\,dy\,ds
	\biggr)^2 t\,dx\,dt	
\\&\leq
	C\int_{\RR^{n+1}_+} \biggl(
	\int_{\RR^{n+1}_+} \frac{s}{(\abs{x-y}+t+s)^{n+2}}\abs{\nabla u(y,s)}\,dy\,ds
	\biggr)^2 t\,dx\,dt	
.\end{align*}

This completes the proof.\end{proof}

\section{A nontangential bound}
\label{sec:nontangential}
We have established that, if $h\in L^2(\RR^n)$, then under appropriate assumptions on $a$, $A$ and~$B$,
\begin{equation}
\label{eqn:Fsquarebound}
\int_{\RR^{n}}\int_0^\infty \abs{\partial_t^2 \F h(x,\pm t)}^2 {t}\,dt \,dx
\leq C\doublebar{h}_{L^2(\RR^n)}^2.
\end{equation}

In this section we will prove that, if $h\in L^2(\RR^n)$, then under appropriate assumptions on $a$, $A$ and~$B$,
\begin{equation*}
\triplebar{t\,\nabla\partial_t \E h)}_+
+\doublebar{\widetilde N_+(\nabla\E h)}_{L^2(\RR^n)}\leq C\doublebar{h}_{L^2(\RR^n)}^2.\end{equation*}
Recall that $\E h= \F h - \s_A\bigl(\frac{1}{a}\s_{B^*}^{\perp,+} h\bigr)$; thus by \eqref{eqn:NnablaS} and \eqref{eqn:Ssquare}, the difference $\E h-\F h$ satisfies square-function estimates and nontangential estimates, and so we may work with $\E$ or $\F$, whichever is more convenient.

This proof will require several steps.

\begin{lem}\label{lem:bdd1}
Suppose that $v\in W^2_{1,loc}(\RR^{n+1}_\pm)$ satisfies the conditions of \lemref{slabs}, that is, that $v$, $\partial_{n+1} v$ and $\partial_{n+1}^2 v$ satisfy the Caccioppoli inequality in~$\RR^{n+1}_\pm$.

Suppose furthermore that there is a constant $C_1$ such that
\begin{equation}
\label{eqn:L2forsquare}
\doublebar{\nabla v(\,\cdot\,,t)}_{L^2(\RR^n)} \leq C_1 
\abs{t}^{-1}.\end{equation}
Then there is some constant $C$, depending only on the constants in the Caccioppoli inequality, such that
\begin{equation*}
\int_0^\infty\int_{\RR^n} \abs{\nabla v(x,\pm t)}^2\, t\,dx\,dt
\leq
C\int_0^\infty\int_{\RR^n} \abs{ \partial_t v(x,\pm t)}^2\, t\,dx\,dt
+C_1^2
.\end{equation*}
\end{lem}

We will immediately apply this lemma to $\partial_t \F h$ in the lower half-space. Later in \crlref{Esquare}, we will apply this lemma to $\partial_t \E h$ in the upper half-space as well. The bound \eqref{eqn:L2forsquare} is simpler to establish in the lower half-space 
because $\F h(x,t)$ is defined in terms of an integral over $\RR^{n+1}_+$, and certain direct bounds can be computed if $(x,t)$ does not lie in the integrand.

\begin{proof} This parallels the proof of \cite[Formula (5.5)]{AlfAAHK11}, where a similar inequality was proven for the single layer potential.

Define $u_m(t)=\int_{\RR^n} \abs{\nabla \partial_t^m v(x,\pm t)}^2 \,dx$.
Observe that by assumption $u_0(t)\leq C_1^2 t^{-2}$, and by the Caccioppoli inequality,
\begin{equation*}
\int_0^\infty t^3 \,u_1(t)\,dt =
\int_{\RR^{n+1}_+} \abs{\nabla \partial_t v(x,\pm t)}^2 \,t^3 \,dx\,dt
\leq
C\int_{\RR^{n+1}_+} \abs{\partial_t v(x,\pm t)}^2\, t \,dx\,dt.
\end{equation*}
We wish to bound $\int_0^\infty t \,u_0(t)\,dt$.

Suppose $t>0$. By \eqref{eqn:slabscts}, $u_0$ is locally Lipschitz continuous.
Thus if $0<\varepsilon<S<\infty$, then 
\begin{align*}
\int_{\varepsilon}^S t \,u_0(t)\,dt
&=
\int_{\varepsilon}^S t\, u_0(S)\,dt
-\int_{\varepsilon}^S t\int_t^S u_0'(s)\,ds\,dt
\\&\leq
\frac{C_1^2}{2}+\frac{1}{2}\int_{\varepsilon}^Ss^{2}\, \abs{u_0'(s)}\,ds.
\end{align*}
But 
$\abs{u_0'(s)}
\leq 2 \sqrt{u_0(s)\,u_1(s)}
$
and so
\begin{align*}
\int_{\varepsilon}^S t \,u_0(t)\,dt
&\leq
\frac{C_1^2}{2}+\int_{\varepsilon}^S s^2\, \sqrt{u_0(s)\,u_1(s)}\,ds
\\&\leq
\frac{C_1^2}{2}
+\frac{1}{2}\int_{\varepsilon}^S s\, u_0(s)\,ds
+\frac{1}{2}\int_{\varepsilon}^S s^3\, u_1(s)\,ds
.\end{align*}
Rearranging terms, we have that
\begin{align*}
\int_{\varepsilon}^S t \,u_0(t)\,dt
&\leq
C_1^2
+\int_{\varepsilon}^S s^3\, u_1(s)\,ds
.\end{align*}
Taking the limit as $\varepsilon\to 0^+$ and $S\to \infty$, we have that
\begin{align*}
\int_0^\infty\int_{\RR^n} \abs{\nabla v(x,\pm t)}^2 \,t\,dx\,dt
&=
\int_{0}^\infty t \,u_0(t)\,dt
\leq
C_1^2+\int_{0}^\infty t^3 \,u_1(t)\,dt
\\&\leq
C_1^2+
C\int_0^\infty\int_{\RR^n} \abs{ \partial_t v(x,\pm t)}^2 \,t\,dx\,dt
\end{align*}
as desired.\end{proof}

We now proceed to nontangential estimates in the lower half-space.
\begin{lem}
\label{lem:bdd2}
Suppose that $a$, $A$ and $B$ are $t$-independent, $a$ is accretive, $A$ satisfies the single layer potential requirements of \dfnref{goodlayer}, $B$ and $B^*$ satisfy the De Giorgi-Nash-Moser condition and $\s_{B^*}$ satisfies the square-function estimate~\eqref{eqn:Svertsquare}.

Then for every $h\in L^2(\RR^n)$, we have that
\begin{align}
\label{eqn:NnablaEminus}
\doublebar{\widetilde N_-(\nabla \E h)}_{L^2(\RR^n)}
&\leq C\doublebar{h}_{L^2(\RR^n)}
\end{align}
and the boundary value $\nabla \E h\big\vert_{\partial\RR^{n+1}_-}$ exists in the sense of \eqref{eqn:trace}.
\end{lem}

\begin{proof}
By the definitions \eqref{eqn:F} and \eqref{eqn:S} of $\F$ and $\s_A$,
\begin{equation}
\label{eqn:um}
\nabla \partial_t \F h(x,-t)
=
-\int_0^\infty
\nabla \partial_t^2\s_A\biggl(\frac{1}{a} \partial_s \s_{B^*} h(s)\biggr)(x,-t-s)\,ds
\end{equation}
where as in the proof of \lemref{gradEcts} we let $\s_{B^*}h(s)(y)=\s_{B^*} h(y,s)$.
Because $A$ and $B^*$ are $t$-independent and have bounded layer potentials,
if $t>0$, then by \eqref{eqn:2slabs}, \eqref{eqn:NSt} and the Caccioppoli inequality,
$\doublebar{\nabla \partial_t \F h(\,\cdot\,,-t)}_{L^2(\RR^n)} \leq Ct^{-1}\doublebar{h}_{L^2(\RR^n)}^2$.
Thus by \lemref{bdd1},
\begin{equation*}
\int_0^\infty\int_{\RR^n} \abs{\nabla \partial_t \F h(x,-t)}^2\, t\,dx\,dt
\leq
C\int_0^\infty\int_{\RR^n} \abs{ \partial_t^2 \F h(x,-t)}^2\, t\,dx\,dt
+C\doublebar{h}_{L^2(\RR^n)}^2
.\end{equation*}
By \thmref{squarebound} this is at most $C\doublebar{h}_{L^2(\RR^n)}^2$.

Since $\Div A\nabla (\partial_t\F h) = 0$ in $\RR^{n+1}_-$,  we may apply \thmref{squareN} in the lower half-space. Thus 
\begin{equation}
\label{eqn:NvertF}
\doublebar{N_-(\partial_t \F h)}_{L^2(\RR^n)}^2
\leq
\int_{\RR^{n+1}_+} \abs{\nabla \partial_t \F h(x,-t)}^2 \,t \,dx\,dt
\leq
C\doublebar{h}_{L^2(\RR^n)}^2
\end{equation}
and the boundary value
\begin{equation*}\F^{\perp,-} h=\partial_{n+1} \F h\big\vert_{\partial\RR^{n+1}_-}\end{equation*}
exists in the sense of $L^2$ functions.

Because $A$ satisfies the single layer potential requirements, $\s_A^{\perp,-}$ is invertible $L^2(\RR^n)\mapsto L^2(\RR^n)$, and so there is some $g\in L^2(\RR^n)$ with $\s_A^{\perp,-} g= \F^{\perp,-} h$. Furthermore, $\doublebar{g}_{L^2(\RR^n)}\leq C\doublebar{h}_{L^2(\RR^n)}$.

By \eqref{eqn:NSt}, if $u=\partial_{n+1}\s_A g$, then $\doublebar{N_- u}_{L^2(\RR^n)}<\infty$.
By \eqref{eqn:NvertF} the same is true of $u=\partial_{n+1}\F h$, and so by uniqueness of solutions to $(D)^A_2$, we have that $\partial_{n+1}\s_A g = \partial_{n+1}\F h$ in $\RR^{n+1}_-$. 
By \eqref{eqn:NnablaS}, and by \eqref{eqn:Fdiffdecay} and \thmref{Caccioppoli}, we have that $u=\s_A g-\F h$ satisfies the conditions of \lemref{utunique}, and so
$\s_A g=\F h$ in $\RR^{n+1}_-$ up to an additive constant.

In particular, by \eqref{eqn:NnablaS} we have that
\begin{equation*}
\doublebar{\widetilde N_- (\nabla \F h)}_{L^2(\RR^n)}\leq C\doublebar{h}_{L^2(\RR^n)}.\end{equation*}
By definition of $\E h$ and by \eqref{eqn:NnablaS}, \eqref{eqn:NnablaEminus} is valid. 
By \thmref{2L2limits}, and because $\Div A\nabla (\E h)=0$ in $\RR^{n+1}_-$, we have that $\nabla \E h\vert_{\partial\RR^{n+1}_-}$ exists and lies in $L^2(\RR^n)$.
\end{proof}

\begin{lem}\label{lem:bdd5}
Suppose that $a$, $A$ and~$B$ are as in \lemref{Fexists}. Let $h\in L^2(\RR^n)$. 
Then
\begin{equation*}
\doublebar{N_+ (\partial_t \E h)}_{L^2(\RR^n)}
\leq
C\doublebar{N_- (\partial_t \E h)}_{L^2(\RR^n)}
+C\doublebar{h}_{L^2}
.\end{equation*}
\end{lem}

\begin{proof}
This is essentially the analogue to Cotlar's inequality for singular integral operators (see, for example, \cite[(8.2.2)]{Gra09}) reformulated to apply to our potential~$\E$. We prove it using similar arguments.

Let $x^*\in\RR^n$ and let $(x,t)\in \gamma_+(x^*)$ for some $t>0$, where $\gamma_+$ is the nontangential cone given by \eqref{eqn:cone}. 
There is some constant $j_0$, depending on the aperture $a$ of the nontangential cones, such that $x\in \Delta(x^*,2^{j_0}t)$.
Let $A_0=\Delta_0 = \Delta(x^*,2^{j_0+1} t)$, and for each $j\geq 1$, let $A_j$ be the annulus $\Delta(x^*,2^{j_0+j+1} t)\setminus \Delta(x^*,2^{j_0+j} t)$. Let $h_j=h\1_{A_j}$, so  $h=\sum_{j=0}^\infty h_j$. Then
\begin{align*}
\abs{\partial_{n+1} \E h(x,t)}
&\leq 
\abs{\partial_{n+1} \E h(x,-t)}
+\sum_{j=0}^\infty
\abs{\partial_{n+1} \E h_j(x,t)-\partial_{n+1} \E h_j(x,-t)}
.\end{align*}
Observe that $\abs{\partial_t \E h(x,-t)}\leq N_-(\partial_{n+1}\E h)(x^*)$.
By \eqref{eqn:Fdecay} and \crlref{4holder}, if $1<p<n$ then
\begin{equation*}\abs{\partial_t \F h_j(x,t)}+\abs{\partial_t \F h_j(x,-t)}
\leq
\frac{C}{t^{n/p}}\doublebar{h_j}_{L^p(\RR^n)}
\leq C 2^{j n/p} M(\abs{h}^p)(x^*)^{1/p}
\end{equation*}
where $M$ denotes the Hardy-Littlewood maximal function. 
Recall that
\begin{align*}
\E h_j(x,\tau)
&=\F h_j(x,\tau)-\s_A\left(\frac{1}{a}\s_{B^*}^{\perp,+} h_j\right)(x,\tau)
\\&=
	-\int_{\RR^{n+1}_+}
	\partial_s \Gamma_{(y,s)}^A(x,\tau)\,\frac{1}{a(y)}
	\,\partial_s \s_{B^*} h_j(y,s)\,ds\,dy
	-\s_A\left(\frac{1}{a}\s_{B^*}^{\perp,+} h_j\right)(x,\tau).
\end{align*}
Applying \eqref{eqn:Shdecay} to $\s_A$ and \eqref{eqn:NSt} to~$\s_{B^*}^{\perp,+}$, we see that the same is true of $\partial_t \E h_j$. We will apply this bound only for $j=0$ and $j=1$; we seek a better bound for $j\geq 2$.

Suppose $j\geq 2$. 
Let $\widetilde A_j=A_{j-1}\cup A_j\cup A_{j+1}$.
We divide $\RR^{n+1}_+$ into three pieces:
$I=\widetilde A_j \times(0,2^{j}t)$, $II=\widetilde A_j\times(2^jt,\infty)$
and $III=(\RR^n\setminus \widetilde A_j)\times(0,\infty)$. 
We integrate by parts in $s$ in $II$ and~$III$. Recalling the definition of $\s_A$, we have that
\begin{align*}
\partial_\tau\E h_j(x,\tau)
&=
		\int_{I}
		\partial_s^2 \Gamma_{(y,s)}^A(x,\tau)\,\frac{1}{a(y)}
		\,\partial_s \s_{B^*} h_j(y,s)\,ds\,dy
	\\&\qquad
		-\int_{II\cup III}
		\partial_s \Gamma_{(y,s)}^A(x,\tau)\,\frac{1}{a(y)}
		\,\partial_s^2 \s_{B^*} h_j(y,s)\,ds\,dy
	\\&\qquad
		+\int_{\widetilde A_j}
		\partial_\tau \Gamma_{(y,2^j t)}^A(x,\tau)\,\frac{1}{a(y)}
		\,\partial_{n+1} \s_{B^*} h_j(y,2^j t)\,dy
	\\&\qquad
		-\partial_{\tau} \s_A\left(\1_{\widetilde A_j}\frac{1}{a}\s^{\perp,+}_{B^*} h_j\right)(x,\tau)
.\end{align*}
Thus $\partial_{n+1} \E h_j(x,t)-\partial_{n+1} \E h_j(x,-t)$ may be written as a sum of four terms. Applying the bound \eqref{eqn:NSt} on $\s_{B^*}$, the bound \eqref{eqn:Shdecay} and H\"older continuity to $\s_A$ and $\s_{B^*}$, and the bounds \eqref{eqn:fundsolnsize:k} and \eqref{eqn:fundsolnholder:k} on $\Gamma_{(y,s)}^A(x,\tau)$, we may show that each term is at most $C2^{-j\alpha} M(\abs{h}^p)(x^*)^{1/p}$.

Thus, if $1<p<n$ then
\begin{align*}
\abs{\partial_t \E h(x,t)}
&\leq 
N_-(\partial_{n+1} \E h)(x^*)
+\sum_{j=0}^\infty C 2^{-j\alpha}M(\abs{h}^p)(x^*)^{1/p}
.\end{align*}
Choosing $1<p<2$ and recalling that $M$ is bounded $L^p\mapsto L^p$ for any $1<p<\infty$, we have that
\begin{align*}
\doublebar{N_+(\partial_{n+1} \E h)}_{L^2(\RR^n)}
&\leq 
\doublebar{N_-(\partial_{n+1} \E h)}_{L^2(\RR^n)}
+C\doublebar{h}_{L^2(\RR^n)}
\end{align*}
as desired.
\end{proof}

\begin{lem}\label{lem:bdd6} 
Suppose that $a$, $A$ and $B$ are $t$-independent, that $a$ is accretive, and that $A$ and~$B^*$ satisfy the single layer potential requirements of \dfnref{goodlayer}.
Then for all $h\in L^2(\RR^n)$, 
\begin{equation}
\label{eqn:NnablaE}
\doublebar{\widetilde N_+ (\nabla \E  h)^2}_{L^2(\RR^n)}
\leq
C\doublebar{h}_{L^2(\RR^n)}
\end{equation}
and $\nabla \E h\vert_{\partial\RR^{n+1}_+}$ exists and satisfies
\begin{equation}
\label{eqn:Ects}
\nabla \E h\big\vert_{\partial\RR^{n+1}_+}
=\nabla \E h\big\vert_{\partial\RR^{n+1}_-}
.\end{equation}
\end{lem}

\begin{proof}
First, assume that $h\in S$ where $S$ is as in \lemref{gradEcts}. By Lemmas~\ref{lem:gradEcts} and~\ref{lem:bdd2}, the formula \eqref{eqn:Ects} is valid, and so we may define 
$\E^\pm h=\E h\vert_{\partial\RR^{n+1}_\pm}$.

Let $x^*\in\RR^n$ and let $(x,t)\in \gamma(x^*)$. 
By \thmref{Caccioppoli} and \crlref{reverseholder},
we have that
\begin{align*}
\fint_{B((x,t),t/2)}\abs{\nabla \E h}^2
&\leq 
\frac{C}{r^2}\fint_{B((x,t),5t/8)} \abs{\E h-E}^2
\leq
\frac{C}{r^2}\left(\fint_{B((x,t),3t/4)} \abs{\E h-E}\right)^{2}
\end{align*}
for any constant $E$.

If $(y,s)\in\RR^{n+1}_+$, then
\begin{align*}
\abs{\E h(y,s)-E}
&\leq \abs{\E h(y,s)-\E^+ h(y)}+\abs{\E^+ h(y)-E}
\\&\leq sN_+(\partial_{n+1} \E h)(y)+\abs{\E^+ h(y)-E}.
\end{align*}
Let $\Delta = \Delta(x^*,Ct)$ where $C$ is large enough that $\Delta(x,t)\subset \Delta(x^*,Ct)$.  Then 
\begin{align*}
\fint_{B((x,t),t/2)}\abs{\nabla \E h}^2
&\leq 
\frac{C}{t^2}\left(
\fint_{\Delta} t N_+(\partial_{n+1} \E h)(y)
+\abs{\E^+ h(y)-E}\,dy\right)^{2}
\\&\leq 
C\left(
\fint_{\Delta} N_+(\partial_{n+1} \E h)(y)\,dy\right)^{2}
+
\frac{C}{r^2}\left(
\fint_{\Delta} \abs{\E^+ h(x)-E}\,dx\right)^{2}
.
\end{align*}
Choose $E=\fint_\Delta \E h(x,0)\,dx$ and apply the Poincar\'e inequality.
We conclude that
\begin{align*}
\widetilde N_+(\nabla\E  h)(x^*)^2
&=\sup_{(x,t)\in\gamma(x^*)}\fint_{B((x,t),t/2)}\abs{\nabla \E  h}^2
\\&\leq 
	C M(N_+(\partial_t \E  h))(x^*)^2 +C M(\nabla_\parallel\E^+ h)(x^*)^2
.\end{align*}
By \lemref{bdd2} and \lemref{bdd5}, and by the $L^2$-boundedness of the Hardy-Littlewood maximal operator, we have that \eqref{eqn:NnablaE} is valid for all $h\in S$. 

We now must pass to arbitrary $h\in L^2(\RR^n)$. Define 
\begin{equation*}\widetilde N_{\delta,R} F(x) = \sup\biggl\{\biggl(\fint_{B((y,s),s/2)} \abs{F}^2\biggr)^{1/2}: (y,s)\in\gamma_+(x), \>\delta<s<R,\>\abs{y}<R\biggr\}.\end{equation*}
By \eqref{eqn:Fdecay}, \eqref{eqn:Shdecay} and \thmref{Caccioppoli}, for any fixed $\delta>0$, $R<\infty$ the map $h\mapsto \widetilde N_{\delta,R} (\nabla\E h)$ is continuous on $L^2(\RR^n)$. Thus because $S$ is dense in $L^2(\RR^n)$, we have that for any $h\in L^2(\RR^n)$,
\begin{equation*}
\doublebar{\widetilde N_{\delta,R} (\nabla \E  h)}_{L^2(\RR^n)}
\leq
C\doublebar{h}_{L^2(\RR^n)}
\end{equation*}
uniformly in $\delta$, $R$. Letting $\delta\to 0$ and $R\to\infty$ establishes \eqref{eqn:NnablaE}. Recall that \eqref{eqn:Ects} is valid for all $h\in S$. By \eqref{eqn:NnablaE} and \lemref{slabs}, we may extend \eqref{eqn:Ects} to all of $L^2(\RR^n)$.
\end{proof}

\begin{cor}
\label{cor:Esquare}
Suppose that $a$, $A$ and~$B$ are as in \lemref{bdd6}. 
If $h\in L^2(\RR^n)$, then
\begin{equation*}
\int_0^\infty\int_{\RR^n} \abs{\nabla \partial_t \E h(x,t)}^2 \,t\,dx\,dt
\leq
C\doublebar{h}_{L^2(\RR^n)}^2
.\end{equation*}
\end{cor}

\begin{proof} By \thmref{squarebound}, \eqref{eqn:Svertsquare} and the $L^2$-boundedness of $\s_{B^*}^{\perp,+}$, we have that 
\begin{equation*}
\int_0^\infty\int_{\RR^n} \abs{ \partial_t^2 \E h(x,t)}^2 \,t\,dx\,dt
\leq
C\doublebar{h}_{L^2(\RR^n)}^2
.\end{equation*}
By Lemmas \ref{lem:slabs} and~\ref{lem:bdd6},
we have that
\begin{equation*}\doublebar{\nabla \partial_t \E h(\,\cdot\,,t)}_{L^2(\RR^n)}\leq \frac{C}{t}\doublebar{\widetilde N_+(\nabla \E h)}_{L^2(\RR^n)}
\leq\frac{C}{t}\doublebar{h}_{L^2(\RR^n)}\end{equation*}
and so the conclusion follows from \lemref{bdd1}.
\end{proof}

\section{The proof of the main theorem}
\label{sec:final}
In this section, we will first prove an invertibility result for the potential $\E=\E_{B,a,A}$. We will use this invertibility result to prove existence of solutions to the fourth-order Dirichlet problem. Finally, we will conclude this paper by proving uniqueness of solutions and establishing a Fatou-type theorem.

We remark that the invertibility argument and the construction of solutions in this section closely parallels the construction of solutions for the biharmonic Dirichlet problem of \cite{DahKV86} and \cite{PipV92}.

The invertibility result we will prove is the following.
\begin{lem} 
\label{lem:invertible}
Suppose that $a$, $A$ and $B$ are $t$-independent, that $a$ is accretive and that $A$ and $A^*$ satisfy the single layer potential requirements of \dfnref{goodlayer}.
Then there is some $\varepsilon>0$ and some $C>0$, depending only on the parameters listed in \thmref{Dirichletexists}, such that if
\begin{equation*}\doublebar{\im a}_{L^\infty(\RR^n)}+\doublebar{A-B}_{L^\infty(\RR^n)}<\varepsilon\end{equation*}
then the mapping $h\mapsto \partial_{n+1} \E_{B,a,A} h\vert_{\partial\RR^{n+1}_\pm}$ is invertible on $L^2(\RR^n)$, and its inverse has norm at most~$C$.
\end{lem}

\begin{proof}
Define 
$\E^\perp_{B,a,A} h=\partial_{n+1} \E_{B,a,A}h\vert_{\partial\RR^{n+1}_-}$.
We begin by showing that $\E^\perp_{A, a, A}$ is invertible for $a$ real; it is for this reason that we require that both $A$ and $A^*$ satisfy the single layer potential requirements.

Choose some $h\in S$, where $S\subset L^2(\RR^n)$ is as in \lemref{gradEcts}.
By definition of the single layer potential,
\begin{equation*}
{\int_{\RR^n} g(x)\,\partial_t \F_{A,a,A} h(x,-t)\,dx}
	=\int_{\RR^n}\int_0^\infty
	\partial_t^{2} \s_{A^T} g(y,s+t) \,\frac{1}{a(y)}
	\,\partial_s \s_{A^*} h(y,s)\,ds\,dy.
\end{equation*}
Integrating by parts and applying the dominated convergence theorem 
yields that
\begin{align*}
\int_{\RR^n} g(x)\,\E^\perp_{A,a,A} f(x)\,dx
&=-\lim_{t\to 0^-}\int_{\RR^n}\int_0^\infty
\partial_s \s_{A^T} g(y,s-t) \,\frac{1}{a(y)}
\,\partial_s^2 \s_{A^*} f(y,s)\,ds\,dy
\\
&=
\int_{\RR^n} 
\s_{A^T}^{\perp,+} g(y)\,\frac{1}{a(y)}
\s_{A^*}^{\perp,+} f(y)
\,dy
-\int_{\RR^n} f(x)\,\E^\perp_{\overline A,a,\overline A} g(x)\,dx
.\end{align*}
By density of~$S$ and by the $L^2$-boundedness of $\E^\perp_{A,a,A}$ and $\E^\perp_{\overline A,a,\overline A}$, we have that for any $h\in L^2(\RR^n)$,
\begin{align*}
\int_{\RR^n} \bar h\,\E^\perp_{A,a,A} h
+\int_{\RR^n} h\,\E^\perp_{\overline A,a,\overline A} \bar h
&=
\int_{\RR^n} \frac{1}{a(y)}\,
\s_{ A^T}^{\perp,+}  h(y)\,
\s_{ A^*}^{\perp,+}  h(y)
\,dy
.\end{align*}
Observe that $\overline{\E_{A,a,A} h(x,t)}=\E_{\overline A,\overline a,\overline A} \bar h(x,t)$.
Thus, if $a$ is real, then
\begin{equation*}
\int_{\RR^n} h \,\E^\perp_{\overline A,a,\overline A} \bar h\,d\sigma
=
\overline{ \int_{\RR^n} \bar h \,\E^\perp_{ A,a, A}  h\,d\sigma}
\end{equation*}
and so
\begin{align*}
\re \int_{\RR^n} \bar h \,\E^\perp_{A,a,A} h\,d\sigma
&=
\frac{1}{2}
\int_{\RR^n}
\frac{1}{a(Y)}
\abs{\s_{ A^*}^{\perp,+}  h(y)}^2
\,d\sigma(Y)
\geq \frac{1}{C} \doublebar{\s_{ A^*}^{\perp,+}  h}_{L^2(\RR^n)}^2.
\end{align*}
Thus because $A^*$ satisfies the single layer potential requirements,
\begin{align*}
\biggabs{\int_{\RR^n} \bar h \,\E^\perp_{A,a,A} h\,d\sigma}
&\geq \frac{1}{C}\doublebar{ h}_{L^2(\RR^n)}^2
\end{align*}
and so $\E^\perp_{A,a,A}$ must be one-to-one. But the adjoint $(\E^\perp_{A,a,A})^*$ must also be one-to-one, and so $\E^\perp_{A,a,A}$ must be invertible. Furthermore, $(\E^\perp_{A,a,A})^{-1}$ is bounded $L^2(\RR^n)\mapsto L^2(\RR^n)$, as desired.

We prove invertibility for $\doublebar{\im a}_{L^\infty}$ or $\doublebar{A-B}_{L^\infty}$ small using standard analyticity arguments. Specifically, the mapping $a\mapsto \E^\perp_{A,a,A}$ is analytic in the sense that if $z\mapsto a_z$ is analytic $D\mapsto L^\infty(\RR^n)$ for some $D\subset\CC$, then $z\mapsto \E^\perp_{A,a_z,A}h$ is analytic $D\mapsto L^2(\RR^n)$ for any $h\in L^2(\RR^n)$. Furthermore, under our assumptions, if $a_0$ is real and accretive then $\doublebar{\E^\perp_{A,a,A}h}_{L^2(\RR^n)}$ is bounded for all $t$-independent $a$ in a $L^\infty$ neighborhood of~$a_0$; thus if $\doublebar{a-a_0}_{L^\infty}$ is small enough, then
\begin{equation*}\doublebar{\E^\perp_{A,a,A}-\E^\perp_{A,a_0,A}}_{L^2(\RR^n)\mapsto L^2(\RR^n)}
\leq C\doublebar{a-a_0}_{L^\infty(\RR^n)}.\end{equation*}
Thus, if $\doublebar{\im a}_{L^\infty(\RR^n)}$ is small enough then $\E^\perp_{A,a,A}$ is invertible on $L^2(\RR^n)$.
Similarly, invertibility of $\E^\perp_{B,a,A}$ for $\doublebar{A-B}_{L^\infty(\RR^n)}$ small enough follows from analyticity of the map $B\mapsto \Gamma_{X}^B$.
\end{proof}

We now use this invertibility result to prove existence of solutions to the fourth-order Dirichlet problem.

\begin{thm} \label{thm:final}
Suppose that $a$, $A$ and $B$ satisfy the conditions of \lemref{invertible}.
Let $f\in W^2_1(\RR^n)$ and let $g\in L^2(\RR^n)$.

Then there exists a constant $c$ and an $h\in L^2(\RR^n)$, with $\doublebar{h}_{L^2(\RR^n)}\leq C\doublebar{\nabla f}_{L^2(\RR^n)}+C\doublebar{g}_{L^2(\RR^n)}$,
such that
\begin{equation*}u(X)=-\D_A f(X) -\s_A g(X) + \E_{B,a,A} h(X)+c\end{equation*}
satisfies
\begin{equation*}
\left\{\begin{aligned}
L_B^*(a\,L_Au)&=0  &&\text{in }\RR^{n+1}_+,\\
u&=f  &&\text{on } \partial\RR^{n+1}_+,\\
\e\cdot A\nabla u &= g &&\text{on }\partial\RR^{n+1}_+
\end{aligned}\right.
\end{equation*}
where $L_B^*(a\,L_A u)=0$ in the sense of \dfnref{weaksoln} and where 
$u=f$, $\e\cdot A\nabla u=g$ in the sense that
\begin{equation*}\lim_{t\to 0^+}
\doublebar{ u(\,\cdot\,,t)- f}_{W^2_1(\RR^n)}
+
\doublebar{\e\cdot A\nabla u(\,\cdot\,,t)-g}_{L^2(\RR^n)}
=0.
\end{equation*}
Furthermore,
\begin{align}
\label{eqn:estimate}
\doublebar{\widetilde N_+(\nabla u)}_{L^2(\RR^n)}
+\triplebar{t\,L_A u}_+
+\triplebar{t\,\nabla\partial_t u}_+
&\leq
C \doublebar{\nabla f}_{L^2(\RR^n)}
+C\doublebar{g}_{L^2(\RR^n)}
.\end{align}
\end{thm}

\begin{proof}
Let $v(X)=-\D_A f(X) -\s_A g(X)$, so that $\Div A\nabla v=0$ in $\RR^{n+1}\setminus\RR^n$. By \eqref{eqn:NnablaD} and \eqref{eqn:NnablaS}, and by \eqref{eqn:Dtsquare} and \eqref{eqn:Ssquare}, we have that
\begin{equation}
\label{eqn:vbounds}
\doublebar{\widetilde N_\pm(\nabla v)}_{L^2(\RR^n)}
+\triplebar{t\,\nabla\partial_t v}
\leq  C \doublebar{\nabla f}_{L^2(\RR^n)}+C\doublebar{g}_{L^2(\RR^n)}.\end{equation}

By \thmref{2L2limits}, $\partial_{n+1} v\vert_{\partial\RR^{n+1}_-}$ exists and lies in $L^2(\RR^n)$.
Let $h\in L^2(\RR^n)$ be such that
\begin{equation*}\partial_{n+1} \E_{B,a,A} h\big\vert_{\partial\RR^{n+1}_-}=-\partial_{n+1} v\vert_{\partial\RR^{n+1}_-}.\end{equation*}
By \lemref{invertible}, $h$ exists and satisfies $\doublebar{h}_{L^2(\RR^n)}\leq C \doublebar{\nabla f}_{L^2(\RR^n)}+C\doublebar{g}_{L^2(\RR^n)}$. Thus by \lemref{bdd6} and \crlref{Esquare}, 
\begin{equation}
\label{eqn:Ebounds}
\doublebar{\widetilde N_\pm(\nabla \E_{B,a,A} h)}_{L^2(\RR^n)}
+\triplebar{t\,\nabla\partial_t \E_{B,a,A}  h}
\leq C \doublebar{\nabla f}_{L^2(\RR^n)}+C\doublebar{g}_{L^2(\RR^n)}.\end{equation}
By \eqref{eqn:LFupper}, $L_A \E_{B,a,A} h=(1/a) \partial_{n+1}^2 \s_{B^*} h$ and so by \eqref{eqn:Svertsquare},
\begin{equation*}\triplebar{t\,L_A \E_{B,a,A}  h}_+
\leq C\doublebar{h}_{L^2(\RR^n)}
\leq C \doublebar{\nabla f}_{L^2(\RR^n)}+C\doublebar{g}_{L^2(\RR^n)}.\end{equation*}

Let $w=v+\E_{B,a,A} h=-\D_A f-\s_A g + \E_{B,a,A} h$.  Then $w$ satisfies \eqref{eqn:estimate}.
By \eqref{eqn:LFupper}, $L_B^*(a\,L_A w) = 0 $ in $\RR^{n+1}_+$ in the sense of \dfnref{weaksoln}. We need only show that for some constant~$c$, $u=w+c$ has the correct boundary values.

By our choice of~$h$, $\partial_{n+1} w=0$ on~$\partial\RR^{n+1}_-$.
By \eqref{eqn:vbounds} and \eqref{eqn:Ebounds}, $\partial_{n+1} w$ is a $(D)^A_2$-solution in $\RR^{n+1}_-$, and so $\partial_{n+1} w=0$ in~$\RR^{n+1}_-$. Again by \eqref{eqn:vbounds} and \eqref{eqn:Ebounds}, $w$ satisfies the conditions of \lemref{utunique}, and so $w$ is constant in~$\RR^{n+1}_-$.

By \eqref{eqn:Sjump}, \eqref{eqn:Scts}, \eqref{eqn:Dcts}, and \eqref{eqn:Djump},
\begin{equation*}\lim_{t\to 0^+}  \nabla_\parallel v(\,\cdot\,,t)- \nabla_\parallel v(\,\cdot\,,-t) = \nabla f,
\quad
\lim_{t\to 0^+}  \e\cdot A \nabla v(\,\cdot\,,t) - \e\cdot A\nabla v(\,\cdot\,,-t) = g\end{equation*}
in $L^2(\RR^n)$, and by \eqref{eqn:Ects} the same is true of~$w$. Since $w$ is constant in $\RR^{n+1}_-$, this yields that 
$\e\cdot A\nabla w\vert_{\partial\RR^{n+1}_+} = g$ and that
$\nabla_\parallel w\vert_{\partial\RR^{n+1}_+} = \nabla f$. As in \rmkref{NTlimit} we have that $w\to f-c$, for some constant~$c$, pointwise nontangentially and in~$L^2(\RR^n)$.
Letting $u=w+c$ completes the proof.
\end{proof}

We conclude this paper by proving a Fatou-type theorem and uniqueness of solutions.

\begin{thm}
\label{thm:converse}
Suppose that $a$, $A$ and $B$ are $t$-independent, that $a$ is accretive, and that $A$ and~$B^*$ satisfy the single layer potential requirements of \dfnref{goodlayer}.
We do not require that $\doublebar{\im a}_{L^\infty}$ or $\doublebar{A-B}_{L^\infty}$ be small.

Suppose that $u$ satisfies
\begin{equation*}
\left\{\begin{aligned}
L_B^*(a\,L_Au)&=0  \text{ in }\RR^{n+1}_+,\\
\doublebar{\widetilde N_+(\nabla u)}_{L^2(\RR^n)}
+\triplebar{t\,L_A u}_+
&<\infty
.\end{aligned}\right.
\end{equation*}

Then $\vec G=\nabla u\vert_{\partial\RR^{n+1}_+}$ exists in the sense of \eqref{eqn:trace} and satisfies
\begin{equation}\label{eqn:fatou}\doublebar{\vec G}_{L^2(\RR^n)}\leq C\doublebar{\widetilde N_+(\nabla u)}_{L^2(\RR^n)}
+C\triplebar{t\,L_A u}_+.\end{equation}

Furthermore, there exist functions $u_A\in W^2_{1,loc}(\RR^{n+1}_+)$ and~$h\in L^2(\RR^n)$ that satisfy the estimate
\begin{equation*}
\doublebar{\widetilde N_+(\nabla u_A)}_{L^2(\RR^n)} +\doublebar{h}_{L^2(\RR^n)}
\leq C\doublebar{\widetilde N_+(\nabla u)}_{L^2(\RR^n)}
+C\triplebar{t\,L_A u}_+\end{equation*}
and such that
\begin{equation*}u=u_A+\E_{B,a,A} h \text{ and } \Div A\nabla u_A=0\text{ in }\RR^{n+1}_+.\end{equation*}

\end{thm}

\begin{proof}
If $L_B^*(a\,L_Au)=0$ in the sense of \dfnref{weaksoln}, then there is some $w\in W^2_{1,loc}(\RR^{n+1}_+)$ such that $w=a\,L_A u$ in the weak sense. Furthermore, $\Div B^*\nabla w=0$ in $\RR^{n+1}_+$.

Now, observe that by the De Giorgi-Nash-Moser condition, \thmref{Caccioppoli} and the Poincar\'e inequality, and the definition of~$\widetilde N_+$,
\begin{align*}\abs{w (x,t)} 
&\leq C\left(\fint_{B((x,t),t/4)} \abs{w}^2\right)^{1/2}
\leq \frac{C}{t}\left(\fint_{B((x,t),t/2)} \abs{\nabla u}^2\right)^{1/2}
\\&\leq \frac{C}{t}\inf_{\abs{x-y}<t/C} \widetilde N_+(\nabla u)(y)
\leq \frac{C}{t} \left(\fint_{\Delta(x,t/C)} \widetilde N_+(\nabla u)^2\right)^{1/2}
\\&\leq Ct^{-1-n/2} \doublebar{\widetilde N_+(\nabla u)}_{L^2(\RR^n)}
.\end{align*}

We define
\begin{equation*}
v(x,t)=\int_t^\infty w(x,s)\,ds = \int_0^\infty w(x,t+s)\,ds.
\end{equation*}
Observe that the integral converges absolutely for all $t>0$. Furthermore, $v\in W^2_{1,loc}(\RR^{n+1}_+)$ and $L_{B^*} v=0$. Using \lemref{slabs}, we may bound $\nabla v$ as follows:
\begin{align*}
\doublebar{\nabla v(\,\cdot\,,t)}_{L^2(\RR^n)}
&\leq C\int_t^\infty \doublebar{\nabla w(\,\cdot\,,s)}_{L^2(\RR^n)}\,ds
\leq \int_t^\infty \frac{C}{s}\doublebar{w(\,\cdot\,,s)}_{L^2(\RR^n)}\,ds
\\&\leq \int_t^\infty \frac{C}{s^2}\doublebar{\widetilde N_+(\nabla u)}_{L^2(\RR^n)}\,ds
\leq  \frac{C}{t}\doublebar{\widetilde N_+(\nabla u)}_{L^2(\RR^n)}.
\end{align*}
Thus, by \lemref{bdd1},
\begin{equation*}
\int_{\RR^{n+1}_+} \abs{\nabla v(x,t)}^2\,t\,dx\,dt
\leq C\int_{\RR^{n+1}_+} \abs{\partial_t v(x,t)}^2\,t\,dx\,dt
+ C\doublebar{\widetilde N_+(\nabla u)}_{L^2(\RR^n)}^2.
\end{equation*}
But since $\partial_{n+1} v=-w=-a\,L_A u$, the right-hand side is at most \begin{equation*}C\doublebar{\widetilde N_+(\nabla u)}_{L^2(\RR^n)}^2
+C\triplebar{t\,L_A u}_+^2.\end{equation*}
Thus, by \thmref{squareN},
\begin{equation*}\doublebar{N_+ v}_{L^2(\RR^n)}\leq C\doublebar{\widetilde N_+(\nabla u)}_{L^2(\RR^n)}
+C\triplebar{t\,L_A u}_+.\end{equation*}
By invertibility of $\s_{B^*}^{\perp,+}$ and by uniqueness of solutions to $(D)^{B^*}_2$, we have that $v=-\partial_{n+1} \s_{B^*} h$ for some $h\in L^2(\RR^n)$. So $a\,L_Au=w=\partial_{n+1}^2 \s_{B^*} h$. 

Let $u_A = u-\E_{B,a,A} h$. By \lemref{Fgradientexists}, $\Div A\nabla u_A=0$ in $\RR^{n+1}_+$.
Furthermore, by \lemref{bdd6}
\begin{equation*}\doublebar{\widetilde N_+(\nabla u_A)}_{L^2(\RR^n)}
\leq \doublebar{\widetilde N_+(\nabla u)}_{L^2(\RR^n)}
+ C\doublebar{h}_{L^2(\RR^n)}\end{equation*}
as desired.

The existence of $\vec G=\nabla u\vert_{\partial\RR^{n+1}_+}$ and the bound \eqref{eqn:fatou} follows immediately from \thmref{2L2limits} and \lemref{bdd6}.
\end{proof}

Finally, we prove uniqueness of solutions.
\begin{cor} Let $a$, $A$, and $B$ be as in \thmref{converse}. Assume in addition that $\E^\perp_{B,a,A}$ is one-to-one, where $\E^\perp_{B,a,A} h=\partial_{n+1} \E_{B,a,A} h\vert_{\partial\RR^{n+1}_\pm}$.

If $u$ satisfies
\begin{equation*}
\left\{\begin{aligned}
L_B^*(a\,L_Au)&=0  &&\text{in }\RR^{n+1}_+,\\
\nabla u &=0 &&\text{on } \partial\RR^{n+1}_+,\\
\doublebar{\widetilde N_+(\nabla u)}_{L^2(\RR^n)}
+\triplebar{t\,L_A u}_+
&<\infty
\end{aligned}\right.
\end{equation*}
then $u$ is constant in $\RR^{n+1}_+$.
\end{cor}

\begin{proof}
By \thmref{converse} we have that $u=u_A+\E_{B,a,A} h$ for some $h\in L^2(\RR^n)$ and some $u_A$ with $\Div A\nabla u_A=0$ in $\RR^{n+1}_+$.

Let $u_+=u_A$, and let $u_-=-\E_{B,a,A} h$ in $\RR^{n+1}_-$. Observe that $\Div A\nabla u_\pm=0$ in $\RR^{n+1}_\pm$.
By \lemref{bdd6}, and because $\nabla u=0$ on $\partial\RR^{n+1}_+$, we have that
\begin{equation*}
\nabla u_-\big\vert_{\partial\RR^{n+1}_-}
=-\nabla \E_{B,a,A} h\big\vert_{\partial\RR^{n+1}_-}
=-\nabla \E_{B,a,A} h\big\vert_{\partial\RR^{n+1}_+}
=\nabla u_+\big\vert_{\partial\RR^{n+1}_+}.\end{equation*}
Thus by \lemref{jumpunique}, $u_+=u_A$ and $u_-=-\E_{B,a,A} h$ are constant in $\RR^{n+1}_\pm$. In particular, if $\E^\perp_{B,a,A}$ is one-to-one, then $h=0$, and so $u=u_A+\E_{B,a,A} h$ is constant in $\RR^{n+1}_+$, as desired.
\end{proof}

\section*{Acknowledgements} The first named author would like to thank Jill Pipher and Martin Dindos for helpful discussions concerning higher order differential equations, and Ana Grau de la Herran for many helpful discussions concerning the boundedness of linear operators that are far from being Calder\'on-Zygmund operators.

\bibliographystyle{amsplain}\bibliography{initials}\end{document}